\documentclass[11pt,a4paper,reqno]{amsart}
\usepackage[english]{babel}
\pdfoutput=1
\usepackage[applemac]{inputenc}
\usepackage[T1]{fontenc}
\usepackage{palatino}
\usepackage{amsmath}
\usepackage{amssymb}
\usepackage{amsthm}
\usepackage{enumitem}
\usepackage{amsfonts}
\usepackage{esint}
\usepackage{graphicx}
\usepackage{xcolor}
\usepackage{mathtools}
\usepackage{comment}

\usepackage[colorlinks = true, citecolor = black]{hyperref}
\title[Geometric lemmas and uniform rectifiability: Part 1]{On various Carleson-type geometric lemmas\\ and uniform rectifiability in metric spaces: Part 1}

\author[K. F\"assler]{Katrin F\"assler}
\author[I. Y. Violo]{Ivan Yuri Violo}

\address{Department of Mathematics and Statistics\\ University of Jyv\"askyl\"a \\ P.O. Box 35 (MaD),
FI-40014 University of Jyv\"askyl\"a, Finland}

\email{katrin.s.fassler@jyu.fi}
 \email{ivan.violo@sns.it {\it (current affiliation: Centro di Ricerca Matematica Ennio De Giorgi, Scuola Normale Superiore, Piazza dei Cavalieri 3, 56126 Pisa (PI), Italy)}}

\thanks{The authors were supported by the Academy of Finland
grants no.\ 321696 (K.F.) and 328846, 321896 (I.Y.V.)}
\date{\today}

\newcommand{\diam}{\operatorname{diam}}
\newcommand{\card}{\operatorname{card}}

\newcommand{\eps}{\varepsilon}

\renewcommand{\d}{{\rm d}}
\newcommand{\sfd}{{\sf d}}
\newcommand{\X}{{\rm X}}
\newcommand{\rr}{\mathbb R}

\def\Barint_#1{\mathchoice
          {\mathop{\vrule width 6pt height 3 pt depth -2.5pt
                  \kern -8pt \intop}\nolimits_{#1}}%
          {\mathop{\vrule width 5pt height 3 pt depth -2.6pt
                  \kern -6pt \intop}\nolimits_{#1}}%
          {\mathop{\vrule width 5pt height 3 pt depth -2.6pt
                  \kern -6pt \intop}\nolimits_{#1}}%
          {\mathop{\vrule width 5pt height 3 pt depth -2.6pt
                  \kern -6pt \intop}\nolimits_{#1}}}

\numberwithin{equation}{section}

\theoremstyle{plain}
\newtheorem{thm}[equation]{Theorem}
\newtheorem{thmdef}[equation]{Theorem and Definition}

\newtheorem{lemma}[equation]{Lemma}

\newtheorem{ex}[equation]{Example}
\newtheorem{cor}[equation]{Corollary}

\newtheorem{proposition}[equation]{Proposition}

\theoremstyle{definition}

\newtheorem{definition}[equation]{Definition}

\theoremstyle{remark}
\newtheorem{remark}[equation]{Remark}

\newcommand{\reg}{\mathrm{Reg}}
\newcommand{\nb}{\iota}

\usepackage{hyperref}
\hypersetup{%
  colorlinks = true,
  linkcolor  = black
}

\allowdisplaybreaks

\begin{document}

\begin{abstract} {We introduce new flatness coefficients, which we call \emph{$\iota$-numbers}, for Ahlfors $k$-regular sets in metric spaces
($k\in \mathbb{N}$). Using these coefficients for $k=1$, we
characterize uniform $1$-rectifiability in rather general metric
spaces, completing earlier work by Hahlomaa and Schul. Our proof
proceeds by} quantifying an isometric embedding theorem due to
Menger, and by an abstract argument that allows to pass from a
local covering by continua to a global covering by $1$-regular
connected sets.
\end{abstract}

\maketitle

\tableofcontents

\section{Introduction}\label{s:intro}
This note is intended as a contribution to a broad program aimed
at extending the theory of \emph{quantitative} (or uniform)
\emph{rectifiability}, pioneered by David and Semmes in Euclidean
spaces \cite{MR1113517,MR1251061}, to other metric spaces. Recent
research in this direction concerns different classes of sets,
depending on the ambient metric space:
\begin{enumerate}
\item[(1)] quantitatively rectifiable sets modelled on \emph{Euclidean} spaces,
such as $1$-regular curves, or sets with big pieces of
(bi-)Lipschitz images of Euclidean sets, \item[(2)] sets that are
quantitatively rectifiable by specific types of
\emph{non-Euclidean} Lipschitz graphs, for instance in Heisenberg
groups
and parabolic spaces.
\end{enumerate}
 Here we
focus on direction (1). We introduce new  quantitative
coefficients called \emph{$\nb$-numbers}. \textcolor{black}{We
characterize uniform $1$-rectifiability in rather general metric
spaces (Theorem \ref{t:Char1URIntro}) by a Carleson-type
summability condition for
 $\nb$-numbers. The proof uses results that we believe to
be of independent interest and which we will explain in more
detail in the subsequent paragraphs.}

\subsection{From $\beta$-numbers  in Euclidean spaces \textcolor{black}{to $\nb$-numbers in metric spaces}}
Uniformly $k$-\textcolor{black}{rectifiable} sets in
$\mathbb{R}^n$ ($k,n\in \mathbb{N}$, $1\leq k<n$) can be
characterized in many equivalent ways, for instance as $k$-regular
sets with big pieces of Lipschitz images of subsets of
$\mathbb{R}^k$, or by means of a geometric lemma for Jones
\emph{$\beta_{q,\mathcal{V}_k}$-numbers} which quantifies the
approximability of the set by $k$-dimensional planes. Here $q$ is
allowed to be any number $1\leq q<\frac{2k}{k-2}$ if $k\geq 2$ and
$1\leq q\leq \infty$ if $k=1$, recall \cite[I,1.4]{MR1251061}.
\textcolor{black}{By ``$k$-regular'' we mean sets that satisfy the
Ahlfors $s$-regularity condition \eqref{eq:regular} for $s=k$.}
For the purpose of this introduction, we say that a $k$-regular
set $E$ in Euclidean space $\mathbb{R}^n$ satisfies the
\emph{$2$-geometric lemma with respect to
$\beta_{q,\mathcal{V}_k}$}, denoted $E\in
\mathrm{GLem}(\beta_{q,\mathcal{V}_k},2)$, if there is a constant
$M\geq 0$ such that
\begin{equation}\label{eq:Beta_q}
\int_{B_R(\textcolor{black}{x_0})\cap E} \int_0^{R}
\beta_{q,\mathcal{V}_k}(B_r(x)\cap
E)^2\,\frac{dr}{r}d\mathcal{H}^k(x)\leq M R^k\,\quad
\textcolor{black}{x_0}\in E,\,R>0,
\end{equation}
where the coefficients
\begin{equation}\label{eq:beta_intro}
\beta_{q,\mathcal{V}_k}(B_r(x)\cap E):= \inf_{V\in \mathcal{V}_k}
\left( \Barint_{B_r(x)\cap
E}\left[\frac{\sfd(y,V)}{\mathrm{diam}(B_r(x)\cap
E)}\right]^q\,d\mathcal{H}^k(y)\right)^{1/q}, \quad q\in
(0,\infty),
\end{equation}
quantify in a scale-invariant and $L^q$-based way how well the set
$E$ is approximated by $k$-planes $V\in \mathcal{V}_k$ at $x\in E$
and scale $r>0$ in the Euclidean distance. The number ``$2$'' in
the definition of $\mathrm{GLem}(\beta_{q,\mathcal{V}_k},2)$
corresponds to the exponent ``$2$'' in the expression
\eqref{eq:Beta_q}.

We study another family of quantitative coefficients that we call
\emph{$\nb$-numbers}. They are well-suited for generalizations to
metric spaces. Roughly speaking, $\nb$-numbers measure
``flatness'' of a set using mappings into model spaces, rather
than using the metric distance from approximating sets.
\textcolor{black}{They} can be used to formulate a geometric lemma
analogous to \eqref{eq:Beta_q}; see Definition \ref{d:GL} for a
very general definition of geometric lemmas, which we state in
terms of systems of Christ-David dyadic cubes. Roughly speaking,
the symbol $\mathrm{GLem}(h,p,M)$ denotes a Carleson measure
condition in the spirit of \eqref{eq:Beta_q} with $\beta$-numbers
replaced by other coefficients given by $h$, and the integrability
exponent ``$2$'' replaced by ``$p$''.

For $k\in \mathbb{N}$ and a $k$-regular set $E$ in a metric space
$(\X,\sfd)$, and for $q\in (0,\infty)$, we denote by
$\nb_{q,k}(B_r(x)\cap E)$ the number
\begin{equation}\label{eq:nb_intro_metric}
\inf_{\|\cdot\|}\inf_{f:B_r(x)\cap E\to \mathbb{R}^k} \left(
\Barint_{B_r(x)\cap E}\Barint_{B_r(x)\cap
E}\left[\frac{|\sfd(y,z)-\|f(y)-f(z)\||}{\mathrm{diam}(B_r(x)\cap
E)}\right]^q\,d\mathcal{H}^k(y)d\mathcal{H}^k(z)\right)^{1/q}.
\end{equation}
Here the first infimum is taken over all norms on $\mathbb{R}^k$, and the functions $f$ in the second infimum are assumed to be \emph{Borel}.

For illustration, suppose that the double integral in
\eqref{eq:nb_intro_metric} vanishes for some $\|\cdot\|$ and $f$.
Then, because $\|\cdot\|$ is bi-Lipschitz equivalent to the
Euclidean norm on $\mathbb{R}^k$, up to $\mathcal{H}^k$ measure
zero, $B_r(x)\cap E$ is bi-Lipschitz equivalent to a subset of
Euclidean $\mathbb{R}^k$.  If also the ambient space $(\X,\sfd)$
is the Euclidean space $\mathbb{R}^n$, actually much more is true.
Since $f$ arises as an isometric embedding from a positive measure
subset of  $(\mathbb{R}^k,\|\cdot\|)$ into (strictly convex)
Euclidean space, one can show that in fact $B_r(x)\cap E$ must be
essentially contained in a $k$-plane, \textcolor{black}{see
\cite[Section 3.1.2]{CarlesonPart2}.}
A refinement of this observation is stated in
Proposition \ref{prop:zero banach}.

For the development of a meaningful theory in metric spaces, it is
crucial  to allow  all possible norms $\|\cdot\|$ on
$\mathbb{R}^k$ in \eqref{eq:nb_intro_metric}, not just the
Euclidean norm. This is similar in spirit to the use of norms in
the definition of the \emph{Gromov-Hausdorff bilateral weak
geometric lemma (BWLG)} in \cite[Definition
3.1.5]{2023arXiv230612933B}, and in both cases the norms are
allowed to depend on the point $x$ and the scale $r$. In the
recent breakthrough \cite{2023arXiv230612933B}, Bate, Hyde, and
Schul characterized, in arbitrary metric spaces, $k$-regular sets
with big pieces of Lipschitz images of $\mathbb{R}^k$ as those
$k$-regular sets that satisfy a Gromov-Hausdorff BWGL, or some
other equivalent conditions inspired by Euclidean quantitative
rectifiability. Not contained in their characterization is,
quoting the authors, ``a condition on the square summability of
some suitable variant of the Jones $\beta$-number'', while they
observed that generalizing the main result in \cite{MR3162255}
could be a first step in this direction. Finding a suitable
Carleson summability condition in this generality is a well-known
problem to which Schul alluded already in
\cite{2007arXiv0706.2517S}. We do not claim to obtain a solution
of this problem \textcolor{black}{for $k>1$}, but we hope that the
present paper could serve as a motivation to investigate
characterizations of geometric lemmas for $\nb$-coefficients.
\textcolor{black}{Here we show that the validity of  a geometric
lemma for the $\nb_{1,1}$-numbers defined in
\eqref{eq:nb_intro_metric} for $k=q=1$ is indeed equivalent to
uniform $1$-rectifiability in rather general metric spaces. In a
companion paper \cite{CarlesonPart2}, for arbitrary $k\in
\mathbb{N}$, we study a variant of the $\nb$-numbers that are
tailored specifically to Euclidean spaces and we prove that these
\emph{$\nb_{1,\mathcal{V}_k}$-numbers} can be used to characterize
uniformly $k$-rectifiable sets in $\mathbb{R}^n$ for any $k\geq
1$.}

\subsection{From Menger curvature to $\nb$-numbers in dimension
$1$} We explain some ideas behind the characterization of
uniformly $1$-rectifiable sets by means of  $\nb_{1,1}$-numbers.

Recall that $\nb_{1,1}(B_r(x)\cap E)=0$ implies the existence of
an isometric embedding from $(B_r(x)\cap E,\sfd)$, up to a $\mathcal{H}^1$ null set, into
$\mathbb{R}$. (Here we may without loss of generality assume that
the target space $\mathbb{R}$ is equipped with the Euclidean
norm). Menger \cite{MR1512479} proved criteria for isometric
embeddability of metric spaces into Euclidean $\mathbb{R}^k$. The
case $k=1$ of one of his results can be stated as follows, see
\cite{MR2888543}. If $(\X,\sfd)$ is a space with at least
five points such that the triangular excess vanishes for any
triple of points in $\X$, that is, every such triple embeds
isometrically into $\mathbb{R}$, then the whole space $\X$ embeds
isometrically into $\mathbb{R}$.

In Theorem \ref{thm:quantified menger} we obtain a quantitative
version of Menger's result that applies to metric spaces where the
triangular excess of point triples is not necessarily identically
zero, but sufficiently small. In particular we give a condition under which
such spaces embed into $(\mathbb{R},|\cdot|)$ by an almost
isometry.  With a further refinement of the  arguments we obtain also an integral version of this statement (see Theorem \ref{thm:l1 menger}) which can then be applied
 to relate $\nb_{1,1}$-numbers to the \emph{metric
$\beta$-numbers} known from the literature.
 The latter are quantitative coefficients  defined in
terms of triangular excess, see \eqref{eq:MetricBeta}. We
call them \emph{$\kappa$-numbers} in this note to emphasize the
connection with Menger curvature, see \eqref{eq:ExcessMenger} and
Example~\ref{ex:MetricBetas}.

Building on Theorem \ref{thm:l1 menger} and heavily on earlier work of Hahlomaa \cite{MR2297880} and
Schul \cite{MR2337487,2007arXiv0706.2517S,MR2554164}, we obtain
the following characterization.

\begin{thm}[Characterizations of uniform $1$-rectifiability]\label{t:Char1URIntro}
 Let $(\X,\sfd)$ be a complete, doubling,
and quasiconvex metric space. The following conditions are
quantitatively equivalent for a $1$-regular set $E$ in
$(\X,\sfd)$:
\begin{enumerate}
\item\label{i:IntroInReg} $E$ is contained in a closed and
connected $1$-regular set, \item\label{i:IntroBPLI} $E$ has big
pieces of Lipschitz images of subsets of $\mathbb{R}$, \item \label{i:IntroBPBI} $E$
has big pieces of bi-Lipschitz images of subsets of $\mathbb{R}$,
\item\label{i:Introkappa} $E$ satisfies the geometric lemma
$\mathrm{GLem}(\kappa,1)$, \item\label{i:nb_char} $E$ satisfies
the geometric lemma $\mathrm{GLem}(\nb_{1,1},1)$.
\end{enumerate}
\end{thm}
The equivalence of the conditions  \eqref{i:IntroBPLI},  \eqref{i:IntroBPBI}, and  \eqref{i:Introkappa} was proven by Schul in \cite{2007arXiv0706.2517S} using also earlier work by Hahlomaa and himself.
Compared to these results, the novelty in our
Theorem \ref{t:Char1URIntro} is the equivalence of the other
conditions with property \eqref{i:nb_char}. The
implication from
\eqref{i:Introkappa} to \eqref{i:IntroInReg} is also new for quasiconvex spaces. For bounded sets in \emph{geodesic} spaces, it was  stated by Schul in \cite[Theorem
3.11]{MR2342818}, attributed to Hahlomaa,  see also
\cite[Theorem 1.5]{MR2337487}. Our
proof of the implication ``\eqref{i:Introkappa} $\Rightarrow$ \eqref{i:IntroInReg}'' (formulated as Corollary \ref{c:IntMen}) is directly based on one
of Hahlomaa's published results (\cite[Theorem 1.1]{MR2297880})
coupled with an abstract argument that allows to pass from a local covering
by continua to a global covering by 1-regular connected sets, see Corollary
\ref{c:FromLipToReg} and Corollary \ref{c:unbounded}. This proof strategy works in quasiconvex spaces and thus makes the result applicable for instance in the
first Heisenberg group $\mathbb{H}^1$ equipped with the
Kor\'{a}nyi distance $d_{\mathbb{H}^1}$. We discuss such an application below in Theorem \ref{t:CL}.

Adapting Menger's ideas we were able to show that conditions
\eqref{i:Introkappa} and \eqref{i:nb_char} are not only equivalent
at the level of geometric lemmas, but the coefficients $\kappa$
and $\nb_{1,1}$ are comparable on neighborhoods of individual
dyadic cubes. For more detailed statements, see Corollary
\ref{eq:equiv1UR} and Theorem \ref{t:Char1URII}, in particular
\eqref{eq:equivalent alk}, later in this note.
By the work of Bate, Hyde, and Schul \cite{2023arXiv230612933B} in
the $1$-regular case, the conditions in Theorem
\ref{t:Char1URIntro} are further equivalent to $E$ satisfying a
(Gromov-Hausdorff) bilateral weak geometric lemma or any of the
other conditions stated in \cite[Theorem B]{2023arXiv230612933B}.

\subsection{Relation to previous work} This note has been
motivated by several lines of research, which we briefly sketch
here. The results are too numerous for an exhaustive list, but we hope that the interested reader will find some
directions for further reading. We  also refer to the survey \cite{MR4520153} by Mattila for more information.

\subsubsection{Euclidean-type (quantitative) rectifiability.} The qualitative theory of Federer-type rectifiability
in metric spaces  (using Lipschitz images of subsets of Euclidean
spaces) \cite{MR1189747,MR1800768,MR4506771} and the already
well-established quantitative theory in Euclidean space
\cite{MR1113517,MR1251061} motivated the recent work by Bate,
Hyde, and Schul \cite{2023arXiv230612933B}. This provides several
equivalent characterizations of sets that are quantitatively
rectifiable modelled on Euclidean spaces. Pivotal examples from
the literature where this notion of uniform rectifiability is
well-suited are low-dimensional sets in Heisenberg groups
\cite{hahheisenberg,MR3678492,MR4299821,MR4375018}, and subsets of
regular curves in metric spaces \cite{MR2297880,MR2337487}. For
the case of $1$-dimensional sets in metric spaces, there is also a
growing body of literature concerned with the  travelling salesman
theorems \textcolor{black}{and quantitative methods for the study
of \emph{qualitatively} rectifiable sets}
\cite{2021arXiv210906753B,2020arXiv200211878B,MR3319560,MR2371434,MR3456155,MR3512421}.
Li introduced and used in
\cite{https://doi.org/10.1112/jlms.12582} \emph{stratified
$\beta$-numbers} to characterize subsets of Carnot groups that are
contained in rectifiable curves. Coefficients of this type
certainly seem promising to study also uniform rectifiability for
low-dimensional sets in Carnot groups. On the other hand, they are
defined specifically for the setting of stratified Lie groups,
while an advantage of the $\kappa$- and the $\nb$-coefficients is
their versatility. The $\kappa$-numbers are tailored to
$1$-dimensional sets, but higher-dimensional variants have been
considered in \cite{MR2554164,MR3162255} for images of Lipschitz
functions $f\colon[0,1]^k \to (\X,\sfd)$. Various coefficients
related to Menger curvatures have also been used to characterize
higher-dimensional (uniform) rectifiability in Euclidean spaces
\cite{MR2558685,MR2848529,MR3729499,10.14321/realanalexch.46.1.0001}.
Investigating connections between $\iota$-numbers and
higher-dimensional variants of $\kappa$-numbers could be an
interesting topic for future research.

\subsubsection{Other notions of quantitative rectifiability} Motivated by
specific PDEs, quantitative theories of rectifiability have also been
developed in settings where the natural building blocks are
different from Lipschitz images of subsets of $\mathbb{R}^k$. This
applies for instance to quantitative rectifiability for
$1$-codimensional sets in parabolic spaces
 \cite{MR1323804,MR1996443,MR2053754,MR4605204} (where \emph{regular
 parabolic Lipschitz graphs} are used)
 and sub-Riemannian Heisenberg
groups \cite{MR3992573,chousionis2020strong,2023arXiv230413711C,MR4127898}
(where \emph{intrinsic Lipschitz graphs} are studied), and is not directly related to the present paper.

\subsubsection{Axiomatic results in metric spaces} In addition to the mentioned papers which concern
specific model spaces, there are also a results available that
deal with concepts related to rectifiability and quantitative
rectifiability in rather abstract, axiomatic settings
\cite{MR3470666,MR4485846,MR4514458}. While
\cite{MR4485846,MR4514458} are motivated by applications to
parabolic spaces and Heisenberg groups, respectively, the main
ingredients in both cases are abstract metric space results. The
paper \cite{MR4514458} contains a sufficient criterion for a
metric space to admit a big bi-Lipschitz piece of a model space.
Unfortunately, the assumptions of the theorem are stronger than
 the information we can deduce from the validity of a
geometric lemma for $\nb$-numbers. In \cite{MR4485846}, the
authors provide a general framework for the study of corona
decompositions and geometric lemmas in metric spaces and we follow
their notation to a large extent. However, the main results in
\cite{MR4485846} do not seem to have direct applications in our
setting, which concerns $\nb$-numbers defined through mappings,
rather than $\beta$-numbers defined through approximating sets.

\subsubsection{Approximate isometries} The definition of
$\nb$-numbers is inspired by the $\mathbf{b}$-numbers
studied by the second-named author in \cite{MR4489627}. The
coefficients employed in the present paper differ from the
$\mathbf{b}$-numbers in two crucial ways: first, they are
$L^q$-based, $q\in [1,\infty)$ instead of $L^{\infty}$-based, and
second, they can be defined in arbitrary metric spaces. The
$\mathbf{b}$-numbers in  \cite{MR4489627} are defined via approximate isometries. A
geometric lemma for $\nb$-numbers heuristically still yields many
almost isometric mappings from the given set to model spaces, but
in general it remains an open question if and how this information
can be used to build big pieces of Euclidean bi-Lipschitz images
inside the set.

\smallskip

\textbf{Structure of the paper.} Section \ref{s:prelim} contains
preliminaries. \textcolor{black}{In particular, we collect various
facts about geometric lemmas that will be used here and in the
sequel \cite{CarlesonPart2}.} In the main part of the paper
(Section \ref{s:SuffRegCoverCurve} and Section \ref{s:Char1UR}) we
discuss $1$-regular sets in metric spaces. In Section
\ref{s:ConstrCover} we give sufficient local conditions for the
existence of global $1$-regular covering continua for sets in
metric spaces. In Section \ref{s:ApplMenger} we present an
application (Corollary \ref{c:IntMen}) where these local
conditions are satisfied thanks to a result by Hahlomaa. As a
corollary, we complete in Section \ref{s:Char1UR} the proof of
\textcolor{black}{the characterization of uniform
$1$-rectifiability in metric spaces stated in Theorem
\ref{t:Char1URIntro}}. Appendix \ref{sec:appendix} contains
technical results needed in Section \ref{s:Char1UR}, related to
the quantification of Menger's theorem about isometric embeddings
into $\mathbb{R}$.

\smallskip

\textbf{Acknowledgements.} We are grateful to Tuomas Orponen for illuminating discussions,
 especially for crucial help related to the construction of $1$-regular covering
 curves. \textcolor{black}{We thank the referees for their careful reading of the
 paper and comments that helped to substantially improve the
 presentation.}

\section{Preliminaries}\label{s:prelim}
\textbf{Notation.} We write $A \lesssim B$ to denote the existence
of an absolute constant $C \geq 1$ such that $A \leq CB$. The
 inequality $A \lesssim B \lesssim A$ is abbreviated to
$A \sim B$. If the constant $C$ is allowed to depend on a
parameter "$p$", we indicate this by writing $A \lesssim_{p} B$.
We denote the diameter of a set $E$ in a metric space by
$\mathrm{diam}(E)$ and use the convention that
$\mathrm{diam}(E)=+\infty$ if $E$ is unbounded.

\subsection{Standard quantitative notions}

Throughout this paper \textcolor{black}{-- and its sequel
\cite{CarlesonPart2} --} we employ quantitative notions that are
ubiquitous in the theory of uniform rectifiability in Euclidean
spaces and that are increasingly applied in other metric spaces as
well. The terminology used in Sections
\ref{ss:Ahlfors}-\ref{sec:glem} follows closely the presentation
in \cite{MR4485846} in the case of Hausdorff measures
$\mu=\mathcal{H}^s|_E$. Readers familiar with the standard
terminology may wish to proceed directly to Section
\ref{s:NewCoeff}, where we introduce new quantitative
coefficients.

\begin{definition}[Quasiconvex metric space]\label{def:quasiconvex}
    A metric space $(\X,\sfd)$ is called \emph{quasiconvex}
     if there exists a constant $L\ge 1$, called \emph{quasiconvexity constant}, such that every couple of points $x,y \in \X$ can be joined by a curve of length at most $L\sfd(x,y).$
\end{definition}

We denote by $B_r(x)=\{y\in \X: \sfd(x,y)< r\}$ the open ball with
center $x$ and radius $r$ in a given metric space $(\X,\sfd)$.

 A metric space $(\X,\sfd)$ is commonly said to be
\emph{doubling} if
    for all $x\in \X$ and $r>0$  the ball $B_{r}(x)$ can be covered by the union of at most $C$ balls of radius $\frac r2,$
    for some constant $C>1$ independent of $x$ and $r$. For our
    purposes it will be more convenient to use the following,
    equivalent, condition:
\begin{definition}[Doubling metric space]\label{def:doubling}
    A metric space $(\X,\sfd)$ is called \emph{doubling} if there exists a constant $D\ge 1$, called \emph{doubling constant}, such that,
 for every
$\varepsilon \in (0,1/2]$,   every subset of $\X$ of diameter $r$ in
can be covered by $\leq \varepsilon^{-D}$ sets of
diameter at most $\varepsilon r$.
\end{definition}
The covering in Definition  \ref{def:doubling} can also be taken
to have uniformly bounded overlap, with multiplicity depending
only on $D$.

\subsubsection{Ahlfors regular sets and dyadic
systems}\label{ss:Ahlfors}

\begin{definition}[$s$-regular sets]
A set $E\subset (\X,\sfd)$ with $\mathrm{diam}(E)>0$ is said to be
\emph{$s$-regular}, $s>0$, if it is closed
 and
there exists $C\ge 1$, called \emph{regularity constant}, such
that
\begin{equation}\label{eq:regular}
    C^{-1}r^s\le \mathcal{H}^s(B_r(x)\cap E)\le C r^s, \quad  x \in E,\,\,
    r\in(0,2\diam(E)),
\end{equation}
in which case we write $E\in\mathrm{Reg}_s(C)$.
 Furthermore if only the first (resp.\ the second) inequality in
     \eqref{eq:regular} is satisfied and  $E$ is not necessarily closed we say that $E$ is \emph{lower} (resp.\ \emph{upper})
      $s$-regular and we write $E\in \reg^-_s(C)$ (resp.\ $E\in \reg^+_s(C)$).
      Finally we say that the metric space $(\X,\sfd)$ is $s$-regular if the whole set $\X$ is an $s$-regular
      set with respect to $\sfd$. We also use the term \emph{Ahlfors regular} to denote the class of sets that are $s$-regular for some exponent $s$.
\end{definition}

Up to replacing $C$ by $2^sC$, the second inequality in
\eqref{eq:regular} holds also for arbitrary $x \in \X$ and $r>0$.
Moreover it can be checked from the definition that if $E_i\in
\mathrm{Reg}_s(C_i)$ for $i\in \{1,2\}$ are intersecting sets of a
common ambient space, then $E_1\cup E_2 \in \mathrm{Reg}_s(C)$
with a constant $C$ that can be taken to depend only on the regularity constants $C_1$
and $C_2$ of the two initial sets.

\medskip

Regular sets in
metric spaces admit systems of generalized dyadic cubes. For
$k$-regular sets in $\mathbb{R}^n$, the existence of such systems
was proven by David in \cite[B.3]{MR1009120}, \cite{MR1123480}.
More generally, Christ constructed dyadic cube systems for spaces
of homogeneous type in \cite[Theorem 11]{MR1096400}, see also
\cite{MR2901199} and references therein for variants of this
construction. We use the version for Ahlfors regular sets in
metric spaces as stated in \cite[Lemma 2.5]{MR4485846}, see also \cite[Sect. 5.5]{MR1616732}, but we
include a separate notational convention for bounded sets
following the comment on \cite[p.22]{MR1616732}.
 If the
regular set $E$ is bounded, we define $\mathbb{J}:= \{j\in
\mathbb{Z}\colon j\geq n\}$ where $n\in \mathbb{Z}$ is such that
$2^{-n} \leq \mathrm{diam}(E) < 2^{-n+1}$, otherwise we denote
$\mathbb{J}:=\mathbb{Z}$.

\begin{thmdef}[Dyadic systems \cite{MR1009120,MR1096400}]\label{dl:dyadic} For any $s>0$ and $C\geq 1$, there exists
a constant $c_0\in (0,1)$  such that in an
arbitrary
metric space, every set $E\in\mathrm{Reg}_s(C)$
 admits a \emph{system of dyadic cubes}
$\Delta=\bigcup_{j\in\mathbb{J}} \Delta_j$, where $\Delta_j$ is a
family of pairwise disjoint Borel sets $Q\subset E$ (\emph{cubes})
satisfying
\begin{enumerate}
\item\label{dyadic1} $E=\bigcup_{Q\in \Delta_j}Q$ for each $j\in
\mathbb{J}$, \item\label{dyadic2} for $i,j\in \mathbb{J}$ with
$i\leq j$, if $Q\in \Delta_i$ and $Q'\in \Delta_j$, then either
$Q'\subset Q$ or $Q\cap Q'=\emptyset$, \item\label{dyadic3} for
$j\in \mathbb{J}$, $Q'\in \Delta_j$ and $i<j$ with $i\in
\mathbb{J}$, there is a unique $Q\in \Delta_i$ (\emph{ancestor})
such that $Q'\subset Q$, \item\label{dyadic4} for $j\in
\mathbb{J}$ and $Q\in \Delta_j$, it holds $\mathrm{diam}(Q)\leq
c_0^{-1} 2^{-j}$, \item\label{dyadic5} for $j\in \mathbb{J}$ and
$Q\in \Delta_j$, there is a point $x_Q \in E$ (\emph{center}) such
that $B_{c_0 2^{-j}}(x_Q)\cap E \subset Q$.
\end{enumerate}
For $j\in \mathbb{J}$ and $Q\in \Delta_j$, we denote $\ell(Q):=
2^{-j}$ and refer to this as the \emph{side length} of the cube.
We also define
\begin{displaymath}
\Delta_{Q_0}:=\{Q\in \Delta\colon Q\subset Q_0\},\quad Q_0 \in
\Delta,
\end{displaymath}
and  for a given constant $K>1$, we set
\begin{displaymath}
KQ:= \{x\in E\colon \mathrm{dist}(x,Q) \leq
(K-1)\,\mathrm{diam}(Q)\}.
\end{displaymath}
\end{thmdef}

Following are additional comments and notations about a system of
dyadic cubes that will be useful in the sequel. If $E\in
\mathrm{Reg}_s(C)$ and $\Delta$ is a dyadic system on $E$, then
for $Q\in \Delta$ and $K>1$, the set $KQ$ is simply the
intersection of $E$ with the closed $(K-1)\,\mathrm{diam}(Q)$
neighborhood of $Q$, and
\begin{equation}\label{eq:KQ_in_Ball}
KQ \subset B_{K\mathrm{diam}(Q)}(x_Q)\cap E.
\end{equation}
 Moreover, it follows from  conditions \eqref{dyadic4}-\eqref{dyadic5} that
\begin{equation}\label{eq:ball_in_Q}
B_{c_0 \ell(Q)}(x_Q) \cap E \subset Q \subset B_{c_0^{-1}
\ell(Q)}(x_Q) \cap E,\quad Q\in \Delta.
\end{equation}
Since $E\in \mathrm{Reg}_s(C)$, this and  \eqref{dyadic4} imply
that
\begin{equation}\label{eq:MeasCube}
C^{-1} (c_0 \ell(Q))^s \leq \mathcal{H}^s(Q) \leq C (c_0^{-1}
\ell(Q))^s \quad \text{and}\quad C^{-2/s} c_0 \ell(Q) \leq
\mathrm{diam}(Q) \leq c_0^{-1} \ell(Q).
\end{equation}
Combining the second \textcolor{black}{estimate} in
\eqref{eq:MeasCube} and condition \eqref{dyadic1} we can infer the
existence of a constant $K=K(s,C)>1$ such that the following holds
for all $z\in E$ and $0<R< \mathrm{diam}E$. If $j\in \mathbb{J}$
is such that $2^{-j}\leq R <2^{-j+1}$, then there exists $Q\in
\Delta_j$ such that
\begin{equation}\label{eq:ball_enlarged_cube}
E\cap B(z,R)\subset K Q.
\end{equation}
For every $Q\in \Delta_{j_0}$ and $j \in \mathbb N\cup\{0\}$ we
define the $j$-th descendants of $Q$ by
 \begin{equation}\label{def:descendants}
    F_j(Q)\coloneqq \{Q'\in  \Delta_{j+j_0} \ : \ Q'\subset Q\}.
 \end{equation}
It is easy to deduce from the first part of \eqref{eq:MeasCube},
and observing  that the cubes in $F_i(Q)$ are pairwise
disjoint, that
\begin{equation}\label{eq:number of descendants}
    \card(F_j(Q))\le c 2^{s\cdot j},
\end{equation}
for some constant $c$ depending only on $s$ and $C$.  Similarly, using again \eqref{eq:MeasCube}, for all $K\ge 1$ and all $Q\in \Delta_j$, $j\in \mathbb J,$ there exist cubes $Q_1,\dots,Q_m \in \Delta_j$, not necessarily distinct, such that
\begin{equation}\label{eq:cube patch}
    KQ\subset \cup_{i=1}^m Q_i\subset K_0Q
\end{equation} where $m\in \mathbb N$ and $K_0>1$ are constants depending only on $s,$  $C$ and $K$.
Finally we  note that combining \eqref{dyadic1} and \eqref{dyadic2} in Definition \ref{dl:dyadic} it follows that
\begin{equation}\label{eq:sum of parts}
    \sum_{Q'\in F_j(Q)}\mathcal{H}^s(Q')=\mathcal{H}^s(Q), \quad Q \in \Delta,\,  j \in \mathbb N \cup\{0\}.
\end{equation}

\subsubsection{Geometric lemmas for various coefficient functions}\label{sec:glem}

Throughout the paper, we will encounter various coefficients that
measure how well an $s$-regular set $E$ satisfies a certain
property at the scale and location of a given dyadic cube $Q$. We
are mainly concerned with the question whether $E$ fulfills a
Carleson-type summability condition in the spirit of a \emph{geometric lemma}
for the given set of coefficients. We first introduce the notation
for discussing these questions in a unified framework.

We let $\mathcal{B}(\X)$ be the Borel $\sigma$-algebra of a metric
space $(\X,\sfd)$. For a closed set $E\subset \X$, the family
$\{B\cap E\colon B\in \mathcal{B}(\X)\}$ coincides with the Borel
$\sigma$-algebra on $E$ with respect to the topology induced by
the metric $\sfd|_E$. We denote by $\mathcal{D}_s(E)$ the family
of bounded Borel sets in $E$ that have positive $\mathcal{H}^s$
measure. In particular, if $E$ is $s$-regular and $\Delta$ a
dyadic system on $E$, then $\Delta \subset \mathcal{D}_s(E)$ and
also $KQ\in \mathcal{D}_s(E)$ for every $Q\in \Delta$ and $K>1$.

\begin{definition}[Geometric lemma]\label{d:GL}
Given $p\in (0,\infty)$, $s>0$, an $s$-regular set $E$ in a metric
space, $\mu\coloneqq\mathcal{H}^s\lfloor_E$
 and a
 function
$h:\mathcal{D}_s(E) \to [0,1]$, we say that $E$ satisfies the
\emph{$p$-geometric lemma with respect to $h$}, and  write
$E\in\mathrm{GLem}(h,p)$, if there exists a constant $M$ such that
for every dyadic system $\Delta$ on $E$, we have
\begin{equation}\label{eq:SGL}
\sum_{Q\in \Delta_{Q_0}} h(2Q)^p\, \mu(Q) \leq M \mu(Q_0), \quad
Q_0\in \Delta.
\end{equation}
In this case, we also write  $E\in\mathrm{GLem}(h,p,M)$.
\end{definition}

In practice, the function $h$ will often arise as $h(S)\coloneqq H(S\cap
E)$, where $H$ depends on the regularity exponent $s$ of $E$, but
is defined for a larger class of Borel sets of the ambient space
$\X$, see Examples \ref{ex:Betas}-\ref{ex:MetricBetas} and
Definition \ref{d:alpha}.

\begin{remark}\label{rmk:indep glem} For many functions $h$ of interest and in particular for all the relevant ones appearing in this note,
the condition ``$E\in \mathrm{GLem}(\gamma,p)$'' is equivalent to
requiring \eqref{eq:SGL} for a \emph{specific} dyadic system
$\Delta(E)$ on $E$, rather than for \emph{all} possible such
systems. See  Lemma \ref{l:coeff} and Remark \ref{rmk:coeff}, or \cite[Remark 2.28]{MR4485846}. A related
statement for multi-resolution families (instead of systems of
dyadic cubes) is \cite[Lemma B.1]{MR4420442}.
\end{remark}

We next give two examples of geometric lemmas that have appeared
in the literature. In the first one, the coefficient function $h$
is a generalization of the classical $\beta$-numbers from Jones'
traveling salesman theorem.

\begin{ex}[$\beta$-numbers]\label{ex:Betas}
{\color{black}
We recall the definition of the usual $L^q$-based $\beta$-numbers  that appeared already in \eqref{eq:beta_intro}:
\[
\beta_{q,\mathcal{V}_k}(B_r(x)\cap E):= \inf_{V\in \mathcal{V}_k}
\left( \Barint_{B_r(x)\cap
E}\left[\frac{\sfd(y,V)}{\mathrm{diam}(B_r(x)\cap
E)}\right]^q\,d\mathcal{H}^k(y)\right)^{1/q}, \quad q\in
(0,\infty),
\]
where $E$ is a $k$-regular subset of $\rr^n$ and $\mathcal V_k$ is the family of $k$-dimensional affine planes. We  now generalize this this notion to an arbitrary metric space  $(\X,\sfd)$ by considering, instead of planes, a general family $\mathcal{A}$ of (non-empty) subsets of $\X$
} such that each point of $\X$ is contained in at least one element $A\in \mathcal{A}$.
 For $q\in (0,\infty]$, $s>0$,  a
closed set $E \subset \X$ of locally finite
$\mathcal{H}^s$-measure, and $\mu:=\mathcal{H}^s\lfloor_E$, we then define
for every $A\in \mathcal A$
\[
\beta_{q,A}(S)\coloneqq \left\{\begin{array}{ll}\left(
\frac{1}{\mu(S)}\int_{S}\left[\frac{\sfd(y,A)}{\mathrm{diam}(S)}\right]^q\,d\mu(y)\right)^{1/q},&
q\in (0,\infty),\\ \sup_{y\in
S}\frac{d(y,A)}{\mathrm{diam}(S)},& q=\infty,\end{array}
\right.\quad S\in \mathcal{D}_s(E)
\]
and
\begin{equation}\label{eq:SetTheoreticBeta}
\beta_{q,\mathcal{A}}(S):= \inf_{A\in
\mathcal{A}}\beta_{q,A}(S).
\end{equation}
This definition is typically applied if
 $S$ is a ``surface ball''
$B_r(x)\cap E$ or a set of the form $KQ$ for a dyadic cube $Q$ on
an $s$-regular set $E$. The definition of
$\beta_{{q},\mathcal{A}}(S)$ however makes sense more
generally, whenever $S$ is a Borel set with $0<\mathrm{diam}(S)<\infty$,
and we will occasionally apply it in this sense. If $E$ is
$s$-regular, the condition
$E\in\mathrm{GLem}(\beta_{q,\mathcal{A}},p)$ is equivalent to the
condition ``$E\in \mathrm{GLem}(\mathcal{A},p,q)$'' stated in
\cite[Defn 2.16]{MR4485846}.
The additional assumption on $\mathcal{A}$ is imposed to ensure that the function $\beta_{q,\mathcal{A}}$ takes values in $[0,1]$, which will be convenient in the following. In practice, milder assumptions would often suffice. Finally, $\beta_{q,\mathcal{A}}$ of course depends on $\mu$ (respectively on the set $E$), but since this dependence will always be clear from the context, we do not indicate it in our notation.
\end{ex}

The next example arises from the study of (quantitative)
$1$-rectifiability in metric spaces, and it involves exclusively
$1$-dimensional Hausdorff measures. The relevant coefficients can
be thought of as $1$-dimensional \emph{metric $\beta$-numbers},
and  they appeared with different notations in the literature. For
the purpose of this paper, we will refer to them as
\emph{$\kappa$-numbers}, where $\kappa$ is indicative of the
connection to Menger curvature. Before stating the definition, we
introduce some notation.

Given a metric space $(\X,\sfd)$ and three points
$x_1,x_2,x_3\in\X$ we define the \emph{triangular excess}
\begin{equation}\label{eq:def excess}
\begin{split}
     \partial (\{x_1,x_2,x_3\})&\coloneqq \inf_{\sigma \in S_3} \left\{\partial_1(x_{\sigma(1)},x_{\sigma(2)},x_{\sigma(3)})\right\}\\
&\coloneqq \inf_{\sigma \in S_3} \left\{\sfd(x_{\sigma(1)},x_{\sigma(2)})+\sfd(x_{\sigma(2)},x_{\sigma(3)})-\sfd(x_{\sigma(1)},x_{\sigma(3)})\right\},
\end{split}
\end{equation}
where $S_3$ is the group of permutations of $\{1,2,3\}$. Note that $\partial(\{x_1,x_2,x_2\})$ depends only on the set  $\{x_1,x_2,x_2\}$, while $\partial_1(x_1,x_2,x_3)$ takes into account also the order.  If the
three points lie at comparable distance from each other, then
their triangular excess is related to their \emph{Menger
curvature} $c(x_1,x_2,x_3)$ as indicated by formula
\eqref{eq:ExcessMenger}. Let $\{x_1',x_2',x_3'\}$ be the image of
$\{x_1,x_2,x_3\}$ under an isometric embedding of the triple into
the Euclidean plane. If $x_1',x_2',x_3'$ are colinear, we define
$c(x_1,x_2,x_3)=0$, otherwise $c(x_1,x_2,x_3)=1/R$, where $R$ is
the radius of the unique circle passing through $x_1',x_2',x_3'$.
With this definition, if $(x_1,x_2,x_3)$ belongs to
\begin{displaymath}
\mathcal{F}:=\{(x_1,x_2,x_3) \in E\times E \times E \ : \
\sfd(x_i,x_j)\le A \sfd(x_k,x_l), \, \text{ for all
}i,j,k,l\in\{1,2,3\}, \, k\neq l \}
\end{displaymath}
for some constant $A>0$, then
\begin{equation}\label{eq:ExcessMenger}
c^2(x_1,x_2,x_3) \,\mathrm{diam}\{x_1,x_2,x_3\}^3 \sim_{A}
\partial(\{x_1,x_2,x_3\}).
\end{equation}
see \cite[Remark 2.3]{MR2342818} or \cite[Remark 1.2]{MR2337487}.

\begin{ex}[$\kappa$-numbers]\label{ex:MetricBetas}
For $s>0$ and  a closed set $E \subset \X$ of locally finite
$\mathcal{H}^1$-measure, and $\mu\coloneqq\mathcal{H}^1\lfloor_E$ we
define
\begin{equation}\label{eq:MetricBeta}
\kappa(S):= \frac{1}{\mu(S)^3}\int_S\int_S\int_{S}
\frac{\partial(\{x_1,x_2,x_3\})}{\mathrm{diam}(S)}\,d\mu(x_1)d\mu(x_2)d\mu(x_3),
\end{equation}
for $S\in\mathcal{D}_s(E)$. Conditions in the spirit of geometric
lemmas for the coefficients $\kappa$ and related rectifiability
results have appeared in work of Hahlomaa and Schul
\cite{MR2163108,MR2297880,MR2456269,MR2337487,MR2342818,2007arXiv0706.2517S}.
\end{ex}

For many applications, it is irrelevant whether the condition
\eqref{eq:SGL} in the definition of the geometric lemma is stated
with constant ``$2$'' on the left-hand side or with any other
constant $K>1$. This is the case for the coefficient functions  in
Examples \ref{ex:Betas} and \ref{ex:MetricBetas} (but also later in Definition \ref{d:alpha}), as the following
result shows (see also Remark \ref{rmk:coeff} below). This is akin
to the situation in the case of Jones' traveling salesman theorem
in $\mathbb{R}^n$, see \cite[Corollary 2.3]{MR3319560}.

\begin{lemma}[Different neighborhoods of cubes]\label{l:coeff} Let
$E\in \mathrm{Reg}_s(C)$, $\mu\coloneqq \mathcal{H}^s\lfloor_E$, and let
$h\colon\mathcal{D}_s(E)\to [0,1]$ be a function with the
following property. For every $N\geq 1$, there exists a constant
$C_N$ such that the following monotonicity condition holds for all
$A,B\in \mathcal{D}_s(E)$:
\begin{equation}\label{eq:h_monot}
A\subset B\text{ and }\mathrm{diam}(B)^s \leq N
\mu(A),\quad\Rightarrow \quad h(A)\leq C_N h(B).
\end{equation}
Then, for every $K>K_0\geq 1$, there exists a constant
$m=m(K_0,K,s,C)$ such that if $\Delta$ is a dyadic system on $E$
and $Q_0\in \Delta$, then
\begin{equation}\label{eq:GL_K}
\sum_{Q\in \Delta_{Q_0}} h(KQ)^p\, \mu(Q) \lesssim \sum_{Q\in
\Delta_{\hat Q_0}} h(K_0Q)^p\, \mu(Q) ,
\end{equation}
where the implicit constant depends only on $K_0,K,s,C$ and $p$, while
 $\hat Q_0\in \Delta$ is such that $Q_0 \subset \hat Q_0$ and
$\ell(\hat Q_0) \leq 2^m \ell(Q)$. In particular the validity of $\mathrm{GLem}(h,p)$ does not depend on the choice of dyadic system.
\end{lemma}

\begin{proof}
We assume first that $E$ is unbounded. Let $K_0\geq 1$ and fix a
constant $K>K_0$. Then there exists $m=m(K_0,K,s,C)\in \mathbb{N}$
such that for every $Q\in \Delta$ there is a unique ancestor $\hat
Q\in \Delta$ with
\begin{equation}\label{eq:AncestorCond}
Q\subset \hat Q,\quad \ell(\hat Q) = 2^m \ell(Q),\quad K Q \subset
K_0 \hat Q.
\end{equation}
Indeed, for arbitrary $m\in \mathbb{N}$, there clearly exists
$\hat Q\in \Delta$ satisfying the first two conditions in
\eqref{eq:AncestorCond}, and for such $\hat Q$,  it follows for
all $y\in KQ$ that
\begin{align*}
\mathrm{dist}(y,\hat Q) &\leq \mathrm{dist}(y,Q) \leq
(K-1)\mathrm{diam}(Q)\leq (K-1)c_0^{-1} \ell(Q) = (K-1)c_0^{-1}
2^{-m} \ell(\hat Q)\\&\overset{\eqref{eq:MeasCube}}{\leq}
(K-1)C^{2/s}c_0^{-2} 2^{-m}\mathrm{diam}(\hat Q),
\end{align*}
which shows that $m$ can be made large enough, depending only on
$s,C$ (also via $c_0$), $K_0$ and $K$, such that also the third
condition in \eqref{eq:AncestorCond} is satisfied. Once $m$ is
fixed, $\hat Q$ is uniquely determined according to condition (3)
in Theorem and Definition \ref{dl:dyadic}.

Since the coefficient function $h$ has the monotonicity property
 \eqref{eq:h_monot}, and since $K Q \subset K_0 \hat Q$
and $\mathrm{diam}(K_0 \hat Q)^s
\lesssim_{s,C,K_0,K}\mathcal{H}^s(Q)$ by \eqref{eq:AncestorCond}
and \eqref{eq:KQ_in_Ball}, it follows that
\begin{equation}\label{eq:ConditionGamma}
h(K Q) \lesssim_{s,C,K,K_0}h(K_0 \hat Q).
\end{equation}
From \eqref{eq:AncestorCond} and \eqref{eq:ConditionGamma}, we
then easily deduce for every $Q_0 \in \Delta$ that
\begin{align*}
\sum_{Q\in \Delta_{Q_0}}h(KQ)^p \mu(Q)\lesssim \sum_{Q\in
\Delta_{Q_0}}h(K_0 \hat Q)^p \mu(Q)\leq \sum_{Q\in
\Delta_{Q_0}}h(K_0 \hat Q)^p \mu(\hat Q)\lesssim \sum_{Q\in
\Delta_{\hat Q_0}}h(K_0 Q)^p \mu( Q),
\end{align*}
where the implicit constants may depend on all the parameters $s$,
$C$, $K$, $K_0$, and $p$. In the last inequality  we used the fact
that for every $Q'\in \Delta$ there are only $\lesssim_{m,s,C}$
many $Q\in \Delta_{Q'}$ such that $\ell(Q')= 2^m \ell(Q)$ (see  \eqref{eq:number of descendants}).

If $E$ is bounded, then there still exists a constant
$m=m(K_0,K,s,C)$ such that \eqref{eq:AncestorCond} holds, but only
for $Q\in \Delta_j$ with $j\geq n+m$ (recall that we have defined
$\Delta_j$ for $j\geq n$, where $2^{-n}\leq \mathrm{diam}(E)
<2^{-n+1}$). If $Q_0 \in \Delta_j$ for some $j< n+m$, then we can
conclude exactly as we did in the case of unbounded sets $E$. If,
on the other hand, $Q_0 \in \Delta_j$ for some $j>n+m$, then we
define $\hat Q_0$ to be the unique cube in $\Delta_n$ with $\hat
Q_0 \supseteq Q_0$. Then we split the relevant sum as follows
\begin{equation}\label{eq:GLsplit}
\sum_{Q\in \Delta_{Q_0}} h(KQ)^p\, \mu(Q)= \sum_{Q\in
\Delta_{Q_0}\cap [\cup_{j\geq n+m}\Delta_j]} h(KQ)^p\, \mu(Q) +
\sum_{Q\in \Delta_{Q_0}\cap [\cup_{j< n+m}\Delta_j]} h(KQ)^p
\mu(Q).
\end{equation}
The first sum on the righthand side can be bounded by $\sum_{Q\in
\Delta_{\hat Q_0}} h(K_0Q)^p\, \mu(Q)$ using the same computations
as in the case of unbounded sets $E$, observing that $Q\in
\Delta_{Q_0}\cap [\cup_{j\geq n+m}\Delta_j]$ implies $\hat Q
\subset \hat Q_0$ with the new definition of $\hat Q$. It remains
to control the second sum on the right-hand side of
\eqref{eq:GLsplit}. Since $|h|$ is assumed to take values in $[0,1]$, we can simply bound this by $
\sum_{Q\in \Delta_{Q_0}\cap [\cup_{j< n+m}\Delta_j]}\mu(Q)$ and
use that the sum runs over at most $m$ generations. Coupled with
\eqref{eq:sum of parts} this yields the claim.

To show that $\mathrm{GLem}(h,p)$ is independent of the dyadic system, fix two dyadic systems $\Delta, \tilde \Delta$ and assume that $\mathrm{GLem}(h,p)$ holds for $\tilde \Delta.$ For all $Q\in \Delta_j$ there exists $\tilde Q\in \tilde \Delta_j$ such that $2Q\subset k\tilde Q$, where $k>1$ is a constant depending only on $s$ and $C.$ Fix any $Q_0 \in \Delta$ and note that for all $Q \in \Delta_{Q_0}$ it holds $\tilde Q\subset k_0Q_0\subset k_0k\tilde Q_0$ with $k_0>1$  a constant depending only on $s$ and $C.$
Hence using \eqref{eq:h_monot} we obtain that
\begin{equation}\label{eq:glem indep}
    \sum_{Q\in \Delta_{Q_0}} h(2Q)^p\, \mu(Q) \lesssim_{s,C} \sum_{Q\in \Delta_{Q_0}} h(k\tilde Q)^p {\mu(Q)}\lesssim_{s,C}  \sum_{Q'\in \tilde \Delta,\, Q'\subset k_0k\tilde Q_0} h(kQ')^p{\mu(Q')},
\end{equation}
where similarly as above we used that for all $Q'\in \Delta$ there exists $\lesssim_{s,C}$ cubes $Q\in \Delta$ such that $\tilde Q=Q'.$ From inequality \eqref{eq:glem indep} applying \eqref{eq:cube patch} and then \eqref{eq:GL_K}  we easily obtain that $\mathrm{GLem}(h,p)$ holds for $\Delta.$
\end{proof}

\begin{remark}\label{rmk:coeff}
    The coefficients of both Example \ref{ex:Betas} and Example \ref{ex:MetricBetas} satisfy assumption \eqref{eq:h_monot} of the above lemma. To see this let $E\in \reg_s(C)$ and $A,B\in \mathcal D_s(E)$ satisfy $A\subset B$ and $\diam(B)^s\le N\mu(A)$ for some constant $N\ge 1.$ Then by the $s$-regularity of $E$ we have
    \[
    \begin{split}
           &\frac{1}{\mu(A)}\le \frac{N}{\diam(B)^s}\le \frac{CN}{\mu(B)}\\
           &\diam(A)\ge C^{-\frac1s}\mu(A)^{\frac1s}\ge  (NC)^{-\frac1s} \diam(B)
    \end{split}.
    \]
  This clearly shows the validity of \eqref{eq:h_monot} for $h\in \{\beta_{q,\mathcal{A}},\kappa\}$.
\end{remark}

\subsection{New coefficients to measure flatness in metric spaces}\label{s:NewCoeff}

We introduce new coefficients to measure flatness of a set in
 a metric space. Here, \emph{flatness} means roughly speaking the
existence of approximate isometric embeddings of the set into
$\mathbb{R}^k$, in an $L^q$-sense. The natural measures to use in
this definition are $k$-dimensional Hausdorff measures, for
integer-valued $k$.

\begin{definition}\label{d:alpha}
For  $k\in \mathbb{N}$, a closed set $E \subset \X$ of locally
finite $\mathcal{H}^k$-measure, \textcolor{black}{and}
$\mu:=\mathcal{H}^k\lfloor_E$, we define
\begin{equation}\label{eq:nb}
\nb_{q,k}(S):=\inf_{\|\cdot\|\text{ norm on
}\mathbb{R}^k}\inf_{f:S\to \mathbb{R}^k}\left(
\frac{1}{\mu(S)^{2}}\int_{S}\int_S\left[\frac{\left|\sfd(x,y)-\|f(x)-f(y)\|\right|}{\mathrm{diam}(S)}\right]^q\,d\mu(x)\,d\mu(y)\right)^{1/q},
\end{equation}
for $S\in \mathcal{D}_s(E)$, {where the functions $f$ in the second infimum are assumed to be Borel.}
Moreover we define the number $\nb_{q,k,{\sf Eucl}}(S)$ by considering in the infimum \eqref{eq:nb} only the Euclidean norm $\|\cdot\|_{\sf Eucl}$ {(which we often denote simply by $|\cdot|$)}.
\end{definition}
 The numbers $\nb_{q,k}$ can be interpreted as an \textit{$L^q$-unilateral  version of the Gromov-Hausdorff $\beta$-numbers} from \cite[Definition 3.1.3]{2023arXiv230612933B}

The function $\nb_{q,k}$ associated to a $k$-regular set $E$
clearly has the monotonicity property required in
 Lemma \ref{l:coeff} (see Remark \ref{rmk:coeff}) and hence, again, for the purpose of this paper, the  constant ``$2$''  in
 the definition of
``$E\in \mathrm{GLem}(\nb_{q,k},p)$'' could be replaced by any
constant $K_0>1$ and the validity of $\mathrm{GLem}(\nb_{q,k},p)$ is independent of the choice of dyadic system.

\begin{remark}
\textcolor{black}{In a companion paper, \cite{CarlesonPart2}, we
consider a variant of the $\nb$-numbers that is more suitable for
comparison with the $\beta$-numbers from Example \ref{ex:Betas} in
Euclidean spaces. In particular, we define coefficients
$\nb_{1,\mathcal{V}_k}$ using orthogonal projections onto
$k$-dimensional affine planes and we prove that $k$-regular set
$E\subset\mathbb{R}^n$ is uniformly $k$-rectifiable if and only if
$E\in \mathrm{GLem}(\nb_{1,\mathcal{V}_k},1)$.}
\textcolor{black}{We refer to  \cite{CarlesonPart2} for the
definition of $\nb_{1,\mathcal{V}_k}$, but for illustrative
purposes we mention that for subsets of Euclidean spaces, already
the numbers $\iota_{q,k,\sf Eucl}$ and $\iota_{q,k}$ are related
to affine $k$-planes, as Propositions  \ref{prop:converse
inequalities} and \ref{prop:zero banach} below show (recall \eqref{eq:beta_intro} or Example \ref{ex:Betas} for the definition of the numbers $\beta_{q,\mathcal V_k}$). }
\end{remark}

\begin{proposition}\label{prop:converse inequalities}
     Let $E\in \reg_k(C)$ be a $k$-regular subset of $\mathbb R^n$, where $k\in \mathbb N$ and $k<n$ and let $\Delta$ be a system of dyadic cubes for $E.$ Then for all $q\in [1,\infty)$ and all $Q\in \Delta$ it holds
     \begin{equation}\label{eq:converse inequality}
         \beta_{q,\mathcal V_k}^2(2Q)\le \beta_{2q,\mathcal V_k}^2(2Q)\le  \tilde C {\nb_{q,k,{\sf Eucl}}(2Q)},
     \end{equation}
     where $\tilde C$ is a constant depending only on $k,p$ and $C$.
\end{proposition}

\begin{proposition}\label{prop:zero banach}
    Let $E\in \reg_k(C)$ be a $k$-regular subset of $\mathbb R^n$, where $k\in \mathbb N$ and $k\le n$ and let $\Delta$ be a system of dyadic cubes for $E.$ Suppose that for some $Q\in \Delta$ and $q\in [1,\infty)$  it holds
\[
\nb_{q,k}(Q)=0.
\]
Then $\nb_{q,k,{\sf Eucl}}(Q)=\beta_{q,\mathcal V_k}(Q)=0$. In particular up to a $\mathcal{H}^k$-zero measure set, $Q$ is contained in a $k$-dimensional plane.
\end{proposition}

\textcolor{black}{Propositions  \ref{prop:converse inequalities}
and  \ref{prop:zero banach} are  proven in \cite[Section
3.1]{CarlesonPart2}.}

\section{Sufficient conditions for the existence of regular covering
curves}\label{s:SuffRegCoverCurve}

This section aims at clarifying several technical points regarding
the question when a $1$-regular set $E$ in a complete, doubling
and quasiconvex metric space is contained in a regular curve
$\Gamma_0$. This discussion is partially motivated by an
application in \cite[Theorem 1.3]{MR3678492} about $1$-dimensional
singular integral operators in the Heisenberg group, and it will
be employed also in Section \ref{s:Char1UR} to state and
characterize a notion of \emph{uniform $1$-rectifiability} in a
large class of metric spaces.

In Section \ref{s:ConstrCover} we explain
how the existence of the covering curve $\Gamma_0$ for a set $E$ can be
derived from certain quantitative \emph{local} information on $E$. We will first prove the
\emph{bounded} case in Theorem \ref{t:FromConnToReg} and Corollary \ref{c:FromLipToReg} and then in Section \ref{s:unbounded} we
present a covering lemma that allows to extend this result to the
\emph{unbounded} case.
Finally, in Section \ref{s:ApplMenger},
these results are applied to construct a covering by a 1-regular
connected set, based on  integral bounds on the Menger curvature.
This proves and generalizes to quasiconvex metric spaces a result
first stated in \cite{MR2342818}. Finally, we review an
application of this result in the Heisenberg group.

\subsection{Construction of $1$-regular covering continua}\label{s:ConstrCover}

In this section, we consider conditions for a bounded set to be
contained in a $1$-regular \emph{curve}. Actually it will be sufficient to show that the set is
contained in a \emph{closed connected $1$-regular set of finite length}.
Indeed,
if $(\X,\sfd)$ is a
complete metric space, and $\Gamma\subset \X$ a closed connected
subset of finite $\mathcal{H}^1$ measure, then $\Gamma$ is
automatically a compact Lipschitz \emph{curve} by \cite[Lemma
2.8]{2021arXiv210906753B}, see also \cite[Lemma
2.2-2.3]{MR2337487}. Moreover, if we assume in addition that
$\Gamma$ is a $1$-regular set, then under the previous
assumptions, it will be automatically a \emph{$1$-regular curve}
in the sense of \cite[1.1]{MR2337487} by  \cite[Lemma
2.3]{MR2337487}.

\begin{thm}[From local covering by continua to global 1-regular covering]\label{t:FromConnToReg}
Let $(\X,\sfd)$ be a complete, doubling, and quasiconvex metric
space. Assume that $K\subset \X$ is a
bounded
set with the following property. There exists a constant $C>0$
such that, for every $x\in K$ and $0<r\leq \mathrm{diam}(K)$,
there is a closed connected set $\Gamma_{x,r}\subset \X$ with
$\Gamma_{x,r}\supseteq K\cap B_r(x)$ and
$\mathcal{H}^1(\Gamma_{x,r})\leq C r$. Then there exists a
(closed) connected set $\Gamma_0\in \mathrm{Reg}_1(C_0)$, with
$C_0$ depending only on the doubling and quasiconvexity constants
of  $(\X,\sfd)$ and on $C$, such that
\begin{displaymath} \Gamma_0 \supseteq K\quad\text{and}\quad
\mathcal{H}^1(\Gamma_0) \leq C \mathrm{diam}(K). \end{displaymath}
Moreover, the set $\Gamma_0$ may be chosen such that
\begin{displaymath}
\mathcal{H}^1(\Gamma_0)= \min \{\mathcal{H}^1(\Gamma):\, \Gamma
\text{ closed and connected, and }\Gamma \supseteq K\}.
\end{displaymath}
\end{thm}

 From Theorem \ref{t:FromConnToReg} above and the extension property of Lipschitz maps, we immediately infer the following result.
\begin{cor}[From local covering by Lipschitz images to global 1-regular covering]\label{c:FromLipToReg}
Let $(\X,\sfd)$ be a complete, doubling, and quasiconvex metric
space. Assume that $K\subset \X$ is a
bounded
set with the following property. There exists a constant $C>0$
such that, for every $x\in K$ and $0<r\leq \mathrm{diam}(K)$,
there is a set $A_{x,r}\subset [0,1]$ and a surjective
$Cr$-Lipschitz map $f_{x,r}\colon A_{x,r} \to K\cap B_r(x)$. Then
there exists a (closed) connected set $\Gamma_0\in
\mathrm{Reg}_1(C_0)$, with $C_0$ depending only on   the doubling
and quasiconvexity constants of  $(\X,\sfd)$ and on $C$, such that
\begin{displaymath} \Gamma_0 \supseteq K\quad\text{and}\quad
\mathcal{H}^1(\Gamma_0) \leq C' \mathrm{diam}(K),
\end{displaymath}
where $C'$ depends only on $C$ and the quasiconvexity constant of
$(\X,\sfd)$.
 Moreover, the set $\Gamma_0$ may be chosen such
that
\begin{displaymath}
\mathcal{H}^1(\Gamma_0)= \min \{\mathcal{H}^1(\Gamma):\, \Gamma
\text{ closed and connected, and }\Gamma \supseteq K\}.
\end{displaymath}
\end{cor}

\begin{proof}[Proof of Corollary \ref{c:FromLipToReg} using Theorem
\ref{t:FromConnToReg}] Let $(\X,\sfd)$ and $K$ be as in the
statement of Corollary \ref{c:FromLipToReg}. Since $(\X,\sfd)$ is
complete and quasiconvex, the pair $(\mathbb{R},\X)$ has the
Lipschitz extension property, see for instance \cite{MR4104279}.
In particular, there exists a constant $C'\geq 1$, depending only
on the given  $C>0$ and the quasiconvexity constant of
$(\X,\sfd)$, such that every $Cr$-Lipschitz map $f_{x,r}\colon
A_{x,r} \to K\cap B_r(x)$ as in the assumptions of the corollary
can be extended to a $C'r$-Lipschitz map $F_{x,r}:([0,1],|\cdot|)
\to (\X,\sfd)$ with ${F_{x,r}}|_{A_{x,r}} = f_{x,r}$. Then $K$
satisfies the assumptions of Theorem \ref{t:FromConnToReg} with $
\Gamma_{x,r}= F_{x,r}([0,1])$ and constant $C'$. Thus Corollary
\ref{c:FromLipToReg} follows from Theorem \ref{t:FromConnToReg}.
\end{proof}

Actually Theorem \ref{t:FromConnToReg}, and hence also Corollary
\ref{c:FromLipToReg},   hold without  the boundedness assumption
on $K$. We will prove this in the next Subsection
\ref{s:unbounded} (see in particular Corollary \ref{c:unbounded}),
since it will be first necessary to prove independently the
bounded case.

The proof of  Theorem \ref{t:FromConnToReg} is based on an idea
which we learned from Tuomas Orponen  in the case
$(\X,\sfd)=(\mathbb{R}^2,|\cdot|)$, see \cite[Exercise
1.6]{OLect}. Similar ideas have been used in \cite[p.197
ff]{MR2129693}. The strategy is to show that there exists a closed
connected set $\Gamma_0$ \emph{of minimal $\mathcal{H}^1$ measure}
containing $K$, and then to prove that this $\Gamma_0$ must in
fact be $1$-regular. Before explaining the details, we list the
key ingredients needed to run the argument in metric spaces. We
will apply the following result, which follows from versions of
Blaschke's theorem and  Go\l\c{a}b's theorem in metric spaces. For
comments and a  proof of  Go\l\c{a}b's theorem in this setting,
see also \cite[p.840, p.846]{MR3018174}.

\begin{thm}[Existence result; Theorem 4.4.20 in
\cite{MR2039660}]\label{t:exist} Let $(\X,\sfd)$ be a proper
metric space. Suppose that $K\subset \X$ is nonempty, and assume
that there exists a closed connected set $\Gamma$ in $\X$ such
that $\mathcal{H}^1(\Gamma)<\infty$ and $\Gamma\supseteq K$. Then
the minimum problem
\begin{displaymath}
\min \left\{\mathcal{H}^1(\Gamma):\,\Gamma\text{ closed and
connected, and }\Gamma \supseteq K \right\}
\end{displaymath}
has a solution.
\end{thm}
We recall that a \emph{proper} metric space is one in which all
closed balls are compact, and that every proper metric space is
complete, and every complete and doubling metric space is proper.

Theorem \ref{t:exist}
 will be used to show the existence of a
minimal-length covering continuum $\Gamma_0$ for the set $K$ in
Theorem \ref{t:FromConnToReg}. To prove $1$-regularity of
$\Gamma_0$, we will argue by contradiction and construct a
covering continuum $\Gamma_0'$ of smaller length in the
hypothetical scenario that $1$-regularity of $\Gamma_0$ fails.
This construction of $\Gamma_0'$ proceeds by locally modifying
$\Gamma_0$ using the assumption on $K$. The only step which is
trickier in our setting than in the case $\X=\mathbb{R}^2$ is to
maintain connectedness of the modified set. In $\mathbb{R}^2$ this
can be achieved simply by adding a suitable circle (and possibly a
line segment). Our construction instead uses the following
observation.

\begin{proposition}[Short paths connecting points in quasiconvex doubling
spaces]\label{p:Conn} Assume that $(\X,\sfd)$ is a quasiconvex
metric space which is doubling with constant $D\geq 1$. Then there
exists a constant $\lambda >0$ (depending on the doubling and
quasiconvexity constants) such that for all $x\in \X$ and $r>0$
the following holds. Every finite set $P\subset \overline{B_r(x)}$
is contained in a closed connected set $\Gamma$ with
\begin{equation}\label{eq:LengthFiniteSetP}
\mathcal{H}^1(\Gamma) \leq \lambda r
\,\mathrm{card}(P)^{(2D-1)/2D}.
\end{equation}
\end{proposition}

\begin{proof} Without loss of generality, we may assume that
$\mathrm{card}(P)\geq 2^{2D}$. Since $(\X,\sfd)$ is doubling with
constant $D\geq 1$, the ball $B_r(x)$ can be covered by balls
\begin{displaymath} B_j=B_{r\,\mathrm{card}(P)^{-1/2D}}(x_j),\quad
j=1,\ldots,N,
\end{displaymath}
with bounded cardinality and overlap:
\begin{equation}\label{eq:ball_cond}
N\lesssim \mathrm{card}(P)^{\frac{1}{2}}\quad \text{and}\quad
\sup_{y\in \X}\mathrm{card}(\{i\in \{1,\ldots,N\}: y\in B_i
\})\lesssim_D 1,
\end{equation}
recall Definition \ref{def:doubling} and the subsequent comment.
It will become clear at the end of the proof why the radii of the
balls $B_j$ were chosen as above. At this point we just note that
they are of the form $\varepsilon r$ with $\varepsilon \in
(0,1/2]$.

Up to  removing unnecessary balls, we can assume that $B_j\cap B_r(x)\neq \emptyset$ for all $j.$
To proceed, we connect every $x_j$, $j=2,\ldots,N$ to $x_1$
by a curve of length $\lesssim r$, using the quasiconvexity of
$(\X,\sfd)$ and the fact that $\sfd(x_1,x_j)\leq {4r}$. The union of
these curves is a closed connected set $\Gamma_0$ with
\begin{displaymath}
\mathcal{H}^1(\Gamma_0) \lesssim r N \lesssim r\,
\mathrm{card}(P)^{\frac{1}{2}}.
\end{displaymath}
Second, for every $j=1,\ldots,N$, each point in $P\cap B_j$ can be
connected to the center $x_j$ of the ball $B_j$ by a curve of
length $\lesssim r\,\mathrm{card}(P)^{-1/2D}$, again thanks to
quasiconvexity. Thus the points $P\cap B_j$ can be connected to
$x_j$ by a closed connected set $\Gamma_j$ of total measure
\begin{displaymath}
\mathcal{H}^1(\Gamma_j) \lesssim \mathrm{card}(P\cap B_j) \, r\,
\mathrm{card}(P)^{-1/2D}.
\end{displaymath}
Moreover,
\begin{displaymath}
\mathcal{H}^1\left(\bigcup_{j=1}^N \Gamma_j\right)\leq\sum_{j=1}^N
\mathcal{H}^1(\Gamma_j) \lesssim r\,
\mathrm{card}(P)^{-\frac{1}{2D}} \sum_{j=1}^N  \mathrm{card}(P\cap
B_j)\lesssim_D r\, \mathrm{card}(P)^{1-\frac{1}{2D}},
\end{displaymath}
where we have used the controlled overlap of the balls $B_j$,
$j=1,\ldots,N$ in the last inequality, as quantified in
\eqref{eq:ball_cond}. Finally, $\Gamma=\Gamma_0 \cup
\left(\bigcup_{j=1}^n \Gamma_j\right)$ is a covering continuum for
$P$ with the desired property \eqref{eq:LengthFiniteSetP} since
\begin{displaymath}
\tfrac{1}{2}\leq 1-\tfrac{1}{2D} \quad \text{for }D\geq 1.
\end{displaymath}
\end{proof}

We are now ready to prove the main theorem of this section.
\begin{proof}[Proof of Theorem \ref{t:FromConnToReg}]
Let $(\X,\sfd)$ and $K$ be as in the statement of Theorem
\ref{t:FromConnToReg}. Applying the assumption to a point $x\in K$
and $r=\mathrm{diam}(K)$, we find a closed connected set $\Gamma
\supset K$ with $\mathcal{H}^1(\Gamma) \leq
C\mathrm{diam}(K)<\infty$. Since every complete and doubling
metric space is proper, we can the apply Theorem \ref{t:exist} to
deduce that there exists a closed and connected set $\Gamma_0
\supseteq K$ with smallest $\mathcal{H}^1$ measure among all such
sets. Being connected, the set $\Gamma_0$ is automatically
\emph{lower} $1$-regular with a universal constant, that is
\begin{displaymath}
\mathcal{H}^1(\Gamma_0 \cap B_r(x))\gtrsim r,\quad x\in
\Gamma_0,\, 0<r\leq 2\,\mathrm{diam}(\Gamma_0),
\end{displaymath}
see for instance \cite[Lemma 4.4.5]{MR2039660}. Hence it suffices
to prove that $\Gamma_0$ is also \emph{upper} $1$-regular.

To this end, fix a  constant $C_0>2C$. The precise value of $C_0$
will be determined later. Towards a contradiction, we assume that
$\Gamma_0$ fails to be upper $1$-regular with constant $C_0$, that
is, $\Gamma_0 \notin \reg_1^+(C_0)$. Thus there exists $x_0\in
\Gamma_0$ and $0<r\leq\mathrm{diam}(\Gamma_0)$ such that
\begin{equation}\label{eq:ExceptBall}
\mathcal{H}^1(\Gamma_0 \cap B_r(x_0))> C_0 r.
\end{equation}
(If upper Ahlfors regularity fails, it has to fail for a radius
$r\leq \mathrm{diam}(\Gamma_0)$.)
 We want to work with the essentially largest radius with
 this property. To be more precise, we set
\begin{align*}
r_0&:= \sup\left\{s\in [r,\mathrm{diam}(\Gamma_0)]:\,
\mathcal{H}^1(\Gamma_0 \cap B_s(x_0))> C_0 s \right\}\\& =
\sup\left\{s\in [r,\mathrm{diam}(\Gamma_0)]:\, q(s)> C_0\right\},
\end{align*}
where
\begin{displaymath}
q(s):= \frac{\mathcal{H}^1(\Gamma_0 \cap B_s(x_0))}{s},\quad s\in
[r,\mathrm{diam}(\Gamma_0)].
\end{displaymath}
Clearly, $q(r)>C_0$ and $q(s)<C_0$ for $s\in
[\mathrm{diam}(\Gamma_0)/2,\mathrm{diam}(\Gamma_0)]$ since
\begin{displaymath}
\mathcal{H}^1(\Gamma_0 \cap B_s(x_0))\leq
\mathcal{H}^1(\Gamma_0)\leq 2
\frac{C\mathrm{diam}(K)}{\mathrm{diam}(\Gamma_0)}
\,\tfrac{1}{2}\mathrm{diam}(\Gamma_0) <  C_0 \, \tfrac{1}{2}
\mathrm{diam}(\Gamma_0)
\end{displaymath}
for all $s$, by the minimality property of $\Gamma_0$ and the
existence of $\Gamma \supset K$ with $\mathcal{H}^1(\Gamma) \leq
C\mathrm{diam}(K)$. It follows that $r\leq r_0
<\mathrm{diam}(\Gamma_0)/2$ (and thus $2r_0<
\mathrm{diam}(\Gamma_0)$).

We will now locally modify $\Gamma_0$ at $B_{r_0}(x_0)$ to
construct a closed connected $\Gamma_0'\supseteq K$ with
$\mathcal{H}^1(\Gamma_0')<\mathcal{H}^1(\Gamma_0)$. This will
contradict the minimality property of $\Gamma_0$ and thus show
that the counter assumption on the existence of a ball as in
\eqref{eq:ExceptBall} cannot be true. Thus $\Gamma_0$ must in fact
be upper $1$-regular with constant $C_0$.

We now explain the modification of $\Gamma_0$ locally around
$B_{r_0}(x_0)$. To ensure connectedness of the modified set
$\Gamma_0'$, we will apply Proposition \ref{p:Conn}. Hence we
would like to consider a ball $B_{\rho}(x_0)$ with $\rho\sim r_0$
such that we have a suitable upper bound on
$\mathrm{card}(\Gamma_0 \cap \partial B_{\rho}(x_0))$. For this purpose, we
apply the coarea (Eilenberg) inequality \cite[Theorem
1]{eilenberg} to the $1$-Lipschitz function $f:(\X,\sfd) \to
\mathbb{R}$ given by $f(x):=\sfd(x,x_0)$. Then
\begin{align*}
\int_{[r_0,2r_0]}^{\ast} \mathcal{H}^0(\Gamma_0 \cap \partial
B_s(x_0))\, ds & \leq\int_{\mathbb{R}}^{\ast}
\mathcal{H}^0([\Gamma_0 \cap B_{2r_0}(x_0)\setminus B_{r_0}(x_0)]
\cap f^{-1}(\{s\}))\, ds
\\&\leq \mathcal{H}^1(\Gamma_0\cap B_{2r_0}(x_0)\setminus B_{r_0}(x_0))\leq C_0  r_0,
\end{align*}
where we  used the maximality property of $r_0$ in the last step.
It follows that there must exist $\rho \in [r_0,2r_0]$ such that
\begin{displaymath}
 \mathcal{H}^0(\Gamma_0 \cap
\partial B_{\rho}(x_0)) \leq  C_0.
\end{displaymath}
We next apply Proposition \ref{p:Conn} to the point set
\begin{displaymath}
P:= \Gamma_0 \cap
\partial B_{\rho}(x_0).
\end{displaymath}
Since $\rho <\mathrm{diam}(\Gamma_0)$, $x_0 \in \Gamma_0$ and
$\Gamma_0$ is connected, $P$ contains clearly at least one point.
Proposition \ref{p:Conn}  allows us to find a closed connected set
$\Gamma_P$ in $(\X,\sfd)$ with
\begin{equation}\label{eq:rho_upper}
\Gamma_P \supseteq P\quad\text{and}\quad \mathcal{H}^1(\Gamma_P)
\leq \lambda \rho \left(C_0\right)^{\frac{2D-1}{2D}} = \lambda
C_0^{\frac{-1}{2D}} \,  C_0 \rho,
\end{equation}
where $\lambda$ depends only on the doubling and quasiconvexity
constants of $(\X,\sfd)$.

The set $[\Gamma_0 \setminus B_{\rho}(x_0)]\cup \Gamma_P$ is
connected by construction. If $K \cap B_{\rho}(x_0)
$ is empty, there is nothing further to be done, but
otherwise, we have to enlarge our continuum in order to cover
$K\cap B_{\rho}(x_0)
$ as well. The assumption of Theorem
\ref{t:FromConnToReg} allows us to find a (possibly empty) closed connected set
$\Gamma_{x_0,\rho}$ such that
\begin{equation}\label{eq:rho_upper2}
\Gamma_{x_0,\rho}\supseteq K\cap B_{\rho}(x_0)\quad\text{and}\quad
\mathcal{H}^1(\Gamma_{x_0,\rho}) \leq 2C \rho.
\end{equation}
By quasiconvexity, it is further possible to connect $\Gamma_P$
and $\Gamma_{x_0,\rho}$ by a curve $\Gamma_{P,x_0,\rho}$ with
\begin{equation}\label{eq:rho_upper3}
\mathcal{H}^1(\Gamma_{P,x_0,\rho}) \lesssim \rho.
\end{equation}
If $\Gamma_{x_0,\rho}=\emptyset$ we simply put $\Gamma_{P,x_0,\rho}\coloneqq \emptyset.$
By construction, we know that
\begin{equation}\label{eq:rho_lower}
\mathcal{H}^1(\Gamma_0 \cap B_{\rho}(x_0))\geq
\mathcal{H}^1(\Gamma_0 \cap B_{r_0}(x_0)) =  C_0 r_0\geq
\tfrac{C_0}{2} \rho.
\end{equation}
Hence it is clear that we can choose the constant $C_0$ large
enough (depending only on the doubling and quasiconvexity
constants of $(\X,\sfd)$ and on $C$) such that the upper bounds
for the $\mathcal{H}^1$ measure in \eqref{eq:rho_upper} -
\eqref{eq:rho_upper3} are each less than $C_0 \rho /6$, so that
\begin{equation}\label{eq:smaller}
\mathcal{H}^1(\Gamma_P \cup \Gamma_{x_0,\rho}
\cup\Gamma_{P,x_0,\rho} ) < \tfrac{C_0}{2} \rho \leq
\mathcal{H}^1(\Gamma_0 \cap B_{\rho}(x_0)).
\end{equation}
Since
\begin{displaymath}
\Gamma'_0:= [\Gamma_0 \setminus B_{\rho}(x_0)] \cup \Gamma_P \cup
\Gamma_{x_0,\rho} \cup \Gamma_{P,x_0,\rho}
\end{displaymath}
is a closed connected set of smaller $\mathcal{H}^1$
measure than $\Gamma_0$, we have reached a contradiction with the
minimality property of $\Gamma_0$. Thus the counter assumption
cannot hold and in fact $\Gamma_0\in \reg^+_1(C_0)$, and
eventually, $\Gamma_0\in \mathrm{Reg}_1(C_0')$, where $C_0'$ the
maximum of the lower and upper regularity constants of $\Gamma_0$.
\end{proof}

\subsubsection{The unbounded case}\label{s:unbounded}
Next we give a sufficient criterion ensuring that an unbounded set
which can be locally covered by connected 1-regular sets, can be
itself covered by a connected 1-regular set. The argument is
inspired by \cite[p.202 ff]{MR2129693}.
\begin{proposition}\label{prop:unbounded connected covering}
    Let $(\X,\sfd)$ be a quasiconvex  and doubling metric space and $t_0>1$ be any constant. Let $E\subset \X$ be a set such that  for every $x\in E$ and $r>0$  there exists a connected set   $\Gamma_{x,r}\in \reg_1(C)$, satisfying $B_{r}(x)\cap E\subset \Gamma_{x,r}\subset B_{t_0r}(x)$ and where $C>0$ is a constant independent of $x$ and $r$. Then $E$ is contained in a connected set $\Gamma_0\in \reg_1(\tilde C)$, where $\tilde C$ is a constant depending only on $t_0,C$ and the doubling and quasiconvexity constants of $(\X,\sfd)$.
\end{proposition}

First we observe that Proposition \ref{prop:unbounded connected covering} allows immediately to improve the statements of Theorem \ref{t:FromConnToReg} and Corollary \ref{c:FromLipToReg}
to the unbounded case.
\begin{cor}\label{c:unbounded}
   Theorem \ref{t:FromConnToReg} and Corollary \ref{c:FromLipToReg} hold true also for unbounded sets $K$  (if the respective assumptions are satisfied for all $0<r<\infty$).
\end{cor}

\begin{proof}
    We only need to remove the boundedness assumption from Theorem \ref{t:FromConnToReg}, then it would be automatically removed also from Corollary \ref{c:FromLipToReg} which is deduced from it. Let $K\subset \X$ be a set satisfying the hypotheses of Theorem \ref{t:FromConnToReg} with constant $C$,
    except that it is unbounded. For every $x \in K$ and $r>0$, the set $K'\coloneqq B_r(x)\cap K$ satisfies all the assumptions of Theorem \ref{t:FromConnToReg}, being a bounded subset of $K.$ Hence by Theorem \ref{t:FromConnToReg} (in the bounded case) we deduce that $B_r(x)\cap K$ is contained in a closed connected set $\Gamma_{x,r}\in \reg_1(C_0)$, where $C_0$ is a constant depending only on $C$ and the doubling and quasiconvexity constants of $(\X,\sfd),$ and such that $\mathcal H^1(\Gamma_{x,r})\le 2Cr.$ This and  the 1-regularity show $\diam(\Gamma_{x,r})\le 2C_0Cr$, so that $\Gamma_{x,r}\subset B_{2C_0Cr}(x)$. Hence the assumptions of Proposition \ref{prop:unbounded connected covering} with $E=K$ are satisfied taking $t_0=2CC_0$, which concludes the proof.
\end{proof}
Theorem \ref{t:FromConnToReg}, in the bounded case, will be needed
for the proof of Proposition \ref{prop:unbounded connected
covering}, which is why we did not prove that theorem directly in
the full version. The main technical tool needed for the proof of
Proposition \ref{prop:unbounded connected covering}   is the
following covering lemma.
\begin{lemma}\label{lem:spiral convering}
    Let $(\X,d)$ be a doubling metric space, fix $x_0 \in \X$ and let $t>1$ be any constant. Then there exist  a (countable) family of balls $\mathcal B$ with radius $\ge 1$ and  a constant $M\ge 2$ depending only on $t$ and the doubling constant of $(\X,d)$ such that the following hold:
    \begin{enumerate}[label=\roman*)]
        \item the family $\mathcal B$ covers $\X$ and the covering $\{tB\}_{B \in \mathcal B}$ has multiplicity less than $M$,
        \item for every $R\ge 1$
        \[
        \sum_{B\in \mathcal B,\, tB\cap B_R(x_0)\neq    \emptyset} r(B)\le
        MR,
        \]
        where $r(B)$ denotes the radius of $B$,
        \item for every $B_r(x)\in \mathcal B$
         it holds that $d(x,x_0)\le Mr$,
        \item   $\# \{B' \in \mathcal B \ : \ tB'\cap tB\neq \emptyset \}\le M$, for every $ B \in \mathcal B$.
    \end{enumerate}
\end{lemma}
\begin{proof}
    Fix $x_0 \in \X$, $t>1$ and a constant $\lambda> 2t$.
    We construct the family $\mathcal B$ as union of families
 $\mathcal B_k$, $k\in \mathbb N\cup\{0\}$,
    of balls defined as follows. Set $\mathcal B_0\coloneqq
    \{B_1(x_0)\}$.
    For every $k \in \mathbb N$ denote $A_k\coloneqq B_{\lambda^k}(x_0)\setminus B_{\lambda^{k-1}}(x_0)$
     and consider a set $\mathcal F_k\subset A_k$ such that $d(x,y)\ge \lambda^{k-3}$ for every distinct $x,y \in \mathcal F_k$
 and assume that $\mathcal{F}_k$ is a maximal set with this property. Set $\mathcal B_k\coloneqq \{B_{\lambda^{k-3}}(x)\}_{x \in \mathcal F_k}$
  and $\mathcal B\coloneqq \cup_{k=0}^\infty \mathcal B_k$. By construction $\mathcal B$ is a covering of $\X$ and $iii)$ holds provided $M\geq \lambda^3$.
    Since the space $(\X,\sfd)$ is doubling and $\diam(A_k)\le 2\lambda^k$, the set $A_k$ can be covered by $\le 4^D \lambda^{3D}$ sets of diameter at most $\lambda^{k-3}/2$, where  $D>0$ is the doubling constant of $(\X,\sfd)$ (recall Def.\ \ref{def:doubling}). Since each of these sets intersects at most one of the centers of the balls in $\mathcal B_k$, we deduce that
    \begin{equation}\label{eq:cardinality Bk}
        \# \mathcal B_k \le 4^D \lambda^{3D}.
    \end{equation}
    By construction the balls $\mathcal B_k$ cover the annulus $A_k$ and again by the doubling property of $(\X,\sfd)$, such covering has multiplicity $\le M$, provided $M$ is big enough depending only on the doubling constant of $(\X,\sfd)$.
    Next we observe that, by the triangle inequality, for every $B\in \mathcal B_k$ and $x \in tB$ it holds that
    \[
    \lambda^{k-2}<\lambda^{k-2}\left(\lambda-\frac{t}{\lambda}\right)\le \lambda^{k-1}-t\lambda^{k-3}\le d(x,x_0)\le \lambda^{k}+t\lambda^{k-3}\le \lambda^{k+1}\left(\frac1\lambda+ \frac{t}{\lambda^4}\right)<\lambda^{k+1},
    \]
    where in the first and last inequality we used that $\lambda> 2t>2$.
    In particular $tB\subset A_{k-1}\cup A_{k}\cup A_{k+1}$ for every $B \in \mathcal B_k$ and $k\ge 2$.
    These observations together with \eqref{eq:cardinality Bk} show at once both $i)$ and $iv)$ for suitably chosen $M$ depending only on $D$. It remains to show $ii)$. Fix $R\ge 1$ and $k \in \mathbb{N}$ such that $R \in[\lambda^{k-1},\lambda^{k}).$ By what we just observed:
    \[
    \{B \in \mathcal B \ : \ tB\cap B_R(x_0)\neq \emptyset\}\subset     \{B \in \mathcal B \ : \ tB\cap B_{\lambda^k}(x_0)\neq \emptyset\}\subset \cup_{i=0}^{k+1} \mathcal B_i.
    \]
    Therefore
    \begin{align*}
            \sum_{B\in \mathcal B,\, tB\cap B_R(x_0)\neq    \emptyset} r(B)&\le  \sum_{i=0}^{k+1} \sum_{B \in \mathcal B_i} r(B)\overset{\eqref{eq:cardinality Bk}}{\le }
         4^D \lambda^{3D} \left(1+\sum_{i=1}^{k+1} \lambda^{i-3}\right)\\
        &\le     4^D \lambda^{3D} (1+\lambda^{k-1})\le     4^D \lambda^{3D}2R,
    \end{align*}
    which proves $ii)$ for suitable $M$. Combining what we
    said so far, $M$ can be chosen large enough, depending only on
    $t$ and the doubling constant of $(\X,\sfd)$, such that the
    conditions $i)$--$iv)$ hold.
\end{proof}

The covering of balls given by Lemma \ref{lem:spiral convering} is useful thanks to the following fact.
\begin{lemma}\label{lem:upper regularity union}
     Fix $(\X,\sfd)$  a doubling metric space, $x_0 \in \X$ and $t>1$ a constant. Let $\mathcal B$   be a family of balls as given by Lemma \ref{lem:spiral convering} applied with $t$ and $x_0$. Let $\{\Gamma_B\}_{B\in \mathcal B}$ be  upper 1-regular sets $\Gamma_B\subset tB$,  with $\Gamma_B \in \reg_1^+(C)$, where $C>0$ is a constant independent of $B.$ Then  $\Gamma\coloneqq \cup_{B\in \mathcal B}\Gamma_B\in \reg_1^+(\tilde C)$, where $\tilde C$ depends only on  $C,t$ and the doubling constant of
     $(\X,\sfd).$ Moreover if $\Gamma_B$ is closed for every $B\in \mathcal B,$ then $\Gamma$ is  closed.
\end{lemma}
\begin{proof}
    Let $z \in\Gamma$ be arbitrary. {Since $\mathcal{B}$ covers $\X$, there exists at least one $B\in \mathcal B$ with $z\in B$.  Denote by $r(B)$ the radius of $B$. Now fix $R>0$ arbitrarily.}
    If $R\le (t-1)r(B)$, then $B_R(z)\subset tB$.
{By the assumption of the lemma, $\Gamma_{B'}\cap B_R(z)\subset tB' \cap B_{R}(z)$ for every $B'\in \mathcal{B}$. It follows that if a set $\Gamma_{B'}$ intersects $B_R(z)$, then necessarily $tB\cap tB'\neq \emptyset$.}
      Therefore, by $iv)$ of Lemma \ref{lem:spiral convering}, we have
    \[
    \mathcal H^1(B_R(z)\cap \Gamma)=\mathcal H^1(B_R(z)\cap tB\cap \Gamma)\le \sum_{B'\in\mathcal B, \, tB\cap {tB'}\neq \emptyset}  H^1(B_R(z)\cap \Gamma_{B'})\le  MC R,
    \]
    where $M\ge 1$ is the constant given by Lemma \ref{lem:spiral convering}.
    If instead $R\ge (1-t)r(B),$  by $iii)$ of Lemma \ref{lem:spiral convering}{ and since $z\in B\in \mathcal{B}$, it holds $d(z,x_0)\le r(B)+Mr(B)\le (1-t)^{-1} (1+M)R.$ Hence
    $B_R(z)\subset B_{c_t R}(x_0)$, where $c_t\coloneqq (1-t)^{-1} (1+M)+1$.} Therefore by $ii)$ of Lemma \ref{lem:spiral convering} we have
    \begin{align*}
        \mathcal H^1(B_R(z)\cap \Gamma) &\le    \mathcal H^1(B_{c_t R}(x_0)\cap \Gamma) \le \sum_{B \in \mathcal B} \mathcal H^1(B_{c_t R}(x_0)\cap \Gamma_B)\\
        &\le     \sum_{B \in \mathcal B,\, tB\cap B_{c_t R}(x_0) \neq \emptyset} \mathcal H^1(\Gamma_B)\le Ct \sum_{B \in \mathcal B,\, tB\cap B_{c_t R}(x_0) \neq \emptyset}r(B)\le Ctc_t\cdot MR,
    \end{align*}
where we used again that $\Gamma_B\subset tB.$

Finally assume that each $\Gamma_B$ is closed and let $\{x_n\}_n\subset  \Gamma$ be a converging sequence in $(\X,\sfd).$ In particular $\{x_n\}_n$ is bounded and by  $iii)$ in Lemma \ref{lem:spiral convering} we deduce that $\{x_n\}_n$ is contained in a finite union of closed sets $\Gamma_B$. This shows that $x_n$ must converge to a point in some $\Gamma_B\subset \Gamma$ and so $\Gamma$ is closed.
\end{proof}

We are now ready to prove Proposition \ref{prop:unbounded connected covering}.
\begin{proof}[Proof of Proposition \ref{prop:unbounded connected covering}]
    Fix a constant $t>t_0$ to be chosen later,  depending only on $t_0$ and the the doubling and quasiconvexity constants of $(\X,\sfd)$.
    Moreover throughout the proof $\tilde C$
    denotes a constant  whose value might change from line to line,
    but depending only on $C,t_0$ and the doubling and quasiconvexity constants of $(\X,\sfd)$.

     Let $\mathcal B$ be the family of  balls given by Lemma \ref{lem:spiral convering} applied to the metric space $(\X,\sfd)$
      and with constant $t>1.$
      For every ball $B=B_r(x)\in \mathcal B$ we build a connected set $\Gamma_B\subset tB$ as follows.
      If $B\cap E=\emptyset$, set $\Gamma_B=\emptyset$, otherwise fix $y \in B\cap E.$
      By assumption there exists a connected set $\Gamma_{B}\in \reg_1(C)$ such that
  $$B\cap E\subset B_{2r}(y)\cap E \subset \Gamma_B \subset B_{2t_0r}(y) \subset B_{2t_0r+r}(x)\subset B_{tr}(x),$$
  provided $t\ge 2t_0+1.$
  Define $\Gamma\coloneqq \cup_{B \in \mathcal B}\Gamma_B$.
  Since the family $\mathcal B$ covers $\X$, the set $\Gamma$ contains $E$.
  Moreover by Lemma \ref{lem:upper regularity union} it follows that $\Gamma\in \reg_1^+(\tilde C)$  and that  if  each $\Gamma_B$ is closed then also $\Gamma$ is closed.  However $\Gamma$ is not necessarily connected.
   To fix this, for every $B=B_r(x)$ we  build an additional closed connected set $\Gamma_B'$ as follows.
   First observe that by Theorem \ref{t:FromConnToReg} and since $(\X,\sfd)$ is quasiconvex we have that for every $x_1,x_2\in \X$
   there exists a \emph{$1$-regular} {(closed)}
   connected set $\gamma\in \reg_1(\tilde C)$
   such that $x_1,x_2 \in \gamma$ and $\diam(\gamma)\le L\sfd(x_1,x_2)$, where $L$ depends only on the doubling and quasiconvexity constants of $\X$.

   We pass to the construction of $\Gamma_{B}'$.
   For every $B'=B_{r'}(x')\in \mathcal B$ such that $B'\cap B\neq \emptyset$
   and $r'\le r$ consider  a $1$-regular closed connected set $\gamma$ containing $x$ and $x'$ (as above)
   and define $\Gamma_B'$ as the union of all such sets.
   Moreover, if $E\cap B\neq \emptyset$, we add to $\Gamma_B'$ also a 1-regular closed connected set $\gamma$ containing
   $x$ and $y \in E\cap B,$ again as above.
   In particular, by $iv)$ of Lemma \ref{lem:spiral convering}, $\Gamma_{B}'$
   is the union of at most an $M$-number of $1$-regular (closed) and connected sets
   of type $\reg_1(\tilde C)$, all containing the point $x$, where $M>0$ is a constant depending only on $t_0$
   and the doubling and quasiconvexity constants of $\X$ (since it depends also on $t$).
   Therefore, up to modifying the value of $\tilde C$, $\Gamma_B'\in \reg_s(\tilde C)$  {(note that $\Gamma_B'$ is closed)}.
   Moreover $\diam(\Gamma_B')\le Lr.$ Finally set $\Gamma'\coloneqq \cup_{B \in \mathcal B}
   \Gamma_B'$.

    Observe that by construction, up to choosing $t>L,$  it holds  $\Gamma_B'\subset tB$ for all $B \in \mathcal B$.
    Therefore by Lemma \ref{lem:upper regularity union} we deduce that $\Gamma'\in \reg_s^+(\tilde C)$ (i.e., it is upper 1-regular) and that $\Gamma'$ is closed.
    Hence  $\Gamma_0\coloneqq \Gamma\cup \Gamma' \in \reg_1^+(\tilde C)$, $\Gamma_0$ contains $E$ and $\Gamma_0$ is  closed if $\Gamma$ is closed (which is the case if all $\Gamma_B$ are closed).
    It remains to show that $\Gamma_0$ is connected (lower 1-regularity would then also follow,
    see e.g.\ \cite[Lemma 4.4.5]{MR2039660}).
     Since each $\Gamma_B$ is connected and intersects $\Gamma'$,
     it is sufficient to show that $\Gamma'$ is connected.
     Suppose by contradiction that $\Gamma'\subset U\cup V$ where  $U$ and $V$
     are two disjoint open subsets of $\X$ such that $U\cap \Gamma'\neq \emptyset \neq V\cap \Gamma'$.
     The centers of the balls in $\mathcal B$ must be contained in $\Gamma'$.
     Indeed if $B=B_r(x)\in \mathcal B$ and $x\notin \Gamma'$, then by construction the open ball $B$ does not intersect any (other) ball in $\mathcal B$,
     which contradicts the connectedness of $\X$ (since it is quasiconvex). Therefore the center of every ball in $\mathcal B$ must belong to either $U$ or $V$. Moreover, since by construction each set $\Gamma_B'$ is connected and contains the center of $B$, each $V$ and $U$ contain at least one center of a ball in $\mathcal B$. However if two balls $B'$ and $B$ are centered in $U$ and $V$ respectively, then $B\cap B'=\emptyset$, otherwise by construction both their centers belong to a connected set contained in $U\cup V$, which contradicts the fact that $U$ and $V$ are open and disjoint. This means that $\X$ is covered by two (non-empty) family of countable balls with the property  that every ball in one family does not intersects any other ball in the other. This concludes the proof.
\end{proof}

\subsection{Application in the context of Menger
curvature}\label{s:ApplMenger}

Corollary \ref{c:FromLipToReg} allows to prove a version of
\cite[Theorem 3.11]{MR2342818} for metric spaces that are
quasiconvex, doubling and complete (but not necessarily geodesic).
The argument is based on the following result proved in
\cite[Theorem 1.1]{MR2297880} (see the comment around
\cite[(2.1)]{MR2297880}). Recall Section \ref{sec:glem} for the
definition of the Menger curvature $c(x_1,x_2,x_3).$

\begin{thm}[Hahlomaa]\label{thm:hah}
    There exists a universal constant $K_0>1$ such that the following holds.
    Let $(E,d)$ be a bounded 1-regular metric space with $E\in \reg_1(C)$, and such that
    \[
    \mathbf{c}(E)\coloneqq
    \int_{\mathcal F} c^2(x_1,x_2,x_3)d \mathcal H^1(x_1)d\mathcal H^1(x_2)d\mathcal H^1(x_3)<+\infty,
    \]
    where $\mathcal{F}\coloneqq \{(x_1,x_2,x_3) \in E \ : \  d(x_i,x_j)\le K_0 d(x_k,x_l), \, \text{ for all }i,j,k,l\in\{1,2,3\}, \, k\neq l  \}$.

    Then there exists  $A\subset [0,1]$ and a Lipschitz surjective function $f:A\to E$ such that
    \[
    {\rm Lip}(f)\le D(\mathbf{c}(E)+\diam(E)),
    \]
    where $D$ is a constant depending only on $C.$
\end{thm}
We will also need the following elementary result that allows to
localize the $s$-regularity condition.
\begin{proposition}[Localizing $s$-regular sets, Lemma 2.2 in \cite{MR4485846}]\label{prop:localize}
    Suppose that $E\in \reg_s(C)$ is a subset of a metric space $(\X,\sfd)$. Then for all $x\in E$ and
     $r\in (0,\diam(E))$ there exists a set $E_{x,r}\in \reg_s(\tilde C)$, where $\tilde C$ is a constant depending only on $s$ and $C$, such that
    \[
    B_r(x)\cap E\subset E_{x,r}\subset B_{3r}(x)\cap E.
    \]

\end{proposition}
In \cite[Lemma 2.2]{MR4485846}, only closed regular sets that are
regular at arbitrarily large scales are considered,
 however the same proof without changes works for the version that we reported above.

\begin{cor}[Integral Menger curvature condition]\label{c:IntMen}
    There exists a universal constant $K_0>1$ such that the following holds.
    Let $E\in \reg_1(C_E)$ be a subset of a complete, doubling and quasiconvex metric space $(\X,\sfd)$
    and suppose that there exists a constant $C>0$ such that
    \begin{equation}\label{eq:assumption}
        \int_{\mathcal F\cap (B_R(x))^3} c^2(x_1,x_2,x_3)d \mathcal H^1(x_1)d\mathcal H^1(x_2)d\mathcal H^1(x_3)\le CR,\quad  x \in E, \,
  R>0,
    \end{equation}
    where $\mathcal{F}\coloneqq \{(x_1,x_2,x_3) \in E^3 \ : \  d(x_i,x_j)\le K_0 d(x_k,x_l), \, \text{ for all }i,j,k,l\in\{1,2,3\}, \, k\neq l  \}$.
 Then $E$ is contained in a closed connected set  $\Gamma_0\in \reg_1(\tilde C_0)$
  with $\mathcal{H}^1(\Gamma_0)\le \tilde C \diam(E)$, where $\tilde C\ge1$ is a constant depending only on $C$, $C_E$ and the quasiconvexity constant of
  $(\X,\sfd)$, while
$\widetilde{C}_0$ is a constant that may additionally depend also
on the doubling constant of
 $(\X,\sfd)$.
\end{cor}

\begin{proof}
We want to apply Corollary \ref{c:FromLipToReg} (recall that by
Corollary \ref{c:unbounded} it also holds for unbounded $K$, with
the same statement). To build the sets $A_{x,r}$ and the
maps $f_{x,r}$ required in its statement, we combine Theorem
\ref{thm:hah} and Proposition \ref{prop:localize}.

If $E$ is
bounded, for every $x\in E$ and $r=\mathrm{diam}(E)$, we can take
directly $A_{x,r}=A$ and $f=f_{x,r}$ given by Theorem
\ref{thm:hah}.

For $E$ \textcolor{black}{bounded or unbounded and}
$r\in(0,\diam(E))$, by Proposition \ref{prop:localize} there
exists a  set $E_{x,r}\subset E$  with $E_{x,r}\in \reg_1(\tilde
C_E)$ satisfying
$$B_r(x)\cap E\subset E_{x,r}\subset B_{3r}(x)\cap E$$
and with $\tilde C_E$ a constant   depending only $C_E$. We have
\begin{align*}
    &\int_{\mathcal F\cap (E_{x,r})^3} c^2(x_1,x_2,x_3)d \mathcal H^1(x_1)d\mathcal H^1(x_2)d\mathcal H^1(x_3)\\
    &\le \int_{\mathcal F\cap (B_{3r}(x))^3} c^2(x_1,x_2,x_3)d \mathcal H^1(x_1)d\mathcal H^1(x_2)d\mathcal H^1(x_3)\le 3Cr.
\end{align*}
Therefore we can apply Theorem \ref{thm:hah} to the metric space $(E_{x,r},d|_{E_{x,r}})$
to obtain a surjective Lipschitz map $\tilde f_{x,r}:\tilde A_{x,r}\to E_{x,r}$ for some $\tilde A_{x,r}\subset[0,1]$ and with
$${\rm Lip}(\tilde f_{x,r})\le D(3Cr+\diam(E_{x,r}))\le  D(3Cr+6r),$$
where $D$ is a constant depending only on $C_E$. Hence ${\rm
Lip}(\tilde f_{x,r})\le \tilde C r$, where $\tilde C$ is constant
depending only on $C$ and $C_E$. Since $B_r(x)\cap E\subset
E_{x,r}$ we immediately see that the set $A_{x,r}\coloneqq \tilde
f_{x,r}^{-1}(B_r(x)\cap E)$ and the map $f_{x,r}\coloneqq \tilde
f_{x,r}|_{A{x,r}}$ have the required properties and an application
of \textcolor{black}{Corollary \ref{c:unbounded}}
yields the result.
\end{proof}

Our initial motivation for Corollary \ref{c:IntMen} was to provide
details for the first step in the proof of one of the main results
in \cite{MR3678492} concerning singular integral operators in the
first \emph{Heisenberg group} $\mathbb{H}^1$.
\textcolor{black}{The group $\mathbb{H}^1=(\mathbb{R}^{3},\cdot)$
is defined by the product
\begin{displaymath}
(x,t) \cdot (x',t')= \Big(x_1+x_1',x_{2}+x_{2'},t+t'+
\omega(x,x')\Big),\quad (x,t),(x',t')\in \mathbb{R}^{2}\times
\mathbb{R},
\end{displaymath}
where
\begin{displaymath}
\omega(x,x'):= \tfrac{1}{2} [ x_1 x_{2}'-x_{2}x_1'],\quad x,x'\in
\mathbb{R}^{2}.
\end{displaymath}
The
left-invariant \emph{Kor\'{a}nyi metric} on $\mathbb{H}^1$ is
defined by
\begin{displaymath}
d_{\mathbb{H}^1}(p,p'):=\|p^{-1}\cdot
p'\|_{\mathbb{H}^1},\quad\text{where}\quad
\|(x,t)\|_{\mathbb{H}^1}:= \sqrt[4]{|x|^4 + 16
t^2},
\end{displaymath}
where $|\cdot|$ denotes the Euclidean norm on $\mathbb{R}^2$.}

The Theorem 1.3 in \cite{MR3678492} states the following:

\begin{thm}[Chousionis, Li]\label{t:CL} Let $K:\mathbb{H}^1 \setminus \{0\}\to
[0,\infty)$ be defined by
\begin{displaymath}
K(p)=\frac{\Omega(p)^2}{\|p\|_{\mathbb{H}^1}},\qquad
\text{where}\quad \Omega(x,y,t):=
\frac{|(x,y)|^{1/2}}{\|(x,y,t)\|_{\mathbb{H}^1}},
\end{displaymath}
and let $E\subset \mathbb{H}^1$ be a $1$-regular set. If the
truncated singular integrals \begin{displaymath}
T_{\varepsilon}f(p)=\int_{E\setminus
B(p,\varepsilon)}K(q^{-1}\cdot p)f(q)\,d\mathcal{H}^1(q)
\end{displaymath}
are uniformly bounded in $L^2(\mathcal{H}^1\lfloor E)$, then $E$
is contained in a $1$-regular curve.
\end{thm}
The proof of this result in \cite{MR3678492} starts with the
observation that it suffices to show for some $\alpha>0$ that
\begin{displaymath}
\int\int\int_{\Sigma(\alpha)\cap
B(p,R)^3}c^2(p_1,p_2,p_3)\,d\mathcal{H}^1(p_1)\,d\mathcal{H}^1(p_2)\,d\mathcal{H}^1(p_3)\lesssim
R,\quad p\in E,\,R>0,
\end{displaymath}
where
\begin{displaymath}
\Sigma(\alpha):= \bigcup_{r>0} \{(p_1,p_2,p_3)\in E\times E\times
E\colon \alpha r \leq d_{\mathbb{H}^1}(p_i,p_j)\leq r,\,i\neq j\}
\end{displaymath}
and the Menger curvature $c$ is computed with respect to $d_{\mathbb{H}^1}$. The
authors refer to \cite[p.123]{MR2297880}, however from this
statement it was  not immediately clear to us how to obtain the
\emph{$1$-regular curve} containing $E$, whose existence is
claimed in Theorem \ref{t:CL}. A good indication is provided by
\cite[Theorem 3.11]{MR2342818}, where Schul stated without proof a
version of Hahlomaa's result \cite{MR2297880} which he derived
from the arguments in \cite{MR2297880}. Corollary \ref{c:IntMen}
is a generalization of \cite[Theorem 3.11]{MR2342818}: it works
not only for geodesic, but for quasiconvex spaces, and also for
unbounded sets $E$. These relaxed assumptions are crucial for the
application in the proof of \cite[Theorem 1.3]{MR3678492}, where
 $\mathbb{H}^1$ is endowed with the (non-geodesic but
quasiconvex) Kor\'{a}nyi distance $d_{\mathbb{H}^1}$. In addition
to this greater flexibility, we believe that it is valuable to
have all the details for the proof of Corollary \ref{c:IntMen} (or
\cite[Theorem 3.11]{MR2342818}) available  as a
combination of published work by Hahlomaa and the arguments we
provide in this paper.

\section{Uniformly 1-rectifiable subsets of metric
spaces}\label{s:Char1UR}

The goal of this section is to characterize uniformly
$1$-rectifiable sets in terms of the $\nb$-numbers introduced in
Definition \ref{d:alpha}. What does uniform $1$-rectifiability
mean in this context? In Euclidean spaces, a $1$-regular set is
rightfully called \emph{uniformly $1$-rectifiable} if it is
contained in a $1$-regular curve since this property is equivalent
with many other notions of quantitative rectifiability that make
sense also for higher dimensional sets \cite{MR1113517,MR1251061}.
The main result of this section, Theorem \ref{t:Char1URII},
together with Corollary \ref{eq:equiv1UR}, motivates an analogous
definition of uniform $1$-rectifiability in a large class of
metric spaces. A combination of these results was stated as
Theorem \ref{t:Char1URIntro} in the
introduction.

Corollary \ref{eq:equiv1UR} is based on relations between regular
curves, Lipschitz images and the $\kappa$-numbers from Example
\ref{ex:MetricBetas}. While these connections are in essence due
to
 Hahlomaa \cite{MR2297880} in one direction and Schul
\cite{MR2337487} in the other direction, the novelty here is the
construction of a closed and connected $1$-regular \emph{global}
covering set based on the other equivalent characterizations. To
make this implication rigorous in arbitrary complete, doubling,
and quasiconvex metric spaces, we will need the results from
Section \ref{s:SuffRegCoverCurve} and in particular we will apply
Hahlomaa's result in the form of Corollary \ref{c:IntMen}. The
proof of Theorem \ref{t:Char1URII} takes up most space in this
section and at the core of it lies Theorem \ref{thm:l1 menger} by
which we can control the new $\nb$-numbers from above by the more
familiar $\kappa$-numbers.

\subsection{Characterization following Hahlomaa and Schul}\label{ss:HahlomaaSchul}

\begin{thm}[based on \cite{MR2297880,MR2337487}]\label{t:Char1URI} Let $(\X,\sfd)$ be a complete, doubling,
and quasiconvex metric space. Then a $1$-regular set $E\subset \X$
is contained in a $1$-regular closed and connected set $\Gamma_0$
if and only if for all $z\in E$ and $R>0$,
\begin{equation}\label{eq:triple}
\int\int\int_{(E\cap B_R(z))^3}
\frac{\partial(\{x_1,x_2,x_3\})}{\mathrm{diam}\{x_1,x_2,x_3\}^3}
\,d\mathcal{H}^1(x_1)d\mathcal{H}^1(x_2)d\mathcal{H}^1(x_3)\leq C
R
\end{equation} for some constant $C \geq 1$.

Moreover, if $E$ is bounded, then the statement also holds with
``closed and connected set'' replaced by ``curve'', and if
\eqref{eq:triple} holds, we can choose $\Gamma_0\in
\mathrm{Reg}_1(\widetilde{C})$ such that $\diam(\Gamma_0)\leq
\widetilde{C}\diam(E)$, where $\widetilde{C}$ depends only on $C$,
the $1$-regularity constant of $E$ and the doubling and
quasiconvexity constants of $(\X,\sfd)$. Conversely, if $E$ is
contained in a $1$-regular closed and connected set $\Gamma_0$,
then the constant $C$ in \eqref{eq:triple} can be bounded in terms
of the $1$-regularity constant of $\Gamma_0$ and the
quasiconvexity constant of $(\X,\sfd)$.
\end{thm}

\begin{proof} We begin by discussing the first part of the
theorem. Assume that \eqref{eq:triple} holds for a $1$-regular set
$E \subset \X$ with regularity constant $C_E$. The goal is to
apply Corollary \ref{c:IntMen} to show that $E$ is contained in a
closed and connected $1$-regular set $\Gamma_0$. To verify the
assumptions for $E$, let $K_0>1$ be the universal constant from
Corollary \ref{c:IntMen} and denote, as before,
\begin{displaymath}
\mathcal{F}=\{(x_1,x_2,x_3) \in E\times E \times E \ : \
d(x_i,x_j)\le K_0 d(x_k,x_l), \, \text{ for all
}i,j,k,l\in\{1,2,3\}, \, k\neq l \}.
\end{displaymath}
By making the domain of integration for the triple integral in
\eqref{eq:triple} possibly smaller, we deduce from the assumption
that
\begin{equation}\label{eq:***_F}
\int\int\int_{\mathcal{F}\cap (B_R(z))^3}
\frac{\partial(\{x_1,x_2,x_3\})}{\mathrm{diam}\{x_1,x_2,x_3\}^3}
\,d\mathcal{H}^1(x_1)d\mathcal{H}^1(x_2)d\mathcal{H}^1(x_3)\leq C
R,\quad z\in E,\,R>0.
\end{equation}
By the equivalence between triangular excess and Menger curvature
stated in \eqref{eq:ExcessMenger} for triples in $\mathcal{F}$,
\eqref{eq:***_F} yields
\begin{equation*}
\int\int\int_{\mathcal{F}\cap (B_R(z))^3} c^2(x_1,x_2,x_3)
\,d\mathcal{H}^1(x_1)d\mathcal{H}^1(x_2)d\mathcal{H}^1(x_3)\lesssim_{K_0}
C R,\quad z\in E,\,R>0.
\end{equation*}
Then Corollary \ref{c:IntMen} shows that there exists a closed and
connected $1$-regular subset $\Gamma_0\in
\mathrm{Reg}_1(\widetilde{C})$ of $(\X,\sfd)$ containing $E$ for
$\widetilde{C}$ as in the statement of the theorem. If $E$ is
bounded, then
 Corollary \ref{c:IntMen}
yields further that  $\mathcal{H}^1(\Gamma_0)\leq
\widetilde{C}\mathrm{diam}(E)<\infty$. Since $(\X,\sfd)$ is
complete, such $\Gamma_0$ is automatically compact, and posteriori
the trace of a $1$-regular Lipschitz curve, recall the discussion
at the beginning of Section \ref{s:ConstrCover}.

\medskip

The other implication was known before. Indeed, assume that $E$ is
contained in a closed and connected $1$-regular subset $\Gamma_0$.
Then \cite[Theorem 1.10]{MR2337487} (see also \cite[Thm 3.12 and
(3.9)]{MR2342818}) shows that \eqref{eq:triple} holds with ``$E$''
replaced by ``$\Gamma_0$''. (The proof argument in  \cite[Section
3]{MR2337487} to deduce Theorem 1.10 from Theorem 1.8 therein
seems to be formulated under the implicit assumption that the
ambient space is geodesic, but the argument works analogously for
quasiconvex spaces, with the implicit constant in
\eqref{eq:triple} depending on the quasiconvexity constant of
$(\X,\sfd)$ in addition to the $1$-regularity constant of
$\Gamma_0$.)

As inequality \eqref{eq:triple}  remains true if the domain of
integration is replaced by a smaller set, \eqref{eq:triple} also
holds for the ($\mathcal{H}^1$-measurable) $E\subset \Gamma_0$.

\end{proof}

Theorem \ref{t:Char1URI}  gives a characterization for a
$1$-regular set $E$ to be contained in a closed and connected
$1$-regular curve. The next result, Corollary \ref{eq:equiv1UR},
provides further justification for calling such sets
\emph{uniformly $1$-rectifiable}. We recall the relevant
terminology.

\begin{definition}\label{d:BPLI} Let $k\in \mathbb{N}$.
A $k$-regular set $E$ in a metric space $(\X,\sfd)$ has \emph{big
pieces of Lipschitz images (BPLI)} if there exist constants
$c,L>0$ such that for every $x\in E$, $0<r<\mathrm{diam}(E)$,
there exists an $L$-Lipschitz function $f:A \to E$, where $A$ is a
subset of the Euclidean ball $B_r(0)\subset \mathbb{R}^k$, and
$\mathcal{H}^k(f(A)\cap B_r(x))\geq c r^k$.
\end{definition}

\begin{definition}\label{d:BPBI}
 Let $k\in \mathbb{N}$.
A $k$-regular set $E$ in a metric space $(\X,\sfd)$ has \emph{big
pieces of bi-Lipschitz images (BPBI)} if there exist constants
$c,L>0$ such that for every $x\in E$, $0<r<\mathrm{diam}(E)$,
there exists an $L$-bi-Lipschitz \textcolor{black}{embedding} $f:A
\to E$, where $A$ is a subset of the Euclidean ball $B_r(0)\subset
\mathbb{R}^k$, and $\mathcal{H}^k(f(A)\cap B_r(x))\geq c r^k$.
\end{definition}

\begin{cor}\label{eq:equiv1UR} Let $(\X,\sfd)$ be a complete, doubling,
quasiconvex metric space and let $E\subset \X$ be $1$-regular. The following conditions are
equivalent:
\begin{enumerate}
\item\label{i:UR} $E$ is contained in a closed and connected
$1$-regular set $\Gamma_0$,
\item\label{ii:UR} $E$ has BPLI,
\item\label{iii:UR} $E$ has BPBI,
\item\label{v:UR} $E\in \mathrm{GLem}(\kappa,1)$, for $\kappa$ as in Example \ref{ex:MetricBetas}.
\end{enumerate}
Moreover, if $E$ is bounded, these conditions are further
equivalent to $E$ being contained in a $1$-regular \emph{curve}.
Also, if \eqref{i:UR} holds, then \eqref{ii:UR} holds with
BPLI constants depending only on the $1$-regularity constants
of $E$ and $\Gamma_0$, as well as the doubling and quasiconvexity
constants of $(\X,\sfd)$. Conversely, if \eqref{v:UR} holds for
$E\in \mathrm{GLem}(\kappa,1,M)\cap \mathrm{Reg}_1(C)$, then
\eqref{i:UR} holds with $\Gamma_0 \in
\mathrm{Reg}_1(\widetilde{C})$, where $\widetilde{C}$ depends only
on $M$, $C$, and the doubling and quasiconvexity constants of
$(\X,\sfd)$.
\end{cor}

We call a $1$-regular set $E$ \emph{uniformly $1$-rectifiable} if
it satisfies one (and thus all) of the properties
\eqref{i:UR}-\eqref{v:UR}. The equivalence of
\eqref{ii:UR}--\eqref{iii:UR} was established by Schul in
\cite[Corollary 1.2]{MR2554164},  while \cite[Corollary
1.3]{2007arXiv0706.2517S} states a third equivalent condition that
is very similar
 to \eqref{v:UR},
but formulated in terms of \emph{multiresolution families} instead
of dyadic systems as we used in the definition of the geometric
lemma. However, the equivalence of these two versions of the
Carleson condition is standard, see the beginning of the proof of
\cite[Theorem 1.1]{MR2337487}. \textcolor{black}{For the
convenience of the reader, we show how the implication
\eqref{iii:UR} to \eqref{v:UR} can be deduced from published
results. This is the content of Proposition \ref{p:BPBItoGLem}.}

\begin{proposition}\label{p:BPBItoGLem}
    \textcolor{black}{Let $(\X,\sfd)$ be a
    complete, doubling, quasiconvex metric space and let $E\subset \X$ be $1$-regular. If $E$ has BPBI, then $E\in \mathrm{GLem}(\kappa,1)$, for $\kappa$ as in Example \ref{ex:MetricBetas}.}
\end{proposition}

\begin{proof}
\textcolor{black}{The proof consists of three steps.}
\textcolor{black}{First we  show that connected $1$-regular sets
in metric spaces satisfy $\mathrm{GLem}(\kappa,1)$. This is
essentially the content of  \cite[Theorem 1.11]{MR2337487}, but
the latter is formulated in terms of \emph{multiresolution
families} instead of dyadic systems. Indeed, \cite[Theorem
1.11]{MR2337487} states that for every connected set $\Gamma\in
\mathrm{Reg}_1(C)$ in a metric space and any associated
multiresolution family, for $z\in \Gamma$ and $R>0$,
\begin{equation}\label{eq:GLemSchul}
\sum_{\substack{{B\in \mathcal{G}^{\Gamma}}\\{B\subset B_R(z)}}}\int_B\int_B\int_B \partial(\{x_1,x_2,x_3\})\mathrm{rad}(B)^{-3}\,d\mathcal{H}^1|_{\Gamma}(x_1)d\mathcal{H}^1|_{\Gamma}(x_2)d\mathcal{H}^1|_{\Gamma}(x_3)\lesssim R,
\end{equation}
where the implicit constant depends on $C$ and the constant ``$A$'' in the definition of the multiresolution family $\mathcal{G}^{\Gamma}$. Recall that a multiresolution family is given by
\begin{displaymath}
    \mathcal{G}^\Gamma \coloneqq \{B_{A 2^{-n}}(x):\, x\in X_n^{\Gamma},\,n\in \mathbb{N} \},
\end{displaymath}
where  $A>1$ is a chosen constant and $X_n^{\Gamma}$ is any  $2^{-n}$-net for $\Gamma,$ i.e.\ a set of points in $\Gamma$ with the properties that $\sfd(x,y)>2^{-n}$ for all $x,y\in X_n^{\Gamma}$ ($2^{-n}$-separation) and for every $z\in \Gamma$ there exists $x\in X_n^{\Gamma}$ such that $\sfd(x,z)\leq 2^{-n}$ (maximality)}.

\textcolor{black}{In order to deduce that $\Gamma\in
\mathrm{GLem}(\kappa,1)$ in the sense of Definition \ref{d:GL}, we
take an arbitrary dyadic system $\Delta$ on $\Gamma$ and fix a
cube $Q_0 \in \Delta$ with $\ell(Q_0)=2^{-j_0}$. We need to bound
from above the expression
\begin{align}\label{eq:ToBoundGLemCube}
   & \sum_{Q\in \Delta_{Q_0}}\kappa(2Q)\mathcal{H}^1(Q)\\&=
    \sum_{j\geq j_0}\sum_{Q\in \Delta_{Q_0}\cap \Delta_j}\frac{\mathcal{H}^1(Q)}{\mathcal{H}^1(2Q)^3}\int_{2Q}\int_{2Q}\int_{2Q}\frac{\partial(\{x_1,x_2,x_3\})}{\mathrm{diam}(2Q)}\,d\mathcal{H}|_{\Gamma}^1
    (x_1)d\mathcal{H}|_{\Gamma}^1
    (x_2)d\mathcal{H}|_{\Gamma}^1
    (x_3).\notag
\end{align}
Now for a fixed $j\geq j_0$, the collection of ``centers'' $\{x_Q:\, Q\in \Delta_j\}$ (as in item (5) in Definition \ref{dl:dyadic}) is $2^{-n_j}$-separated, where $n_j$ is the smallest natural number such that $2^{-n_j}<c_0 2^{-j}$ with the constant $c_0\in (0,1)$ from Definition \ref{dl:dyadic}. The collection is not necessarily \emph{maximal}, but by adding points if necessary, one can enlarge it to a family $X_{n_j}^{\Gamma}$ as in the definition of multiresolution families. By construction, every $n\in \mathbb{N}$ appears as $n_j$ for at most one index $j$. Moreover, by \eqref{eq:KQ_in_Ball} and \eqref{eq:MeasCube}, we can choose a constant $A>1$, depending only on the $1$-regularity constant $C$, such that
\begin{equation}\label{eq:CubeBallIncl}
    2Q\subset B_{A 2^{-n_j}}(x_Q)\cap \Gamma,\quad Q\in \Delta_j.
\end{equation}
Associated to the given dyadic system $\Delta$, we fix now a multiresolution family $\mathcal{G}^{\Gamma}$ with the chosen constant $A$ and such that  $\{x_Q:\, Q\in \Delta_j\}\subset X_{n_j}^{\Gamma}$. Moreover, there exists a constant $K$, depending only on  $C$, such that
\begin{equation}\label{eq:BallCubeIncl}
B_{A2^{-n_j}}(x_Q)\cap \Gamma \subset B_{K\mathrm{diam}(Q_0)}(x_{Q_0})\cap \Gamma,\quad Q\in \Delta_{Q_0}\cap \Delta_j.
\end{equation}
Indeed, for $z\in B_{A2^{-n_j}}(x_Q)\cap \Gamma$, we have
\begin{displaymath}
\mathrm{dist}(z,Q_0)\leq \sfd(z,x_Q)<A2^{-n_j}<A c_0 2^{-j}\overset{\eqref{eq:MeasCube}}{\lesssim}_C \mathrm{diam}(Q_0),
\end{displaymath}
which proves \eqref{eq:BallCubeIncl} for a suitable constant $K$.
Then, by the construction of $\mathcal{G}^{\Gamma}$, \eqref{eq:CubeBallIncl}, \eqref{eq:BallCubeIncl}, and the property \eqref{eq:MeasCube} of dyadic cubes, we can bound the expression in \eqref{eq:ToBoundGLemCube} from above as follows
\begin{align}\label{eq:ToBoundGLemCube2}
 &
    \sum_{j\geq j_0}\sum_{Q\in \Delta_{Q_0}\cap \Delta_j}\frac{\mathcal{H}^1(Q)}{\mathcal{H}^1(2Q)^3}\int_{2Q}\int_{2Q}\int_{2Q}\frac{\partial(\{x_1,x_2,x_3\})}{\mathrm{diam}(2Q)}\,d\mathcal{H}|_{\Gamma}^1
    (x_1)d\mathcal{H}|_{\Gamma}^1
    (x_2)d\mathcal{H}|_{\Gamma}^1
    (x_3).\notag\\
    &\lesssim_C \sum_{\substack{B\in \mathcal{G}^{\Gamma}\\ B\subset B_{K\mathrm{diam}(Q_0)}(x_{Q_0})}}\int_B\int_B\int_B \partial(\{x_1,x_2,x_3\})\mathrm{rad}(B)^{-3}\,d\mathcal{H}^1|_{\Gamma}(x_1)d\mathcal{H}^1|_{\Gamma}(x_2)d\mathcal{H}^1|_{\Gamma}(x_3).
\end{align}
Combined with Schul's result \eqref{eq:GLemSchul}, the two inequalities \eqref{eq:ToBoundGLemCube} and \eqref{eq:ToBoundGLemCube2} yield
\begin{displaymath}
    \sum_{Q\in \Delta_{Q_0}}\kappa(2Q)\mathcal{H}^1(Q)\lesssim_C \mathcal{H}^1(Q_0),
\end{displaymath}
and hence, as $Q_0\in \Delta$ was arbitrary, $\Gamma\in \mathrm{GLem}(\kappa,1,M)$, with $M$ depending only on $C$.
}

\medskip

\textcolor{black}{In the second step of the proof of Proposition
\ref{p:BPBItoGLem}, we observe that the  set $E$ given in the
statement has big pieces of connected $1$-regular sets since it
has BPBI (with some constants $c$ and $L$). Indeed, for $x\in E$
and $0<r<\mathrm{diam}(E)$, let $f(A)$ be the associated
bi-Lipschitz piece as in the definition of BPBI. Then, for
arbitrary $y\in f(A)$ and $s>0$, the set $f^{-1}(f(A)\cap B_s(y))$
is contained in an interval of length $2Ls$ by the
$L$-bi-Lipschitz property of $f$. Upon translating and rescaling
by $2Ls$, we thus find a set $A_{y,s}\subset [0,1]$ and a
$2L^2s$-Lipschitz function from $A_{y,s}$ onto $f(A)\cap B_s(y)$.
Since $y$ and $s$ were arbitrary,  Corollary \ref{c:FromLipToReg}
then yields a connected set $\Gamma_0\in \mathrm{Reg}_1(C_0)$ with
$C_0$ depending only on doubling and quasiconvexity constants of
$(X,\sfd)$ and on $L$ such that $\Gamma_0\supseteq f(A)$. Since
$f(A)$ was chosen as in the definition of BPBI, we have in
particular $\mathcal{H}^1(E\cap \Gamma_0\cap B_r(x))\geq cr$.
Repeating the same argument for all $x$ and $r$, we find that $E$
has big pieces (as defined in \cite[Definition 2.11]{MR4485846})
of connected $1$-regular sets.}

\medskip
\textcolor{black}{Finally, by the first two steps of the proof, we
know that $E$ has big pieces of sets that satisfy
$\mathrm{GLem}(\kappa,1)$ with uniform constants. Then it follows
from an abstract argument that $E$ itself satisfies
$\mathrm{GLem}(\kappa,1)$. This abstract argument is known as
stability of geometric lemmas under the ``big pieces functor'' and
was formulated in great generality in \cite[Proposition
2.23]{MR4485846}, which we can apply to conclude the proof.}
\end{proof}


Our main contribution to Corollary \ref{eq:equiv1UR} is the proof
of ``\eqref{v:UR} $\Rightarrow$ \eqref{i:UR}'', and it is
precisely this implication which arises as a corollary of Theorem
\ref{t:Char1URI}. To complete the circle of equivalent statements,
we also discuss the implication ``\eqref{i:UR} $\Rightarrow$
\eqref{ii:UR}''.

\begin{proof}[Proof of Corollary \ref{eq:equiv1UR}] We assume first that \eqref{i:UR} holds, that is,
$E$ is a $1$-regular set contained in a closed and connected
$1$-regular set $\Gamma_0$, and we will deduce \eqref{ii:UR}. We
call the \emph{data} of $(E,\Gamma_0,\X)$ the collection of the
$1$-regularity constants of $E$ and $\Gamma_0$, as well as the
doubling and quasicovexity constants of $(\X,\sfd)$.
 We could
essentially use $\Gamma_0$ to construct a big Lipschitz image in
$E\cap B_r(x)$, for an arbitrarily given point $x\in E$ and
$0<r<\mathrm{diam}(E)$, but the localization argument will be
simpler if we use curves given by Theorem  \ref{t:Char1URI}. Let
$C\geq 1$ be a large enough constant, to be chosen momentarily.
Using the localization property for $s$-regular sets, stated in
Proposition \ref{prop:localize}, we can find a $1$-regular set
$E_{x,r}$, with regularity constant depending only the regularity
constant of $E$, such that
\begin{displaymath}
B_{r/3C}(x)\cap E\subset E_{x,r}\subset B_{r/C}(x)\cap E.
\end{displaymath}
As a subset of $E$, the set $E_{x,r}$ is still covered by
$\Gamma_0$. By \cite[Theorem 1.10]{MR2337487}, the connected
$1$-regular set $\Gamma_0$ satisfies
\begin{equation*}
\int\int\int_{(\Gamma_0\cap B_R(z))^3}
\frac{\partial(\{x_1,x_2,x_3\})}{\mathrm{diam}\{x_1,x_2,x_3\}^3}
\,d\mathcal{H}^1(x_1)d\mathcal{H}^1(x_2)d\mathcal{H}^1(x_3)\leq
C_0 R,
\end{equation*}
where $C_0$ depends on the Ahlfors regularity constant of
$\Gamma_0$. Since $E_{x,r}\subset \Gamma_0$, it follows that
$E_{x,r}$ satisfies condition \eqref{eq:triple} in Theorem
\ref{t:Char1URI}, with the same constant ``$C_0$''. Hence, by
Theorem \ref{t:Char1URI}, there exists a $1$-regular curve
$\Gamma_{x,r}\supset E_{x,r}$ with $\Gamma_{x,r}\in
\mathrm{Reg}_1(\widetilde{C})$ and
$\mathrm{diam}(\Gamma_{x,r})\leq
\widetilde{C}\mathrm{diam}(E_{x,r})(\leq 2\widetilde{C}r/C)$,
where $\widetilde{C}$ depends only on the data of
$(E,\Gamma_0,\X)$. In particular, by choosing $C$ large enough
depending only on the data of $(E,\Gamma_0,\X)$, we may assume
that $\Gamma_{x,r} \subset B_r(x)$.

Now \cite[Lemma 2.8]{2021arXiv210906753B}
and a straightforward reparametrization show that there exists an
$L$-Lipschitz function $\gamma:[-r,r] \to \X$ with
$\gamma([-r,r])=\Gamma_{x,r}$ and $L$ bounded in terms of the data
of $(E,\Gamma_0,\X)$. Moreover,
\begin{displaymath}
\mathcal{H}^1(\gamma([-r,r])\cap E\cap B_r(x))\geq
\mathcal{H}^1(E_{x,r})\geq \mathcal{H}^1(E\cap B_{r/3C}(x))\sim_C
r.
\end{displaymath}
Repeating the same argument for every  $x\in E$ and
$0<r<\mathrm{diam}(E)$ proves that $E$ satisfies \eqref{ii:UR},
that is, it has BPLI (with constants depending only on the data of
$(E,\Gamma_0,\X)$). \textcolor{black}{Then $E$ has also BPBI
(condition \eqref{iii:UR}) by \cite[Corollary 1.2]{MR2554164} and
finally, satisfies the geometric lemma in condition \eqref{v:UR}
by Proposition \ref{p:BPBItoGLem}.}


In the converse direction, assume now that  \eqref{v:UR} holds for
a set $E\in \mathrm{Reg}_1(C)$ in $\X$. Fix a dyadic system
$\Delta$ on $E$ and consider arbitrary $z\in E$ and
$0<R<\mathrm{diam}(E)/2$. Then there exists $j_0\in \mathbb{J}$ such
that $B_R(z)\cap E \subset \bigcup_{i=1}^m Q_{0,i}$ with
$Q_{0,i}\in \Delta_{j_0}$, $2^{-j_0-1}\leq R <2^{-j_0}$ and $m$
depending only on $C$. To see this, recall that $B_R(z)\cap E$ is
covered by the union of the cubes $Q\in \Delta_{j_0}$. Now if
$Q\cap B_R(z)\neq \emptyset$, then $Q\subset
B_{\widetilde{C}R}(z)\cap E$ for a constant
$\widetilde{C}=\widetilde{C}(C)$. Since $E\in
\mathrm{Reg}_1(C)$, the elements in $\Delta_{j_0}$ are disjoint
and thanks to the lower bound for the $\mathcal{H}^1$-measure of $Q\in
\Delta_{j_0}$ stated in \eqref{eq:MeasCube}  we have that at most
$m\lesssim_C 1$ elements $Q\in \Delta_{j_0}$ can intersect
$B_R(z)$.

Finally, we will show that there exists a constant $K=K(C)$ such
that, for every $Q_{0,i}\in \Delta_{j_0}$,
\begin{equation}\label{eq:IntSum}
\int\int\int_{[E\cap
B(z,R)]^3}\frac{\partial(\{x_1,x_2,x_3\})}{(\mathrm{diam}\{x_1,x_2,x_3\})^3}\,d\mathcal{H}^1(x_1)d\mathcal{H}^1(x_2)d\mathcal{H}^1(x_3)
\lesssim_C \sum_{i=1}^m\sum_{Q\in \Delta_{Q_{0,i}}} \kappa(KQ)
\mathcal{H}^1(Q).
\end{equation}
This will allow us to verify the assumption of Theorem
\ref{t:Char1URI} via \eqref{v:UR} ($E\in \mathrm{GLem}(\kappa,1)$).
 Then we deduce
that $E$ is contained in a closed and connected $1$-regular  set
$\Gamma_0$, thus \eqref{i:UR} holds.

To conclude the proof we verify  inequality
\eqref{eq:IntSum}. We first decompose the domain of integration as
$[E\cap B_R(z)]^3=\bigcup_{j\geq j_0} A_j$, where
\begin{equation}\label{eq:A_j}
A_j:= \{(x_1,x_2,x_3)\in [E\cap B_R(z)]^3:\, 2^{-j}\leq
\mathrm{diam}\{x_1,x_2,x_3\}< 2^{-j+1}\}.
\end{equation}
It suffices to consider $j\geq j_0$ since $A_j\neq \emptyset$ implies that
\begin{equation}\label{eq:Bound_j}
2^{-j}\leq 2R < 2^{-j_0+1}.
\end{equation}
Thus, if $(x_1,x_2,x_3)\in A_j$, then $\{x_1,x_2,x_3\}\subset
B_{r}(x_3)$ for some $j\geq j_0$ with $2^{-j}\leq r<2^{-j+1}$. Recalling the basic property of dyadic cubes stated in
\eqref{eq:ball_enlarged_cube}, there exists a constant
$K=K(1,C)>1$, depending only on the Ahlfors regularity constant of
$E\in \mathrm{Reg}_1(C)$ such that there is $Q\in \Delta_j$ with
$
\{x_1,x_2,x_3\}\subset B(x_3,r) \subset KQ
$
and $Q\subset Q_{0,i}$ for some $i\in \{1,\ldots,m\}$. Thus
\begin{equation}\label{eq:A_j decomp}
    A_j \subset \bigcup_{i=1}
^m \bigcup_{Q\in \Delta_{Q_{0,i}}\cap \Delta_j} (KQ)^3.\end{equation}
Therefore,
\begin{align*}
&\int\int\int_{[E\cap
B_R(z)]^3}\frac{\partial(\{x_1,x_2,x_3\})}{(\mathrm{diam}\{x_1,x_2,x_3\})^3}\,d\mathcal{H}^1(x_1)d\mathcal{H}^1(x_2)d\mathcal{H}^1(x_3)
\\&\overset{\eqref{eq:A_j decomp}, \eqref{eq:MeasCube}}{\lesssim_{C}} \sum_{j\geq j_0}  \sum_{i=1}^m\sum_{Q\in
\Delta_{Q_{0,i}}\cap \Delta_j}
\int\int\int_{[KQ]^3}\frac{\partial(\{x_1,x_2,x_3\})}{\mathrm{diam}(KQ)^3}\,d\mathcal{H}^1(x_1)d\mathcal{H}^1(x_2)d\mathcal{H}^1(x_3)\\
&\overset{\eqref{eq:MeasCube}}{\lesssim_{C}}
\sum_{i=1}^m \sum_{Q\in \Delta_{0,i}} \kappa(KQ) \mathcal{H}^1(Q).
\end{align*}
By the assumption that $E\in \mathrm{GLem}(\kappa,1)$, and recalling Lemma \ref{l:coeff} (see also Remark \ref{rmk:coeff}), this shows as desired that \eqref{eq:triple} holds for all $z\in E$ with a constant depending only on $M$, $C$, and the doubling and quasiconvexity constants of $(X,\sfd)$, at least for $R<\mathrm{diam}(E)/2$. If $E$ is unbounded, we are done. Otherwise, if $E$ is bounded and $R\geq \mathrm{diam}(E)/2$, then we apply the preceding argument with $j_0=n$, where $n$ is the smallest integer in $\mathbb{J}$, and the arguments go through verbatim if we replace in \eqref{eq:Bound_j} the bound ``$2R$'' by ``$\mathrm{diam}(E)$''. In any case, the assumption of Theorem \ref{t:Char1URI} is satisfied for $E$, and the first part of the corollary follows.

The part concerning the covering of bounded sets $E$ by $1$-regular curves is also a consequence of Theorem \ref{t:Char1URI}.
\end{proof}

\subsection{Characterization using $\nb_{1,1}$-numbers}\label{ss:Char_iota11}
In this section we complement Corollary \ref{eq:equiv1UR} by providing a further equivalent characterization for uniform $1$-rectifiability, now in terms of the $\nb$-numbers from Definition \ref{d:alpha}.

\begin{thm}\label{t:Char1URII} Let $(\X,d)$ be a  metric space and $E\in \reg_1(c_E)$.
Then the following are equivalent:
\begin{enumerate}
\item\label{i:SGLmetric}  $E\in \mathrm{GLem}(\kappa,1)$,
\item\label{ii:SGLmetric}  $E\in \mathrm{GLem}(\nb_{1,1},1)$.
\end{enumerate}
In fact, if $\triangle$ is a system of
Christ-David dyadic cubes on the $1$-regular set $E$, then
\begin{equation}\label{eq:equivalent alk}
    3^{-1}\kappa(2Q) \le \nb_{1,1}(2Q)\le C \kappa(7Q), \quad Q \in \Delta,
\end{equation}
where $C\ge 1$ is a constant depending only on $c_E.$

Moreover, if $(\X,d)$ is complete, quasiconvex, and doubling, and if
one (and thus both) of  conditions \eqref{i:SGLmetric} and
\eqref{ii:SGLmetric} hold, then $E$ is uniformly $1$-rectifiable.
\end{thm}

For the implication $\eqref{ii:SGLmetric}\implies \eqref{i:SGLmetric}$ in Theorem \ref{t:Char1URII}
 we will need the following result, the proof of which is postponed to the next subsection.

\begin{thm}[$L^1$-quantified  Menger theorem]\label{thm:l1 menger}
    Let $(\X,\d)$ be a bounded and 1-regular metric space. Then there exists a Borel map $f: \X \to \mathbb{R}$ such that
    \begin{equation}\label{eq:l1 menger}
        \fint_{\X}\fint_{\X} \frac{||f(x)-f(y)|-\sfd(x,y)|}{\diam(\X)}\d\mu(x)\d \mu(y) \le C\fint_{\X} \fint_{\X} \fint_{\X}\frac{\partial(\{x,y,z\})}{\diam(\X)}\, \d \mu(x) \d \mu(y) \d \mu(z),
    \end{equation}
    where $\mu\coloneqq \mathcal H^1$ and $C$ is a constant  depending only on the regularity constant of $\X$.
\end{thm}

\begin{proof}[Proof of Theorem \ref{t:Char1URII}]
Once the equivalence of \eqref{i:SGLmetric} and
\eqref{ii:SGLmetric} is established for a set $E$ in a complete, quasiconvex and doubling metric space, it follows from Corollary
\ref{eq:equiv1UR} that $E$ with these properties is
uniformly $1$-rectifiable.

Thus we concentrate on the equivalence of \eqref{i:SGLmetric} and
\eqref{ii:SGLmetric}, for which is enough to show \eqref{eq:equivalent alk}. Let $\triangle$ be a system of Christ-David
cubes on $E$, and fix $Q \in \triangle$.
 We
first prove the following inequality:
\begin{equation}\label{eq:metric_lesssim_alpha}
\kappa(2Q)\le3
\nb_{1,1}(2 Q).
\end{equation}
To this end, let $x_1,x_2,x_3\in 2Q$ be arbitrary, and consider any function $f:2Q \to \mathbb{R}$ (which we later take to be Borel)
and any norm $\|\cdot\|$ on $\mathbb{R}$. Then there
exists $\sigma \in S_3$ (depending on $x_1,x_2,x_3$) such that
\begin{displaymath}
f(x_{\sigma(1)}) \leq f(x_{\sigma(2)}) \leq f(x_{\sigma(3)})
\end{displaymath}
and hence
$\partial_1(f(x_{\sigma(1)}),f(x_{\sigma(2)}),f(x_{\sigma(3)}))=0$,
where $\partial_1(\cdot)$ is computed with respect to $\|\cdot\|$, which
is a constant multiple of $|\cdot|$ (see \eqref{eq:def excess} for the expressions of $\partial_1,\partial$). Therefore,
\begin{align*}
\frac{\partial(\{x_1,x_2,x_3\})}{\diam(2Q)} \leq &
\frac{\partial_1(x_{\sigma(1)},x_{\sigma(2)},x_{\sigma(3)})}{\diam(2Q)}\\=&
\frac{\partial_1(x_{\sigma(1)},x_{\sigma(2)},x_{\sigma(3)})-\partial_1(f(x_{\sigma(1)}),f(x_{\sigma(2)}),f(x_{\sigma(3)}))}{\diam(2Q)}\\
 \leq &
\tfrac{\left|d(x_{\sigma(1)},x_{\sigma(2)})-\|f(x_{\sigma(1)})-f(x_{\sigma(2)})\|\right|}{\diam(2Q)}
+
\tfrac{\left|d(x_{\sigma(2)},x_{\sigma(3)})-\|f(x_{\sigma(2)})-f(x_{\sigma(3)})\|\right|}{\diam(2Q)}\\&+
\tfrac{\left|d(x_{\sigma(1)},x_{\sigma(3)})-\|f(x_{\sigma(1)})-f(x_{\sigma(3)})\|\right|}{\diam(2Q)}.
\end{align*}
Note that the last expression is unchanged if we replace each $``\sigma(i)"$ by $``i"$. Hence
integrating and taking the infimum over all Borel functions $f:2Q\to \mathbb{R}$ and $\|\cdot\|$, using the definition of $\kappa(\cdot)$ and $\nb_{1,1}(\cdot)$, proves the
inequality \eqref{eq:metric_lesssim_alpha}. In particular,
\eqref{ii:SGLmetric} implies \eqref{i:SGLmetric}.

Next we prove the following opposite inequality
\begin{equation}\label{eq:converse alk}
    \nb_{1,1}(2Q)\le C\kappa(7 Q),
\end{equation}
where $C>0$ is a constant depending only on $c_E$. Fix $Q \in \triangle$, $z\in Q$, and set $d(Q)\coloneqq \diam(Q)$.
Applying Proposition
\ref{prop:localize} we can  find  a 1-regular set $E_Q\subset E$
such that $B_{2d(Q)}(z)\cap E\subset E_Q\subset B_{6d(Q)}(z)\cap E$  and
with a regularity constant depending only on $c_E$ (if $2d(Q)\ge\mathrm{diam}(E)$ we simply take $E_Q=E$). In particular
\[
2Q\subset E_Q\subset 7Q
\]
Then applying Theorem \ref{thm:l1 menger} to the  metric  space $(E_Q,\sfd|_{E_Q})$ we obtain a Borel map $f: E_Q \to \rr$ satisfying
\begin{equation}\label{eq:preliminary EQ inequality}
\begin{split}
    \fint \fint_{(E_Q)^2}||f(x)-f(y)|&-\sfd(x,y)|\d\mathcal H^1(x)\d\mathcal H^1(y)  \\
    &\le \tilde C\fint \fint \fint_{(E_Q)^3}\partial\{x,y,z\}\,\d\mathcal H^1(x)\d\mathcal H^1(y)\d\mathcal H^1(z),
\end{split}
\end{equation}
where $\tilde C$ is a constant depending only on $c_E.$
Moreover by the 1-regularity of $E$ and by the property of dyadic systems stated in \eqref{eq:MeasCube}, we have
\[
c_E^{-1}c_0d(Q)\le \mathcal H^1(Q)\le \mathcal H^1(2Q)\le \mathcal H^1(E_Q)\le  \mathcal H^1(7Q)\le 7c_E d(Q),
\]
where $c_0$ is the constant, depending only on $c_E,$ appearing in the definition of dyadic system.
  Therefore  using \eqref{eq:preliminary EQ inequality}
\[
\begin{split}
    \frac{1}{\mathcal H^1(2Q)^2}\int \int_{(2Q)^2}&\frac{||f(x)-f(y)|-\sfd(x,y)|}{\diam (2Q)}\d\mathcal H^1(x)\d\mathcal H^1(y)
    \\
    &\le  c_0^{-2}7^2c_E^4\fint \fint_{(E_Q)^2}\frac{||f(x)-f(y)|-\sfd(x,y)|}{\diam (2Q)}\}\d\mathcal H^1(x)\d\mathcal H^1(y) \\
    &\le \tilde C\cdot c_0^{-2}7^2c_E^4 \fint \fint \fint_{(E_Q)^3}\frac{\partial\{x,y,z\}}{\diam (2Q)}\,\d\mathcal H^1(x)\d\mathcal H^1(y)\d\mathcal H^1(z),\\
    &\le C \fint \fint \fint_{(7 Q)^3}\frac{\partial\{x,y,z\}}{\diam(7Q)}\,\d\mathcal H^1(x)\d\mathcal H^1(y)\d\mathcal H^1(z),
\end{split}
\]
where $C$ is a constant depending only on $c_E$. This proves \eqref{eq:converse alk}, which combined with Lemma \ref{l:coeff} (whose assumption are satisfied for the coefficients $\kappa$ by Remark \ref{rmk:coeff}) gives also the implication  \eqref{i:SGLmetric} $\implies$ \eqref{ii:SGLmetric}.
\end{proof}

\subsection{Constructing good maps into
$\mathbb{R}$}\label{ss:ConstrGoodMapsIntoR}

The main goal of this subsection is to prove Theorem \ref{thm:l1
menger}. This requires us to construct maps $f:\X \to \mathbb{R}$
with good properties using  a suitable control for the
triangular excess of point triples in $\X$. We introduce some
notation to make this precise. We say that
a map between two metric spaces $f: (\X_1,\sfd_1)\to
(\X_2,\sfd_2)$ is a \emph{$\delta$-isometry}, for some $\delta\ge
0$, if
\[
|\sfd_2(f(x),f(y))-\sfd_1(x,y)|\le \delta, \quad  x,y \in \X.
\]
For every $S\subset (\X,\sfd)$, we define
\[
\partial S\coloneqq \sup_{\{x,y,z\}\subset  S} \partial(\{x,y,z\}).
\]
Given three points $x,y,z$ in a metric space $(\X,\sfd)$, we write
$[xyz]$ if
\[
\sfd(x,y)+\sfd(y,z)-\sfd(x,z)=\partial(\{x,y,z\}),
\]
which is in fact equivalent to
\[
\sfd(x,z)\ge\max( \sfd(x,y),\sfd(y,z)).
\]
Clearly $[xyz]\iff [zyx]$ and moreover at least one of the properties
$[xyz]$, $[xzy]$, or $[zxy]$ always holds.

\begin{definition}[Almost circular points]\label{def:circular points}
    Let $(\X,\sfd)$ be a metric space and fix a number $\eta \ge 0$. We say that four points $P_1,P_2,P_3,P_4 \in \X$ are \emph{$\eta$-circular} if
    \begin{equation}
        |\sfd(P_i,P_j)-\sfd(P_k,P_l)|\le \eta,
    \end{equation}
    for any choice of (distinct) indices $i,j,k,l\in\{1,2,3,4\}$.
\end{definition}
The name $\eta$-circular comes from the fact that any two couples
of antipodal points in the sphere (of any dimension) give rise to
four $0$-circular points. The motivation behind the above
definitions is the following result by Menger \cite{MR1512479} (see
also \cite{MR1476755,MR2888543}).
\begin{thm}[Menger]
    Suppose that a metric space $(\X,\sfd)$ satisfies $\partial \X=0$. Then either $\X$ contains only four points which are $0$-circular, or $\X$ can be isometrically embedded in $\mathbb{R}$.
\end{thm}
Theorem \ref{thm:l1 menger} is a sort of $L^1$-quantified
generalization of the above theorem in the case of 1-regular
metric measure spaces. In fact, an $L^\infty$-quantified version
also holds for arbitrary metric spaces. This is a bit similar in
spirit to \cite[Lemma 6.4]{MR4236801}, which concerns the
construction of good maps locally from a \emph{curve} intersected with a ball into
$\mathbb{R}$.

\begin{thm}[$L^\infty$-quantified Menger theorem]\label{thm:quantified menger}
    Let $(\X,\sfd)$ be a metric space such that  $\partial \X \le \beta$, with $\beta \ge 0,$
     and containing five points having pairwise distances strictly grater
     than $30\beta$. Then there exists a map $f: (\X,\sfd) \to (\rr,|\cdot|)$ that is a $40\beta$-isometry.
\end{thm}

Even if not needed in this note, we will prove this at the end of this section,
as it follows easily from the preliminary results needed in the proof of Theorem \ref{thm:l1 menger}.
 More precisely, both Theorem \ref{thm:l1 menger} and Theorem \ref{thm:quantified menger}
  are consequences of the following elementary technical lemmas.
  The first one says that given four points in a metric space, either they are almost circular or they can be embedded in $\rr$
  with an explicit almost isometry.
\begin{lemma}[4-points lemma]\label{prop:4 points lemma}  Let $(\X,\sfd)$ be a metric
space. Let  $P,Q,R,S\in \X$ and $\beta\ge 0 $ such that
$\partial\{P,Q,R,S\}\le \beta$ and $\sfd(P,Q)> 2\beta.$ Then at
least one of the the following holds:
\begin{enumerate}[label=\roman*)]
    \item the map $f: \{P,Q,R,S\}\to \rr$, defined by  $f(Q)\coloneqq\sfd(P,Q)$, and for  $x\in \{P,R,S\}$ by
    \[
    f(x)\coloneqq\begin{cases}
        -\sfd(x,P) & \text{ if } \sfd(x,Q)\ge \max(\sfd(P,Q),\sfd(x,P)),\\
        \sfd(x,P) & \text{ otherwise,}\\
    \end{cases}
    \]  is a $2\beta$-isometry,
    \item the points $P,Q,R,S$ are $2\beta$-circular.
\end{enumerate}
\end{lemma}

The second technical lemma essentially implies that if $\partial
\X$ is small and $\X$ contains four almost circular points, then
all the points in $\X$ must be close to those points. It is
instructive to keep in mind the example where $P_1,P_2,P_3,P_4$ are
given by two pairs of antipodal points on the circle $S^1$
equipped with the inner distance.

\begin{lemma}[Attraction to circular points]\label{prop:attraction}
Let $(\X,\sfd)$ be a metric space and let $\beta\ge 0$. Suppose
that the points $P_1,P_2,P_3,P_4\in  \X$  are $4\beta$-circular,
$\partial\{P_1,P_2,P_3,P_4\}\le\beta,$ and $\sfd(P_i,P_j)>
15\beta$ for all $i \neq j$.

Then for every $Q\in \X$ at least one of the following holds:
\begin{enumerate}[label=\roman*)]
    \item  $\sfd(Q,P_i)\le 15\beta$ for some $i\in\{1,2,3,4\},$
    \item $\partial\{P_1,P_2,P_3,P_4,Q\}>\beta.$
\end{enumerate}
\end{lemma}
The proofs of Lemma \ref{prop:4 points lemma} and Lemma
\ref{prop:attraction} are elementary but rather tedious and can be
found in Appendix \ref{sec:appendix}. Assuming their validity, we
now prove the main result of this subsection,
\textcolor{black}{Theorem \ref{thm:l1 menger}}. {\color{black} We
will split the proof in several lemmas, but before that we fix
some notations.  From now on $(\X,\sfd)$ will be a bounded
1-regular metric space with regularity constant $C_\X\ge 1$, i.e.\
$\X \in \reg_1(C_\X).$ We also set $\mu \coloneqq \mathcal{H}^1$,
the 1-dimensional Hausdorff measure in $(\X,\sfd)$ and } set
    \begin{equation}\label{eq:beta menger}
        \beta\coloneqq \fint_\X \fint_\X \fint_\X\frac{\partial\{x,y,z\}}{r}\, \d \mu(x) \d \mu(y) \d
        \mu(z),
    \end{equation}
    where $r\coloneqq \mathrm{diam}(\X)$.
    Note that the map $\X^3\ni (x,y,z)\mapsto \partial\{x,y,z\}$ is continuous as {\color{black}infimum} of a finite number of continuous functions.
    Without loss of generality  we can assume that $\beta \le \delta$, for some $\delta>0$ small to be chosen later and depending only on $C_{\X}.$ Indeed by taking $f:\X\to \rr$ as $f\equiv 0$, we can always make the left-hand side of \eqref{eq:l1 menger}  less than or equal to one.
    Fix also a constant $C>0$ big enough to be chosen later depending only on $C_\X$.

{\color{black}We start by giving an upper bound on the number
$\partial \X=\sup_{x,y,z\in \X}\partial\{x,y,z\}$.}
\textcolor{black}{\begin{lemma} It holds
    \begin{equation}\label{eq:god points}
        \partial \X \le \frac r{200} .
    \end{equation}
\end{lemma}}
\begin{proof}
 It suffices to consider the case $\partial \X>0$.   Set
    \[
    \beta_\infty \coloneqq \frac1r \partial \X.
    \]
     Let $x_1,x_2,x_3 \in \X$ be such that
    \[
    \frac{\partial\{x_1,x_2,x_3\}}r \ge \beta_\infty/2.
    \]
    Then, setting $\zeta\coloneqq r\beta_\infty/12$, we have
    \[
    \frac{\partial\{x,y,z\}}r\ge \beta_\infty/4, \quad \forall  (x,y,z)  \in B_{\zeta}(x_1)\times B_\zeta(x_2)\times B_\zeta(x_3).
    \]
    Hence by 1-regularity:
    \[
    \beta \ge \mu(\X)^{-3}\int_{B_{\zeta}(x_1)\times B_\zeta(x_2)\times B_\zeta(x_3)} r^{-1}\partial\{x,y,z\} \ge \tilde C \beta_\infty^4,
    \]
    using again $\mu(\X)\le C_\X r$ and where $\tilde C>0$ is a constant depending only on $C_\X.$ This shows that
    \begin{equation}\label{eq:beta infinity bound}
        \beta_\infty \le (\tilde C^{-1}\beta)^{\frac14}\le  (\tilde C^{-1}\delta)^{\frac14}.
    \end{equation}
    Therefore, choosing $\delta$ small enough we have $\beta_\infty \le \frac1{200}$ which is \eqref{eq:god points}.
\end{proof}

{\color{black}
\begin{lemma}\label{lem:PQ points}
   There exists two points $P,Q \in \X$ satisfying:
    \begin{enumerate}[label=\roman*)]
        \item\label{it:1} $\int_{\X\times \X} \frac{\partial \{P,x,y\}}{r}\d \mu(x)\d \mu(y)\le C\beta \mu(\X)^2, \quad\int_{\X\times \X}\frac{\partial \{Q,x,y\}}{r}\d \mu(x)\d \mu(y)\le C \beta \mu(\X)^2$,
        \item\label{it:2} $\int_{\X} \frac{\partial\{P,Q,x\}}r\d \mu(x)\le C \beta \mu(\X),$
        \item $\sfd(P,Q)\ge  r/2.$
    \end{enumerate}
\end{lemma}}
\begin{proof}
     Define the sets
    \[
    \begin{split}
         &A_1\coloneqq \left \{P \in \X \ : \ \fint_{\X\times \X} r^{-1}\partial\{P,x,y\}\d \mu(x)\d \mu(y)\le C \beta \right\} \subset \X,\\
         & A_2\coloneqq \left\{(P,Q)\in \X\times \X : \ \fint_\X r^{-1} \partial\{P,Q,x\} \d \mu(x)\le C \beta \right\}\subset \X\times \X.
    \end{split}
    \]
    By the dominated convergence theorem both $A_1$ and $A_2$ are closed sets.
    By \eqref{eq:beta menger} and the Markov inequality
    \begin{equation}\label{eq:A_1_A_2_ineq}
    \mu(\X\setminus A_1)\le \frac{\mu(\X)}{C},\quad
    (\mu\otimes \mu)((\X\times \X)\setminus A_2)\le \frac{\mu(\X)^2}{C}.
    \end{equation}
    The first inequality above gives
    \[
    (\mu\otimes \mu) (\X\times \X \setminus (A_1\times A_1))\le 2 \mu((\X\setminus A_1) \times \X)\le 2\frac{\mu(\X)^2}{C}.
    \]
    Additionally by 1-regularity we have
    $$(\mu\otimes \mu)(\{(x,y)\in \X\times\X \ : \sfd(x,y)\ge r/2\})\geq C_\X^{-2}\mu(\X)^2/2 ,$$
    where we used that $\mu(\X)\le C_\X r$. Together with \eqref{eq:A_1_A_2_ineq},
    this shows that if we choose $C$ big enough it holds
    \[
    A_2\cap (A_1\times A_1)\cap \{(x,y) \ : \sfd(x,y)\ge r/2\}\neq \emptyset
    \]
    and any couple $(P,Q)$ in this set
    \textcolor{black}{has the three desired properties.}
\end{proof}
From now on we fix two points $P,Q\in \X$ as given by Lemma \ref{lem:PQ points}. We  define the map $f: \X\to \rr$ by imposing $f(P)\coloneqq0$, $f(Q)\coloneqq\sfd(P,Q)$ and
    \[
    f(x)\coloneqq\begin{cases}
        -\sfd(x,P) & \text{ if } [xPQ],\\
        \sfd(x,P) & \text{ otherwise,}\\
    \end{cases}
    \]
    for every $x\notin\{P,Q\}$. Recall that $[xPQ]$ means
    $\sfd(x,Q)\geq \max\{\sfd(x,P),\sfd(P,Q)\}$.
\textcolor{black}{\begin{lemma}
    $f$ is Borel measurable and
    \begin{equation}\label{eq:f trick}
        ||f(x)-f(y)|-\sfd(x,y)|\le 2\min(\sfd(x,P),\sfd(y,P)), \quad  x,y\in \X.
    \end{equation}
\end{lemma}}
\begin{proof}
    The Borel measurability follows noting that the restriction of $f$ to either the closed set $\{x \ : \ [xPQ] \text{ holds}\}$ or its complement is continuous (in fact 1-Lipschitz).
    To show \eqref{eq:f trick}, note that by the triangle inequality $|\sfd(x,y)-\sfd(y,P)|\le \sfd(x,P)$ and that by definition $|f(y)|=\sfd(y,P)$ and $|f(x)|=\sfd(x,P)$, hence
    {\color{black}
    \begin{align*}
        \big||f(x)-f(y)|-\sfd(x,y)\big|&=\big||f(x)-f(y)|-\sfd(x,y)+ |f(y)|-|f(y)|\big|\\
        &\le  \big||f(x)-f(y)|-|f(y)|\big| + \big|-\sfd(x,y)+|f(y)|\big|\\
        &\le  |\sfd(x,y)-|f(y)||+|f(x)|\le 2\sfd(x,P).
    \end{align*}}
    Arguing in the same for $y$ we get \eqref{eq:f trick}.
\end{proof}

We are now ready to prove the main result of this section.

\begin{proof}[Proof of Theorem \ref{thm:l1 menger}]
{\color{black}Recall that our goal is to show that
 \begin{equation}\label{eq:goal}
        \fint_{\mathcal S} \frac{||f(x)-f(y)|-\sfd(x,y)|}{\diam(\X)}\d\mu(x)\d \mu(y) \le \tilde C \beta \mu(\X)^2,
    \end{equation}
    holds with $\mathcal S=\X\times \X$ and for some constant $\tilde C$ depending only on $C_\X.$ We proceed by proving \eqref{eq:goal} for different sets $\mathcal S$ that partition $\X\times \X$.}
    Define
    \begin{align*}
        &\mathcal{G}\coloneqq \{(x,y) \in \X\times \X\ :  \ ||f(x)-f(y)|-\sfd(x,y)|\le 5\partial\{x,y,P,Q\}\},\\
        &\mathcal{B}\coloneqq \X\times \X\setminus \mathcal{G}.
    \end{align*}
    {\color{black}Estimate \eqref{eq:goal} holds with $\mathcal S=\mathcal G$. To see this, by definition of $\mathcal G$ we have}
    \[
    \int_{\mathcal{G}} r^{-1}||f(x)-f(y)|-\sfd(x,y)|\d\mu(x)\d \mu(y)\le  \int_{\mathcal{G}} 5r^{-1}\partial\{x,y,P,Q\}\d\mu(x)\d \mu(y),
    \]
    {\color{black}from this and the obvious inequality
    $$\partial\{x,y,P,Q\}\le \partial\{P,Q,x\}+ \partial\{P,Q,y\}+\partial\{Q,x,y\}+\partial\{P,x,y\},$$
    we obtain}
    \begin{align*}
         \int_{\mathcal{G}} &r^{-1}||f(x)-f(y)|-\sfd(x,y)|\d\mu(x)\d \mu(y)\\
         &\le   5r^{-1}\int_{\mathcal{G}}  \partial\{P,Q,x\}+ \partial\{P,Q,y\}+\partial\{Q,x,y\}+\partial\{P,x,y\}\d\mu(x)\d \mu(y),
    \end{align*}
    {\color{black}from which plugging in the estimates $i)$  and $ii)$ in Lemma \ref{lem:PQ points}, which are satisfied by $P$ and $Q$, we have}
    \begin{equation*}\label{eq:booung G}
        \begin{split}
\int_{\mathcal{G}}& r^{-1}||f(x)-f(y)|-\sfd(x,y)|\d\mu(x)\d \mu(y)\le \\
            &\le  10 \mu(\X)\int_{\X}  r^{-1}\partial\{P,Q,x\}\d \mu(x)+5\int_{\X\times \X} r^{-1}\partial\{P,x,y\}\d \mu(x)\d \mu(y)\\
            &+5\int_{\X\times \X} r^{-1}\partial\{Q,x,y\}\d \mu(x)\d \mu(y)\le 20C \beta \mu(\X)^2.
        \end{split}
    \end{equation*}
    {\color{black}Our goal is  now show that \eqref{eq:goal} holds with $\mathcal S=\mathcal{B}$.} Thanks to \eqref{eq:god points} we have that $\sfd(P,Q)>2\partial\{x,y,P,Q\} $ for every $x,y \in \X.$ Hence for every $x,y\in \mathcal B$ we can apply Lemma \ref{prop:4 points lemma} to the points $x,y,P,Q$ and  and get
    \begin{align*}
        \mathcal{B}\subset\{(x,y)\in\X\times \X\ : \  \text{the points $x,y,P,Q$ are $2\partial\{x,y,P,Q\}$-circular}\},
    \end{align*}
    indeed the first case in Lemma \ref{prop:4 points lemma} cannot happen by definition of $\mathcal B.$ We further divide $\mathcal{B}$ as
    \begin{align*}
        &\mathcal{B}_1\coloneqq \{ (x,y)  \in \mathcal{B}\ : \ \sfd(\{x,y\},\{P,Q\})\le 30\partial\{x,y,P,Q\} \},\\
        &\mathcal{B}_2\coloneqq \mathcal{B}\setminus \mathcal{B}_1.
    \end{align*}
 {\color{black}Estimate \eqref{eq:goal} holds with $\mathcal S=\mathcal{B}_1$.}  To see this let $x,y \in \mathcal B_1.$ If  $\sfd(x,P) < 30\partial\{x,y,P,Q\}$ or  $\sfd(y,P) < 30\partial\{x,y,P,Q\}$, by \eqref{eq:f trick} we have
    \begin{equation}\label{eq:68}
        ||f(x)-f(y)|-\sfd(x,y)|\le 64 \partial\{x,y,P,Q\}.
    \end{equation}
    If instead $\sfd(x,Q) < 30\partial\{x,y,P,Q\}$ (or $\sfd(y,Q) < 30\partial\{x,y,P,Q\}$), by the  $2\partial\{x,y,P,Q\}$-circularity, we have $\sfd(y,P) < 32\partial\{x,y,P,Q\}$ (or $\sfd(x,P) < 32\partial\{x,y,P,Q\}$) and so by \eqref{eq:f trick} we get again \eqref{eq:68}.
    Hence using \eqref{eq:68} and then estimate $ii)$ of Lemma \ref{lem:PQ points} we have
    \begin{align*}
          \int_{\mathcal{B}_1} r^{-1}||f(x)-f(y)|-\sfd(x,y)|\d\mu(x)\d \mu(y)&\le  \int_{\mathcal{B}_1} \frac{ 64 \partial\{x,y,P,Q\}}{r}\d\mu(x)\d \mu(y)\\
          &\le 64\cdot 4 \beta C\mu(\X)^2.
    \end{align*}
   {\color{black} It remains to prove that \eqref{eq:goal} holds with $\mathcal S=\mathcal B_2.$  This is the most difficult set to deal with, because} the   couples $(x,y)\in \mathcal{B}_2$ are circular and spread apart. To estimate their contribution we will need to consider also  the other points in $\X.$  For $x,y \in \mathcal{B}_2$ define the number $D(x,y)$ as the minimum distance of two points in $\{x,y,P,Q\}$. By the definition of $\mathcal{B}_2$, by \eqref{eq:god points} and by $2\partial\{x,y,P,Q\}$-circularity it holds:
    \begin{equation}\label{eq:the right condition}
        D(x,y) > 30\partial\{x,y,P,Q\}.
    \end{equation}
    We claim that
    \begin{equation}\label{eq:not far}
        D(x,y) \le  15 \beta_\infty r, \quad  x,y\in \mathcal{B}_2.
    \end{equation}
    Indeed suppose this is not the case, i.e., $D(x,y)>15 \beta_\infty r=15\partial \X.$ Then by Lemma \ref{prop:attraction} applied to the whole $\X$, with the points $x,y,P,Q$ and with $\beta=\partial \X$ (note that $x,y,P,Q$ are $4\partial \X$-circular because $\partial\{x,y,P,Q\}\le \partial \X$ and $x,y\in \mathcal{B}_2\subset \mathcal{B}$),
    it must hold that
    $$\sfd(z,\{P,Q,x,y\})\le 15\partial \X=15r\beta_\infty\overset{\eqref{eq:beta infinity bound} }{\le} (\tilde C^{-1}\delta)^\frac14 15 \diam(\X),
     \quad z \in \X.$$
    This however contradicts the 1-regularity of $\X$, provided $\delta$ is small enough, hence \eqref{eq:not far} holds. Using \eqref{eq:not far} we can also conclude that
    \begin{equation}\label{eq:true Dxy}
        \sfd(P,Q)>D(x,y),\quad  \sfd(x,y)>D(x,y).
    \end{equation}
    Indeed combining \eqref{eq:god points} with \eqref{eq:not far} we get $\sfd(P,Q)>3D(x,y).$ Then using  $\sfd(P,Q)>3D(x,y)$ and the $2\partial\{x,y,P,Q\}$-circularity of $x,y,P,Q$ we get
    $$\sfd(x,y)\ge \sfd(P,Q)-2\partial\{x,y,P,Q\}\overset{\eqref{eq:the right condition}}{\ge } 3D(x,y)-D(x,y)\ge 2D(x,y)>0,$$
    which shows \eqref{eq:true Dxy}.
    For fixed $x,y \in \mathcal{B}_2$ we now consider all the points in $\X$ and  divide them into two sets, the ``attracted'' and the ``non attracted'' points:
    \begin{align*}
        &\mathcal{A}(x,y)\coloneqq \{z \in \X \ : \ \sfd(z,\{P,Q,x,y\})\le 15D(x,y)  \},\\
        &\overline {\mathcal{A}}(x,y)\coloneqq \X\setminus \mathcal{A}(x,y).
    \end{align*}
    The set $\mathcal{A}(x,y)$ is small in measure. Indeed by 1-regularity:
    \[
    \mu(\mathcal{A}(x,y))\le 4 C_\X 15 D(x,y)\overset{\eqref{eq:not far}}{\le}  900 C_\X r\beta_\infty\overset{\eqref{eq:beta infinity bound}}{\le } 900 C_\X^2 \mu(\X) (\tilde C^{-1}\delta)^\frac14.
    \]
    Hence assuming $\delta$ small enough we have $\mu(\overline {\mathcal{A}}(x,y))\ge \mu(\X)/2$.
    We now apply Lemma \ref{prop:attraction} with the points $x,y,P,Q$ with $\beta\coloneqq D(x,y)/20$ (note that these points are $4\beta$-circular \textcolor{black}{since they are $2\partial\{x,y,P,Q\}$-circular and $4\beta>2\partial\{x,y,P,Q\}$
     by \eqref{eq:the right condition}). By Lemma \ref{prop:attraction} we}  obtain that
    \begin{align}\label{eq:barA}
        \overline {\mathcal{A}}(x,y)\subset \{z \in \X \ : \  \partial\{z,x,y,P,Q\}> D(x,y)/20 \},
    \end{align}
    since the first option in Lemma \ref{prop:attraction} can not happen by definition of $ \overline {\mathcal{A}}(x,y).$
    Because $\partial\{x,y,P,Q\}<D(x,y)/30$ (recall \eqref{eq:the right condition}), the \textcolor{black}{inclusion \eqref{eq:barA}} implies that
    \begin{equation}\label{eq:lowerb}
        \partial\{z,x,y\}+\partial\{z,x,P\}+\partial\{z,x,Q\}+\partial\{z,y,P\}+\partial\{z,y,Q\}>
        \frac{D(x,y)}{20},  \quad z \in   \bar {\mathcal{A}}(x,y).
    \end{equation}
    Note now that by \eqref{eq:true Dxy} and the definition of $D(x,y)$ we have
    $$D(x,y)=\min\big (\sfd(x,P),\sfd(x,Q),\sfd(y,P),\sfd(y,Q)\big ).$$ Moreover, since $x,y,P,Q$ are $D(x,y)$-circular (\textcolor{black}{since $D(x,y)>2\partial\{x,y,P,Q\}$ }by \eqref{eq:the right condition}), we get that $\sfd(x,P)\le \sfd(y,Q)+D(x,y)$ and $\sfd(y,P)\le \sfd(x,Q)+D(x,y)$. Combining the last two observations we deduce that $\min(\sfd(x,P),\sfd(y,P))\le 2D(x,y)$.
    Hence by \eqref{eq:f trick}
    \begin{equation}\label{eq:f le D}
        ||f(x)-f(y)|-\sfd(x,y)| \le 4 D(x,y), \quad x,y \in \mathcal{B}_2.
    \end{equation}
Recalling $\mu(\overline{\mathcal{A}}(x,y))\geq \mu(X)/2$,   we can finally estimate
    \begin{align*}
        &\int_{(x,y)\in \mathcal{B}_2} r^{-1}||f(x)-f(y)|-\sfd(x,y)|\d\mu(x)\d \mu(y)=  \\
        &\int_{(x,y)\in \mathcal{B}_2} \mu(\bar {\mathcal{A}}(x,y))^{-1} \int_{z \in \bar {\mathcal{A}}(x,y)} r^{-1}||f(x)-f(y)|-\sfd(x,y)| \d \mu(z)\d\mu(x)\d \mu(y)\\
        &\overset{\eqref{eq:f le D}}{\le}  8\mu(\X)^{-1} \int_{(x,y)\in \mathcal{B}_2}\int_{z \in \bar {\mathcal{A}}(x,y)} r^{-1}D(x,y)\d \mu(z)\d\mu(x)\d \mu(y)\\
        &\overset{\eqref{eq:lowerb}}{\le} \frac{8\cdot 20}{\mu(\X) r} \int_{\X^3} \partial\{z,x,y\}+\partial\{z,x,P\}+\partial\{z,x,Q\}+\partial\{z,y,P\}+\partial\{z,y,Q\}\d \mu(z)\d\mu(x)\d \mu(y)\\
        &\le 8\cdot 20 (4C+1) \beta \mu(\X)^2,
    \end{align*}
    where in the last inequality we used the definition of $\beta$ in \eqref{eq:beta menger}, and the property \ref{it:1} of the points $P,Q$. This shows that \eqref{eq:goal} holds with $\mathcal S=\mathcal B_2$. Since we showed previously that \eqref{eq:goal} holds with $\mathcal S \in \{\mathcal G,\mathcal B_1\}$ and $\X\times \X=\mathcal G\cup \mathcal B_1\cup \mathcal B_2 $,  the proof is concluded.
\end{proof}

We conclude with the proof of Theorem \ref{thm:quantified menger} which, even if not used in the sequel, we believe to be interesting on its own.
\begin{proof}[Proof of Theorem \ref{thm:quantified menger}]
    If $\diam(\X)\le 40\beta$ the statement is trivial. Hence we can assume the existence of two points $P,Q\in\X$ such that $\sfd(P,Q)>40\beta.$ We define the map $f: \X\to \rr$ as follows. Set $f(P)\coloneqq0$, $f(Q)\coloneqq\sfd(P,Q)$ and for any other point $x\in \X$
    \[
    f(x)\coloneqq\begin{cases}
        -\sfd(x,P) & \text{ if } \sfd(x,Q)\ge \max(\sfd(P,Q),\sfd(x,P)),\\
        \sfd(x,P) & \text{ otherwise.}\\
    \end{cases}
    \]
    We need to prove that for every $x,y \in \X$ it holds
    \begin{equation}\label{eq:40beta_approx_isom}
    ||f(y)-f(x)|-\sfd(x,y)|\le 40 \beta.
    \end{equation}
    Fix $x,y \in \X$. If $ ||f(y)-f(x)|-\sfd(x,y)|\le 5 \beta$ there is nothing to prove, hence we can assume that  $||f(x)-f(y)|-\sfd(x,y)|>5\beta$. Hence from Lemma \ref{prop:4 points lemma}  we deduce that the points $x,y,P,Q$ are $2\beta$-circular.
    Observe that this implies that $\sfd(x,y)>15\beta$. Suppose now that $\sfd(x,P)<20\beta$, then by the triangle inequality $|\sfd(x,y)-\sfd(y,P)|<20\beta$, hence
    $$||f(y)-f(x)|-\sfd(x,y)|=||f(y)-f(x)|-\sfd(x,y)\pm |f(y)||\le |\sfd(x,y)-|f(y)||+|f(x)|< 40\beta,$$
    because $|f(y)|=\sfd(y,P)$ and $|f(x)|=\sfd(x,P).$  The same holds if $\sfd(y,P)<20\beta$. Hence we are left
    to prove \eqref{eq:40beta_approx_isom}
   in the case  $\sfd(x,P),\sfd(y,P)\ge 20\beta$, which thanks to $2\beta$-circularity of $\{x,y,P,Q\}$ gives also that
    $\sfd(y,Q),\sfd(x,Q)>15\beta.$ Recall also that $\sfd(P,Q)\ge 40\beta$ and $\sfd(x,y)\ge 15 \beta.$ Then we can apply Lemma \ref{prop:attraction} and  deduce that for every $z\in \X$ it holds that $\sfd(z,R)\le 15\beta$ for some $R\in \{P,Q,x,y\}$ (note that the second alternative in  Lemma \ref{prop:attraction} does not occur because $\partial \X\le \beta$). This contradicts the fact that $(\X,\sfd)$ contains five points at pairwise  distance strictly greater  than $30\beta$ and concludes the proof.
\end{proof}

\appendix

\section{Almost circular points}\label{sec:appendix}

This appendix contains the proofs of Lemmas \ref{prop:4 points
lemma} and  \ref{prop:attraction} concerning almost circular
points. We start by introducing short-hand notation that we will
often use in the proofs of this section. Let $a,b,c$ be real
numbers. We write $a\sim_{\eps} b$ to denote $|a-b|\le \eps$. This
convention is used exclusively within this section, so that there
should be no confusion with the notation introduced at the beginning of
Section \ref{s:prelim}. Note that if $a\sim_{\eps} b$ and
$c\sim_{\eps'} d$, then $a-c\sim_{\eps+\eps'} b-d$, and that
$a\sim_{\eps} b$ if and only if $a-c\sim_{\eps} b-c.$

We also recall from Section \ref{ss:ConstrGoodMapsIntoR} that we
write $[xyz]$ for points $x,y,z$ in a  metric space $(\X,\sfd)$ if
\[
\sfd(x,y)+\sfd(y,z)-\sfd(x,z)=\partial(\{x,y,z\}),
\]
which is equivalent to
\[
\sfd(x,z)\ge\max( \sfd(x,y),\sfd(y,z)).
\]
Moreover, $\partial S = \sup_{\{x,y,z\}\subset S}\partial(\{x,y,z\})$.

We start with a simple criterion to check that four points are
almost circular (recall Definition \ref{def:circular points}).
\begin{lemma}\label{lem:circular critierion}
    Let $(\X,\sfd)$ be a metric space and $x_1,x_2,x_3,x_4 \in \X$ such that $\partial\{x_1,x_2,x_3,x_4\}\le \delta$ and
    \begin{equation}\label{eq:cyclic}
        [x_1x_2x_3],\,  [x_2x_3x_4],\,  [x_3x_4x_1],\,  [x_4x_1x_2]
    \end{equation}
    hold. Then the points $x_1,x_2,x_3,x_4$ are $2\delta$-circular.
\end{lemma}
\begin{proof} According to our definitions, we have to check that
$\sfd(x_i,x_j)\sim_{2\delta}\sfd(x_k,x_l)$ for any choice of
distinct $i,j,k,l\in \{1,2,3,4\}$.
    The assumptions imply that
    \begin{align*}
        &i)\,  \sfd(x_1,x_3)\sim_\delta\sfd(x_1,x_2)+\sfd(x_2,x_3), \quad ii)\, \sfd(x_2,x_4)\sim_\delta \sfd(x_2,x_3)+\sfd(x_3,x_4)\\
        &iii)\,\sfd(x_1,x_3)\sim_\delta\sfd(x_1,x_4)+\sfd(x_3,x_4), \quad iv)\,\sfd(x_2,x_4)\sim_\delta \sfd(x_1,x_2)+\sfd(x_1,x_4).
    \end{align*}
    Subtracting $i)$ and $iii)$ and subtracting $ii)$ and $iv)$ we get
    \begin{equation}\label{eq:x1x2x3x4}
        \begin{split}
            &\sfd(x_1,x_2)+\sfd(x_2,x_3)\sim_{2\delta}\sfd(x_1,x_4)+\sfd(x_3,x_4)\\
            &\sfd(x_2,x_3)+\sfd(x_3,x_4)\sim_{2\delta}\sfd(x_1,x_2)+\sfd(x_1,x_4).
        \end{split}
    \end{equation}
    Subtracting the two in \eqref{eq:x1x2x3x4} we obtain $\sfd(x_1,x_2)-\sfd(x_3,x_4)\sim_{4\delta}\sfd(x_3,x_4)-\sfd(x_1,x_2),$
     which gives $\sfd(x_1,x_2)\sim_{2\delta}\sfd(x_4,x_3).$ Switching the order of the second in \eqref{eq:x1x2x3x4}
     and subtracting again the two shows that
    $\sfd(x_2,x_3)-\sfd(x_1,x_4)\sim_{4\delta}\sfd(x_1,x_4)-  \sfd(x_2,x_3),$ from which $\sfd(x_2,x_3)\sim_{2\delta}\sfd(x_1,x_4).$ Finally summing $i)$ and $iii)$ and summing $ii)$ and $iv)$ gives
    \[
    \begin{split}
        &2\sfd(x_1,x_3)\sim_{2\delta}\sfd(x_1,x_2)+\sfd(x_2,x_3)+\sfd(x_1,x_4)+\sfd(x_3,x_4)\\
        &2\sfd(x_2,x_4)\sim_{2\delta}\sfd(x_1,x_2)+\sfd(x_2,x_3)+\sfd(x_1,x_4)+\sfd(x_3,x_4).
    \end{split}
    \]
    Hence $2\sfd(x_1,x_3)\sim_{4\delta}2\sfd(x_2,x_4)$ and so $\sfd(x_1,x_3)\sim_{2\delta}\sfd(x_2,x_4),$ which concludes the proof.
\end{proof}

Next we prove the \emph{$4$-points lemma}, which gives a
quantitative condition for four points $\{P,Q,R,S\}$ to either
admit an (explicitly given) $2\beta$-isometry $f$ into
$\mathbb{R}$, or to be $2\beta$-circular. (A similar conclusion
was obtained under different assumptions in \cite[Lemma
2.2]{MR2163108}.)

\begin{proof}[Proof of Lemma \ref{prop:4 points lemma}]
It is sufficient to prove the lemma for $\beta>0.$
As in the statement, we define $f:\{P,Q,R,S\}\to
\mathbb{R}$ by $f(Q)\coloneqq\sfd(P,Q)$, and for  $x\in \{P,R,S\}$
by
    \[
    f(x)\coloneqq\begin{cases}
        -\sfd(x,P) & \text{ if } \sfd(x,Q)\ge \max\{\sfd(P,Q),\sfd(x,P)\},\\
        \sfd(x,P) & \text{ otherwise.}\\
    \end{cases}
    \]
It is straightforward to see that $f$ satisfies the rough isometry
condition at least with respect to the points $P$ and $Q$. Indeed,
 by definition, $||f(P)-f(x)|-\sfd(P,x)|=0$ for every $x\in \{Q,R,S\}$. Next we show that  $\sfd(x,Q)\sim_\beta|f(x)-f(Q)|$ for every $x \in\{P,R,S\} $. If $[xPQ]$, and so $f(x)=-\sfd(x,P)$, this is immediate because $\partial \{x,P,Q\}\le \beta$, hence we assume $f(x)=\sfd(x,P)$. In this case, if we have $[xQP]$, then
    \[
    ||f(x)-f(Q)|-\sfd(x,Q)|= |\sfd(x,P)-\sfd(P,Q)-\sfd(x,Q)|\le \beta.
    \]
    If instead  $[PxQ]$, then
    \[
    ||f(x)-f(Q)|-\sfd(x,Q)|= |\sfd(P,Q)-\sfd(P,x)-\sfd(x,Q)|\le \beta.
    \]
    Thanks to these observations, to show that $f:\{P,Q,R,S\}\to \rr$ is a $2\beta$-isometry, it is enough to show that
    \begin{equation}\label{eq:xy beta isom}
        ||f(R)-f(S)|-\sfd(R,S)|\le 2\beta.
    \end{equation}
    Hence to prove the lemma it is sufficient to show that either \eqref{eq:xy beta isom} holds or  that the points $\{P,Q,R,S\}$
    are $2\beta$-circular. Throughout the proof, we will
    repeatedly use the following fact, often without mentioning it
    explicitly: If $[x_1 x_2 x_3]$ holds, then
\begin{displaymath}
\partial
    (\{x_1,x_2,x_3\})\leq \beta \quad \Leftrightarrow \quad
    \sfd(x_1,x_2)+\sfd(x_2,x_3)\leq \sfd(x_1,x_3)+\beta.
\end{displaymath}

    It will be more convenient to name $x\coloneqq R$ and $y\coloneqq S$, to better distinguish these points from $P$ and $Q.$
    Up to swapping $x$ and $y$ we can assume that $f(x)\le f(y)$, hence we need  to consider only the following three cases:
\begin{enumerate}
\item $f(x)=-\sfd(x,P)$ and $f(y)=-\sfd(y,P)$, \item
$f(x)=-\sfd(x,P)$ and $f(y)=+\sfd(y,P)$, \item $f(x)=+\sfd(x,P)$
and $f(y)=+\sfd(y,P)$.
\end{enumerate}

\medskip

    \noindent \textbf{Case 1}:  $f(x)=-\sfd(x,P)$, $f(y)=-\sfd(y,P)$,
    which is equivalent to the validity of both $[xPQ]$ and
    $[yPQ]$. Using the assumption $f(x)\le f(y)$ and the
    triangle inequality, we see that in this case, \eqref{eq:xy beta isom} is  equivalent
    to
    \begin{equation}\label{eq:xy beta isomCase1}
    \sfd(x,y)+\sfd(y,P)\leq \sfd(x,P)+2\beta.
    \end{equation}
    We need to consider also the point $Q$. At least one of the conditions $[xQy]$, $[xyQ]$ and $[yxQ]$ holds. Suppose first that $[xQy]$ holds.
Using the assumptions $\partial \{P,Q,R,S\}\leq \beta$ and
$\sfd(P,Q)\geq 2\beta$, we find that
\begin{align*}
\sfd(x,y)+4\beta&\leq \left[\sfd(x,P)+\sfd(P,Q)\right]+\left[\sfd(y,P)+\sfd(P,Q)\right]\\
&\overset{[xPQ],[yPQ]}{\leq}\sfd(x,Q)+\sfd(y,Q)+2\beta.
\end{align*}
However, this leads to a contradiction since then
\begin{align*}
\sfd(x,y)\leq \sfd(x,Q)+\sfd(y,Q)-2\beta \overset{[xQy]}{\leq}
\sfd(x,y)-\beta,
\end{align*}
which is impossible (since $\beta>0$). Thus $[xQy]$ in Case 1 cannot occur.
If instead $[xyQ]$, we have
    \[
    \sfd(x,P)+\sfd(P,Q)\ge \sfd(x,Q)\overset{[xyQ]}{\ge} \sfd(x,y)+\sfd(y,Q)-\beta \overset{[yPQ]}{\ge} \sfd(x,y)+\sfd(y,P)+\sfd(P,Q)-2\beta,
    \]
    which shows \eqref{eq:xy beta isomCase1}. Finally if $[yxQ]$ holds, we have
    \[
    \sfd(y,P)+\sfd(P,Q)\ge \sfd(x,Q)\overset{[yxQ]}{\ge} \sfd(x,y)+\sfd(x,Q)-\beta \overset{[xPQ]}{\ge} \sfd(x,y)+\sfd(x,P)+\sfd(P,Q)-2\beta,
    \]
    which coupled with the assumption $\sfd(x,P)\ge \sfd(y,P)$ shows again \eqref{eq:xy beta isomCase1}.

\medskip

    \noindent \textbf{Case 2}:  $f(x)=-\sfd(x,P)$, $f(y)=\sfd(y,P)$, which means that $[xPQ]$ holds, and $[yPQ]$ does not hold. In this case \eqref{eq:xy beta
    isom} is equivalent to
    \begin{equation}\label{eq:xy beta
isomCase2}
    \sfd(y,P)+\sfd(x,P)\leq \sfd(x,y)+2\beta.
    \end{equation}
    Since $[yPQ]$ does not hold at least one of the two options $[PyQ]$, $[PQy]$ must
 be valid, so we only need to prove that in these two sub-cases either \eqref{eq:xy beta
    isomCase2} is true or that $x,y,P,Q$ are $2\beta$-circular.

    \smallskip

    \noindent \textbf{\emph{Case 2.a}}: $[PyQ]$ holds. We have
\begin{align*}\sfd(y,P)+\sfd(x,P)&\overset{[PyQ]}{\leq}\sfd(P,Q)-\sfd(y,Q)+\sfd(x,P)+\beta\\
&\overset{[xPQ]}{\leq} \sfd(x,Q)-\sfd(y,Q)+2\beta\\&\leq
\sfd(x,y)+2\beta,\end{align*} which yields \eqref{eq:xy beta
    isomCase2} in this case.

\smallskip

    \noindent \textbf{\emph{Case 2.b}}: $ [PQy]$ holds. If
    $[xPy]$, then \eqref{eq:xy beta
    isomCase2} trivially holds true.
     If instead $[xyP]$, then
\begin{align*}
\sfd(P,Q)&\overset{[xPQ]}{\leq} \sfd(x,Q)-\sfd(x,P)+\beta
\overset{[xyP]}{\leq} \sfd(x,Q)-\sfd(x,y)-\sfd(y,P)+2\beta\\
&\overset{[PQy]}{\leq}
\sfd(x,Q)-\sfd(x,y)-\sfd(P,Q)-\sfd(Q,y)+3\beta\\
&\leq -\sfd(P,Q)+3\beta.
\end{align*}
Therefore, $\sfd(P,Q)\leq 3\beta/2$, which is impossible since
$\sfd(P,Q)>2\beta$ by assumption.

 Hence it remains to
consider the case when $[Pxy]$ holds.
    We need to consider now also the point $Q$ and the cases $[yQx]$,$[yxQ]$, and $[xyQ]$. If $[yQx]$, then
    \[
    \begin{split}
        \sfd(x,y)&\overset{[yQx]}{\ge} \sfd(y,Q)+\sfd(Q,x)-\beta \overset{[xPQ]}{\ge} \sfd(y,Q)+\sfd(x,P)+\sfd(P,Q)-2\beta\\
        &\ge  \sfd(y,P)+\sfd(x,P)-2\beta,
    \end{split}
    \]
    which  implies \eqref{eq:xy beta
    isomCase2} in this case.
If instead $[yxQ]$, then \begin{align*}
\sfd(P,Q)&\overset{[PQy]}{\leq} \sfd(P,y)-\sfd(Q,y)+\beta
\overset{[yxQ]}{\leq} \sfd(P,y)-\sfd(y,x)-\sfd(x,Q)+2\beta\\
&\overset{[xPQ]}{\leq}
\sfd(P,y)-\sfd(x,y)-\sfd(x,P)-\sfd(P,Q)+3\beta\\
&\leq -\sfd(P,Q)+3\beta.
\end{align*}
Therefore, analogously as in a previous case, $\sfd(P,Q)\leq
3\beta/2$, which is impossible since $\sfd(P,Q)>2\beta$ by
assumption.

 Hence we are left with the case when $[xyQ]$ holds. Summarizing the current
    assumptions, we are in the situation where
 $[PQy],[xPQ],[Qyx],[yxP]$ hold. Applying Lemma \ref{lem:circular critierion} we obtain that $x,y,P,Q$ are $2\beta$-circular.

\medskip

    \noindent \textbf{Case 3}:  $f(x)=\sfd(x,P)$, $f(y)=\sfd(y,P)$, that is, neither $[xPQ]$ nor $[yPQ]$
    holds. Since $\sfd(x,P)=f(x) \leq f(y)=\sfd(y,P)$ by assumption, the desired condition \eqref{eq:xy beta isom}
  simplifies in this case to
  \begin{equation}\label{eq:Case3a}
  \sfd(x,y)+\sfd(x,P)\leq \sfd(y,P)+ 2\beta.
  \end{equation}
  If $[yxP]$, then \eqref{eq:Case3a} holds true even with
``$2\beta$'' replaced by ``$\beta$'' on the right-hand side. In
the following, we assume therefore that  $[yxP]$ does \emph{not}
hold. Since $\sfd(y,P)\geq \sfd(x,P)$, the only way $[yxP]$ can
fail is if $\sfd(x,y)>\sfd(y,P)$. In that case we have $[xPy]$,
which we now add as a standing assumption to all the following
sub-cases.

\smallskip
    \noindent \textbf{\emph{Case 3.a}}: $[xQP]$ and $[yQP]$ hold. This is similar to Case 1:
 Using  $\sfd(P,Q)\geq
2\beta$, we find that
\begin{align*}
\sfd(x,y)+4\beta&\leq \left[\sfd(x,Q)+\sfd(P,Q)\right]+\left[\sfd(y,Q)+\sfd(P,Q)\right]\\
&\overset{[xQP],[yQP]}{\leq}\sfd(x,P)+\sfd(y,P)+2\beta \overset{[xPy]}{\leq} \sfd(x,y)+3\beta,
\end{align*}
which is impossible (since $\beta>0$). Thus the Case 3.a cannot occur under the
standing assumption that $[xPy]$.

    \smallskip

    \noindent \textbf{\emph{Case 3.b}}:
   exactly one of   $[xQP]$ and $[yQP]$
    holds. Assume first that $[xQP]$ holds and $[yQP]$ does \emph{not} hold.
Since in Case 3 also $[yPQ]$ does not hold, we must necessarily
have that $[PyQ]$. Together with the assumption $\sfd(x,P)\leq
\sfd(y,P)$, this yields
\begin{displaymath}
\sfd(y,P)\overset{[PyQ]}{\leq} \sfd(P,Q)\overset{[xQP]}{\leq}
\sfd(x,P)\leq \sfd(y,P).
\end{displaymath}
Therefore $\sfd(y,P)=\sfd(P,Q)$, which by $[PyQ]$ would imply
that also $[yQP]$ holds, which is a contradiction. Thus, Case 3.b can only occur if $[yQP]$
holds and $[xQP]$ does \emph{not} hold.

    Since also $[xPQ]$ does not hold in Case 3,  we must necessarily have that $[PxQ]$.
    Then
    \[
    \begin{split}
        \sfd(y,P)&\overset{[yQP]}{\ge} \sfd(y,Q)+\sfd(Q,P)-\beta \overset{[PxQ]}{\ge} \sfd(y,Q)+\sfd(P,x)+\sfd(x,Q)-2\beta \\
        &\ge \sfd(y,x)+\sfd(P,x)-2\beta.
    \end{split}
    \]
This concludes the proof of \eqref{eq:Case3a}
in Case 3.b.
    \smallskip

    \noindent \textbf{\emph{Case 3.c}}: neither $[xQP]$ nor $[yQP]$
    holds. As the assumptions in Case 3 also rule out the validity of $[xPQ]$ and $[yPQ]$, we must necessarily have that $[PxQ]$ and $[PyQ]$ in
    Case 3.c. We also recall the standing assumption $[xPy]$ to which we reduced the discussion at the beginning of Case 3.

    We  need to consider also the points $x,y,Q$ together, and
    distinguish the cases $[xyQ]$, $[yxQ]$, and $[xQy]$.
     If $[xyQ]$, then
     \begin{align*}\sfd(x,y)+\sfd(x,P)&\overset{[xyQ]}{\leq}\sfd(x,Q)-\sfd(y,Q)+\sfd(x,P)+\beta\overset{[PxQ]}{\leq}\sfd(P,Q)
     -\sfd(y,Q)+2\beta\\&\leq \sfd(P,y)+\sfd(y,Q)-\sfd(y,Q)+2\beta.\end{align*}
     Thus \eqref{eq:Case3a} holds true in this case. Next we assume $[yxQ]$ instead of $[xyQ]$. We apply an analogous argument as before, but use additionally
     the assumption $\sfd(x,P)\leq \sfd(y,P)$. This yields
\begin{align*}\sfd(x,y)+\sfd(x,P)&\leq \sfd(x,y)+\sfd(y,P)\overset{[yxQ]}{\leq}\sfd(y,Q)-\sfd(x,Q)+\sfd(y,P)+\beta\\&\overset{[PyQ]}{\leq}\sfd(P,Q)
     -\sfd(x,Q)+2\beta\leq \sfd(P,x)+\sfd(x,Q)-\sfd(x,Q)+2\beta\\&\leq \sfd(P,y)+2\beta,\end{align*}
which confirms \eqref{eq:Case3a} also in  this case.
    It remains the case when $[xQy]$ holds.
    Summarizing, the current assumptions are $[xPy],[PyQ],[yQx],[QxP]$. We can
    apply Lemma \ref{lem:circular critierion} and obtain that $x,y,P,Q$ are $2\beta$-circular.
\end{proof}

We now prove Lemma \ref{prop:attraction} concerning the attraction
to circular points.

\begin{proof}[Proof of Lemma \ref{prop:attraction}]  It is sufficient to prove the statement for $\beta>0.$
 Let $P_1,P_2,P_3,P_4 \in \X$ be four points as in the statement, i.e., $\partial\{P_1,P_2,P_2,P_4\}\le\beta,$ $\sfd(P_i,P_j)>15\beta$
 for all $i \neq j$ and they are $4\beta$-circular:
    \begin{equation}\label{eq:non aligned coupled}
        \sfd(P_i,P_j)\sim_{4\beta}\sfd(P_k,P_l),
    \end{equation}
    for any choice of (distinct) indices $1\le i,j,k,l\le 4$.

    To conclude the statement of the lemma, it is sufficient to prove that if  $\partial\{P_1,P_2,P_2,P_4,Q\}\le \beta$, then
    \begin{equation}\label{eq:attracted}
        \sfd(Q,\{P_1,P_2,P_3,P_4\})\le 15\beta, \quad  Q \in \X.
    \end{equation}
    We argue by contradiction, that is, we assume that $\partial\{P_1,P_2,P_2,P_4,Q\}\le \beta$ and that there exists $Q \in \X$ such that $\sfd(Q,P_i)> 15\beta$ for every $i\in\{1,2,3,4\}$. We make the following claim.
    \begin{equation*}
        \parbox{12cm}{\textbf{{Claim}}:  for every  choice  of (pairwise distinct) indices $i,j,k \in \{1,2,3,4\}$, there exists a $2\beta$-isometry $f: \{Q,P_i,P_j,P_k\}\to \rr$.}
    \end{equation*}
    This claim will be applied in ``Case 2'' later in the proof.
    To prove the claim, assume towards a contradiction that its
    statement is not true for some choice of $i,j,k \in \{1,2,3,4\}$.
    Then, since $\sfd(Q,P_i)\ge 2\beta$ for every $i=1,2,3,4,$ from Lemma \ref{prop:4 points lemma}
    we must have that the points $Q,P_i,P_j,P_k$   are $2\beta$-circular. This implies that
    \[
    \sfd(Q,P_i)\sim_{2\beta}\sfd(P_j,P_k), \quad    \sfd(Q,P_j)\sim_{2\beta}\sfd(P_i,P_k), \quad    \sfd(Q,P_k)\sim_{2\beta}\sfd(P_i,P_j),
    \]
    which combined with \eqref{eq:non aligned coupled} gives
    \begin{equation}\label{eq:l is q}
        \sfd(Q,P_i)\sim_{6\beta}\sfd(P_l,P_i), \quad    \sfd(Q,P_j)\sim_{6\beta}\sfd(P_l,P_j), \quad    \sfd(Q,P_k)\sim_{6\beta}\sfd(P_l,P_k),
    \end{equation}
    where $\{l\}=\{1,2,3,4\}\setminus \{i,j,k\}$.
    Up to reordering the indices $i,j,k$ we can assume that $\sfd(P_l,P_i)\ge \max\{\sfd(P_l,P_j),\sfd(P_l,P_k)\}$.
    From this last inequality and $\sfd(P_i,P_j)\sim_{4\beta}\sfd(P_l,P_k)$ (recall \eqref{eq:non aligned coupled})
    we obtain $\sfd(P_i,P_j)\le \sfd(P_l,P_i)+4\beta$.
  In fact, this inequality can be improved
    to $\sfd(P_i,P_j)\le \sfd(P_l,P_i)$. Assume towards a
    contradiction that $\sfd(P_l,P_i)<\sfd(P_i,P_j)$. Then, under
    the current assumptions, $\partial \{P_l,P_i,P_j\}\leq \beta$ would imply that
\begin{displaymath}
\sfd(P_l,P_j)+\sfd(P_l,P_i)\leq \sfd(P_i,P_j)+\beta \leq
\sfd(P_l,P_i)+5\beta,
\end{displaymath}
which contradicts the initial assumption $\sfd(P_l,P_j)>15\beta$. Thus we know that
in fact $\max\{\sfd(P_i,P_j),\sfd(P_l,P_j)\}\leq \sfd (P_l,P_i)$,
from which $\partial \{P_l,P_i,P_j\}\leq \beta$ implies that
    \begin{equation}\label{eq:solita storia}
        \sfd(P_l,P_i)\sim_\beta \sfd(P_l,P_j)+\sfd(P_j,P_i).
    \end{equation}
    Similarly, since $  \sfd(Q,P_i)\sim_{6\beta}\sfd(P_l,P_i)$, $\sfd(Q,P_l)>15\beta$, $\sfd(P_l,P_i)>15\beta$
    and we are assuming $\partial\{P_i,P_l,Q\}\le \beta$, we can check that the only possibility is
    \begin{equation}\label{eq:solita storia 2}
        \sfd(Q,P_l)\sim_\beta \sfd(Q,P_i)+\sfd(P_i,P_l).
    \end{equation}
    Indeed, if the maximum of
    $\{\sfd(Q,P_i),\sfd(P_i,P_l),\sfd(Q,P_l)\}$ was achieved by
    $\sfd(Q,P_i)$ or $\sfd(P_l,P_i)$, then using $  \sfd(Q,P_i)\sim_{6\beta}\sfd(P_l,P_i)$ and $\partial\{P_i,P_l,Q\}\le \beta$
    would lead to a contradiction with $\sfd(Q,P_l)>15\beta$.  Hence $\sfd(Q,P_l)\geq
\max\{\sfd(Q,P_i),\sfd(P_i,P_l)\}$, which implies
\eqref{eq:solita storia 2}.
    Combining \eqref{eq:solita storia 2} and  \eqref{eq:l is q} we obtain
    \begin{align*}
    2\sfd(P_i,P_l)-7\beta \overset{ \eqref{eq:l is q}}{\le}\sfd(Q,P_i)+\sfd(P_i,P_l)-\beta&\overset{ \eqref{eq:solita storia 2}}{\le}
      \sfd(Q,P_l) \le \sfd(Q,P_j)+\sfd(P_j,P_l)\\&\overset{ \eqref{eq:l is q}}{\le} 2\sfd(P_l,P_j)+6\beta,
    \end{align*}
    which contradicts \eqref{eq:solita storia}, since $\beta>0$ and $\sfd(P_i,P_j)\ge 15\beta.$ This concludes the proof of the above
    claim on thus the existence of  $2\beta$-isometry $f: \{Q,P_i,P_j,P_k\}\to
    \rr$ for all (pairwise distinct) $i,j,k \in \{1,2,3,4\}$.

    \medskip

  We now return to the proof of \eqref{eq:attracted} by
  contradiction. Up to relabelling we can also assume that $\sfd(P_1,P_3)=\max_{1\le i,j \le 4} \sfd(P_i,P_j).$  Then, Since $\partial\{P_1,P_2,P_3,P_4\}\leq \beta$,
   we have $[P_1P_4P_3]$, $[P_1P_2P_3]$ and
    \begin{equation}\label{eq:subtle 2}
        \sfd(P_1,P_3)\sim_\beta\sfd(P_1,P_2)+ \sfd(P_2,P_3), \quad \sfd(P_1,P_3)\sim_\beta\sfd(P_1,P_4)+ \sfd(P_4,P_3).
    \end{equation}
 Moreover we must have
  \begin{equation}\label{eq:subtle}
        \sfd(P_2,P_4)\sim_\beta \sfd(P_2,P_3)+\sfd(P_3,P_4), \quad \sfd(P_2,P_4)\sim_\beta \sfd(P_2,P_1)+\sfd(P_1,P_4),
    \end{equation}
as can be  easily deduced from \eqref{eq:subtle
    2}, recalling also that  $\sfd(P_2,P_4)\sim_{4\beta} \sfd(P_1,P_3)$ (by $4\beta$-circularity), that $ \sfd(P_1,P_3)=\max_{1\le i,j \le 4} \sfd(P_i,P_j)$ and that $\sfd(P_i,P_j)>15\beta$. For example if we had instead $\sfd(P_2,P_3)\sim_\beta \sfd(P_2,P_4)+\sfd(P_3,P_4)$, it would imply that $\sfd(P_2,P_3)>  \sfd(P_1,P_3)$, which is false.

    Consider now the points $P_1,Q,P_3$. There are three possible
    cases:  $[P_1P_3Q]$, $[P_1QP_3]$, and  $[P_3P_1Q]$.
   However, since the assumptions on $P_1,P_3$ are symmetric, up to swapping $P_1$ and $P_3$ we can distinguish only two cases: $[P_1P_3Q]$ or $[P_1QP_3].$
\medskip

    \noindent\textbf{ Case 1:}  $[P_1P_3Q]$ holds. Then
    \[
    \begin{split}
        \sfd(P_1,P_4)+ &\sfd(P_4,P_3)+\sfd(P_3,Q)-2\beta  \overset{\eqref{eq:subtle 2}}{\le} \sfd(P_1,P_3)+\sfd(P_3,Q)-\beta \\
        & \overset{[P_1P_3Q]}{\le} \sfd(P_1,Q)\le  \sfd(P_1,P_4)+ \sfd(P_4,Q),
    \end{split}
    \]
    which shows that
    \begin{equation}\label{eq:suspicious 1}
        \sfd(P_4,Q)\sim_{2\beta}\sfd(P_4,P_3)+\sfd(P_3,Q).
    \end{equation}
    Analogously exchanging $P_4$ with $P_2$ (again using \eqref{eq:subtle 2}) we can show that
    \begin{equation}\label{eq:suspicious 2}
        \sfd(P_2,Q)\sim_{2\beta}\sfd(P_2,P_3)+\sfd(P_3,Q).
    \end{equation}
    However the above two relations will lead to a contradiction,
    recalling $\partial\{P_2,Q,P_4\}\le \beta$. Indeed, up to exchanging $P_2$ and $P_4$
we can assume that either $[P_2QP_4]$ or $[P_2P_4Q]$ holds. If
$[P_2QP_4]$ holds, then
\begin{align*}
    \sfd(P_2,P_4) &\overset{[P_2QP_4]}{\ge} \sfd(P_2,Q)+\sfd(Q,P_4) -\beta \overset{\eqref{eq:suspicious 1},\eqref{eq:suspicious 2}}{\ge} \sfd(P_2,P_3)+\sfd(P_4,P_3)+2\sfd(P_3,Q)-5\beta \\
    &\ge \sfd(P_2,P_4)+2\sfd(P_3,Q)-5\beta.
\end{align*}
    The above however contradicts  $\sfd(P_3,Q)>15\beta.$
    Suppose instead $[P_2P_4Q]$, then
    \[
    \sfd(P_2,P_4)\overset{[P_2P_4Q]}{\le}  \sfd(P_2,Q)-\sfd(P_4,Q)+\beta \overset{\eqref{eq:suspicious 1}}{\le}  \sfd(P_2,P_3)-\sfd(P_4,P_3)+3\beta<\sfd(P_2,P_4),
    \]
    where in the second step we used also the triangle inequality and in the last inequality we used that $\sfd(P_4,P_3)\ge 15\beta$ and that $\sfd(P_2,P_4)\ge  \sfd(P_2,P_3)$ (which comes from
 the first part of \eqref{eq:subtle} and $\sfd(P_3,P_4)>15\beta$). This is a contradiction which shows that Case 1 can not happen.

\medskip

    \noindent\textbf{ Case 2:} $[P_1QP_3]$ holds. Recall that by the claim below \eqref{eq:attracted},
    we have the existence of maps $f: \{Q,P_1,P_2,P_3\}\to \rr$, $g: \{Q,P_1,P_3,P_4\}\to \rr$ that are $2\beta$-isometries.
    From the fact that $\sfd(P_1,P_3)\ge \sfd(P_i,P_j)$ for all $i,j$ and that $\sfd(P_i,P_j)\ge 15\beta$  for all distinct $i,j$  (part of the initial assumptions), the point $f(P_2)$ must lie  in the interval with endpoints $f(P_1),f(P_3)$.   Similarly from the assumption $[P_1QP_3]$  (and again from $\sfd(Q,P_i)\ge 15\beta$ for every $i=1,2,3,4$)  it follows that also  $f(Q)$ must lie  in the same interval. Hence, up to replacing $f$ with $-f$ and  swapping the labels of $P_1$ and $P_3$  (observe that the current assumptions are symmetric in $P_1$ and $P_3$), we can assume that
    \begin{equation}\label{eq:f order}
          f(P_1)\le f(Q)\le f(P_2)\le f(P_3).
    \end{equation}
    Analogously the point $g(P_4)$ must lie  in the interval with endpoints $g(P_1),g(P_3)$ and up to replacing $g$ with $-g$ we can also assume that
    \[
    g(P_1)\le g(P_4)\le g(P_3).
    \]
    It remains two possibilities for the position of  $g(Q)$:
    \begin{equation}\label{eq:last cases}
            g(P_1)\le g(Q)\le g(P_4) \text{ or } g(P_4)\le g(Q)\le g(P_3).
    \end{equation}
    Suppose that $g(P_1)\le g(Q)\le {g}(P_4)$ holds. As $f,g$ are $2\beta$-isometries
    (and since $\sfd(P_i,P_j)\ge 15\beta$, $\sfd(Q,P_i)\ge 15\beta$ for every $i,j=1,2,3,4$, $i\neq j$) from this and \eqref{eq:f order}
     we deduce that $\sfd(P_1,P_2)\sim_{\beta}\sfd(P_1,Q)+\sfd(Q,P_2)$ and $\sfd(P_1,P_4)\sim_{\beta}\sfd(P_1,Q)+\sfd(Q,P_4)$. Therefore
\begin{align*}
    \sfd(P_2,P_4) &\le \sfd(P_2,Q)+\sfd(Q,P_4) \le \sfd(P_1,P_2)+\sfd(P_1,P_4)-2\sfd(P_1,Q)+2\beta \\
    &\overset{\eqref{eq:subtle}}{\le} \sfd(P_2,P_4)-2\sfd(P_1,Q)+3\beta,
\end{align*}
    which contradicts $\sfd(P_1,Q)\ge 15\beta>0$. If instead $g(P_4)\le g(Q)\le g(P_3)$
is satisfied, we have
$\sfd(P_1,P_2)\sim_{\beta}\sfd(P_1,Q)+\sfd(Q,P_2)$ and
$\sfd(P_4,P_3)\sim_\beta\sfd(P_4,Q)+\sfd(Q,P_3) $, therefore
    \[
    \sfd(P_1,P_3)+\sfd(P_2,P_4)\le \sfd(P_1,Q)+ \sfd(Q,P_3)+ \sfd(P_2,Q)+ \sfd(Q,P_4)\le  \sfd(P_1,P_2)+ \sfd(P_4,P_3)+2\beta,
    \]
    from which using the first of both \eqref{eq:subtle 2} and \eqref{eq:subtle} on the left-hand side we deduce
    \[
    \sfd(P_1,P_2)+2\sfd(P_2,P_3)+\sfd(P_3,P_4)-2\beta\le  \sfd(P_1,P_2)+ \sfd(P_4,P_3)+2\beta.
    \]
    This however gives $\sfd(P_2,P_3)\le 2\beta$, which is a contradiction because by assumption $\sfd(P_2,P_3)\ge
    15\beta>0$.
\end{proof}

\def\cprime{$'$}

\end{document}